\label{key}
\documentclass[10]{article}
\usepackage{amsfonts,amssymb,amsmath,amsthm,cite}
\usepackage{graphicx}

\usepackage[applemac]{inputenc}
\usepackage{amsmath,amssymb,amsthm, hyperref, euscript}
\usepackage[matrix,arrow,curve]{xy}
\usepackage{graphicx}
\usepackage{tabularx}
\usepackage{float}
\usepackage{hyperref}
\usepackage{tikz}
\usepackage{slashed}
\usepackage{mathrsfs}

\usetikzlibrary{matrix}

\usepackage[T1]{fontenc}
\usepackage{amsfonts,cite}
\usepackage{graphicx}

\usepackage[applemac]{inputenc}

\usepackage[sc]{mathpazo}
\DeclareFontFamily{T1}{pzc}{}
\DeclareFontShape{T1}{pzc}{m}{it}{1.8 <-> pzcmi8t}{}
\DeclareMathAlphabet{\mathpzc}{T1}{pzc}{m}{it}

\theoremstyle{plain}
\newtheorem{prop}{Proposition}[section]

\newtheorem{lem}[prop]{Lemma}
\newtheorem{cor}[prop]{Corollary}
\newtheorem{thm}[prop]{Theorem}

\theoremstyle{definition}
\newtheorem{defn}[prop]{Definition}
\newtheorem{empt}[prop]{}
\newtheorem{exm}[prop]{Example}
\newtheorem{rem}[prop]{Remark}

\newtheorem{cond}{Condition}        
\newtheorem{fact}[prop]{Fact}
\newtheorem*{notn}{Notation}        

\DeclareMathOperator{\Dom}{Dom}              
\newcommand{\Dslash}{{D\mkern-11.5mu/\,}}    

\newcommand{\vertiii}[1]{{\left\vert\kern-0.25ex\left\vert\kern-0.25ex\left\vert #1
    \right\vert\kern-0.25ex\right\vert\kern-0.25ex\right\vert}}
\newcommand{\Ga}{\Gamma}                     
\newcommand{\Coo}{C^\infty}                  

\usepackage[sc]{mathpazo}
\newbox\ncintdbox \newbox\ncinttbox 
\setbox0=\hbox{$-$} \setbox2=\hbox{$\displaystyle\int$}
\setbox\ncintdbox=\hbox{\rlap{\hbox
    to \wd2{\hskip-.125em \box2\relax\hfil}}\box0\kern.1em}
\setbox0=\hbox{$\vcenter{\hrule width 4pt}$}
\setbox2=\hbox{$\textstyle\int$} \setbox\ncinttbox=\hbox{\rlap{\hbox
    to \wd2{\hskip-.175em \box2\relax\hfil}}\box0\kern.1em}

\newcommand{\ncint}{\mathop{\mathchoice{\copy\ncintdbox}%
           {\copy\ncinttbox}{\copy\ncinttbox}%
           {\copy\ncinttbox}}\nolimits}  
           
\newcommand{\oxyox}{\otimes\cdots\otimes}  
\newcommand{\wyw}{\wedge\cdots\wedge}      

\newcommand{\id}{\mathrm{id}}                

\newcommand{\A}{\mathcal{A}}                 

\newcommand{\braket}[2]{\langle#1\mathbin|#2\rangle} 
\newcommand{\C}{\mathbb{C}}                  
\newcommand{\cc}{\mathbf{c}}                 
\DeclareMathOperator{\Cl}{C\ell}             
\renewcommand{\H}{H}               
\newcommand{\hookto}{\hookrightarrow}        
\newcommand{\ketbra}[2]{|#1\rangle\langle#2|} 
\newcommand{\La}{\Lambda}                    
\newcommand{\la}{\lambda}                    

\newcommand{\nb}{\nabla}                     
\newcommand{\om}{\omega}                     
\newcommand{\opp}{{\mathrm{op}}}             
\newcommand{\ox}{\otimes}                    
\newcommand{\eps}{\varepsilon}                    
\newcommand{\set}[1]{\{\,#1\,\}}             

\DeclareMathOperator{\spec}{sp}              
\DeclareMathOperator{\tr}{tr}                
\newcommand{\del}{\partial}                  
\DeclareMathOperator{\tsum}{{\textstyle\sum}} 
\newcommand{\w}{\wedge}                      
\newcommand{\x}{\times}                      
\renewcommand{\:}{\colon}                    

\newcommand{\sA}{\mathcal{A}}       
\newcommand{\sB}{\mathcal{B}}       
\newcommand{\sE}{\mathcal{E}}       
\newcommand{\sH}{\mathcal{H}}       
\newcommand{\sK}{\mathcal{K}}       
\newcommand{\sL}{\mathcal{L}}       
\newcommand{\sS}{\mathcal{S}}       
\newcommand{\sT}{\mathcal{T}}       

\newcommand{\Om}{\Omega}       

\newcommand{\inv}[1]{\frac{1}{#1}}                 

\newcommand{\bC}{\mathbb{C}}       
\newcommand{\bN}{\mathbb{N}}       
\newcommand{\bR}{\mathbb{R}}       
\newcommand{\bZ}{\mathbb{Z}}       

\newcommand{\al}{\alpha}          
\newcommand{\bt}{\beta}           
\newcommand{\dl}{\delta}          
\newcommand{\ga}{\gamma}          
\newcommand{\ka}{\kappa}          
\newcommand{\sg}{\sigma}          

\newcommand{\gX}{\mathfrak{X}}      

\DeclareMathOperator{\rH}{H}        
\DeclareMathOperator{\rK}{K}        
\DeclareMathOperator{\spn}{span}       

\newcommand{\cz}{{\bullet}}         
\newcommand{\isom}{\cong}          
\newcommand{\ul}{\underline}        
\renewcommand{\:}{\colon}           

\newcommand{\pairing}[2]{(#1\stroke #2)} 
\def\<#1|#2>{\langle#1\stroke#2\rangle} 
\def\?#1|#2?{\{#1\stroke#2\}}        

\renewcommand{\Tilde}[1]{\widetilde{#1}} 

\DeclareMathOperator{\bCl}{\bC l}   

\newcommand{\word}[1]{\quad\text{#1}\enspace} 
\newcommand{\words}[1]{\quad\text{#1}\quad} 

\def\<#1,#2>{\langle#1,#2\rangle}            
\def\ee_#1{e_{{\scriptscriptstyle#1}}}       
\def\wick:#1:{\mathopen:#1\mathclose:}       

\newbox\ncintdbox \newbox\ncinttbox 
\setbox0=\hbox{$-$}
\setbox2=\hbox{$\displaystyle\int$}
\setbox\ncintdbox=\hbox{\rlap{\hbox
           to \wd2{\box2\relax\hfil}}\box0\kern.1em}
\setbox0=\hbox{$\vcenter{\hrule width 4pt}$}
\setbox2=\hbox{$\textstyle\int$}
\setbox\ncinttbox=\hbox{\rlap{\hbox
           to \wd2{\hskip-.05em\box2\relax\hfil}}\box0\kern.1em}

\newcommand{\stroke}{\mathbin|}   

\newcommand{\bra}[1]{\langle{#1}\rvert} 

\newcommand{\End}{\mathrm{End}}       
\newcommand{\Hom}{\mathrm{Hom}}       
\newcommand{\Spin}{\mathrm{Spin}}       
\newcommand{\Aut}{\mathrm{Aut}}       

\title{Inverse Limits of Spectral Triples}
\begin{document}
\maketitle  \setlength{\parindent}{0pt}
\begin{center}
\author{
{\textbf{Petr R. Ivankov*}\\
e-mail: * monster.ivankov@gmail.com \\
}
}
\end{center}

\vspace{1 in}

\noindent

\paragraph{}

Gelfand - Na\u{i}mark theorem supplies a one to one correspondence between commutative $C^*$-algebras and locally compact Hausdorff spaces. So any noncommutative $C^*$-algebra can be regarded as a generalization of a topological space.  Similarly a spectral triple is a generalization of a Riemannian manifold. An (infinitely listed) covering of a Riemannian manifold has natural structure of  Riemannian manifold. Here we will consider the noncommutative generalization of this result.

\tableofcontents

\section{Motivation. Preliminaries}

\paragraph{} 
Some notions of the geometry have noncommutative generalizations based from the Gelfand-Na\u{i}mark theorem.
\begin{thm}\label{gelfand-naimark}\cite{arveson:c_alg_invt} (Gelfand-Na\u{i}mark). 
Let $A$ be a commutative $C^*$-algebra and let $\mathcal{X}$ be the spectrum of A. There is the natural $*$-isomorphism $\gamma:A \to C_0(\mathcal{X})$.
\end{thm}

\paragraph{}From the theorem it follows that a (noncommutative) $C^*$-algebra may be regarded as a generalized (noncommutative)  locally compact Hausdorff topological space. 
\paragraph{} Any Riemannian manifold $M$ with a Spin structure defines the standard spectral triple $\left(M, L^2\left(M, S\right), \slashed D\right)$ \cite{hajac:toknotes}, i.e. the spectral triple is pure algebraic construction of manifold with a Spin structure. If $\widetilde{M} \to M$ is an infinitely listed covering projection then there is a sequence of finitely listed covering projections

\begin{equation*}
... \to M_2 \to M_1 \to M_0 = M
\end{equation*}
which induces natural covering projections $\pi_n : \widetilde{M} \to M_n$ and for any $\widetilde{x}_1, \widetilde{x}_2 \in \widetilde{M}$ from $\widetilde{x}_1 \neq \widetilde{x}_2$ it follows that there is $k\in \mathbb{N}$  which satisfies to the condition $\pi_k\left(\widetilde{x}_1\right)\neq \pi_k\left(\widetilde{x}_2\right)$. From the topology and the  differential geometry one can construct properties of $\widetilde{M}$ from the above sequence of finite covering projections. We will develop a pure algebraic construction of $\widetilde{M}$. If we have a sequence of spectral triples $\left\{\left(C^\infty\left(M\right), L^2\left(M_n, S_n\right), \slashed D_n\right)\right\}_{n \in \mathbb{N}}$ then we will construct a triple $\left(C^\infty\left(\widetilde{M}\right), L^2\left(\widetilde{M}, \widetilde{S}\right), \widetilde{\slashed D}\right)$ which reflects properties of the  Riemannian manifold $\widetilde{M}$ with its Spin structure. Since our construction is pure algebraic we can apply it to the noncommutative case. A  noncommutative example of this construction is presented in the Section \ref{nt_sec}.
\paragraph{} This article assumes elementary knowledge of following subjects:
 \begin{enumerate}
 \item Set theory \cite{halmos:set},
 \item Category theory  \cite{spanier:at},
 \item General topology \cite{munkres:topology},
 \item Algebraic topology including $K$-theory, \cite{spanier:at,karoubi:k},
 \item Differential geometry \cite{brickell_clark:diff_m,do_carmo:rg,koba_nomi:fgd},
 \item $C^*$-algebras, $C^*$-Hilbert modules and $K$-theory  \cite{blackadar:ko,kakariadis:corr,pedersen:ca_aut}.
  
 \end{enumerate}
 \paragraph{}
 The terms ``set'', ``family'' and ``collection'' are synonyms.
  
\paragraph{}Following table contains a list of special symbols.
\newline
\begin{tabular}{|c|c|}
\hline
Symbol & Meaning\\
\hline
$A^+$  & Unitization of $C^*-$ algebra $A$\\
$A^0$  & Opposite algebra of  $A$  consisting of elements
$\{a^0 : a \in A\}$ \\
& with product $a^0b^0 = (ba)^0$.\\
$\widehat{A}$ & Spectrum of  $C^*$ - algebra $A$  with the hull-kernel topology \\
 & (or Jacobson topology)\\
$A^G$  & Algebra of $G$ invariants, i.e. $A^G = \{a\in A \ | \ ga=a, \forall g\in G\}$\\
$\mathrm{Aut}(A)$ & Group * - automorphisms of $C^*$  algebra $A$\\
$B(H)$ & Algebra of bounded operators on Hilbert space $H$\\
$\mathbb{C}$ (resp. $\mathbb{R}$)  & Field of complex (resp. real) numbers \\
$C(\mathcal{X})$ & $C^*$-algebra of continuous complex valued \\
 & functions on a compact space $\mathcal{X}$\\
$C_0(\mathcal{X})$ & $C^*$-algebra of continuous complex valued functions\\ 
 &  on a locally compact topological  space $\mathcal{X}$ equal to $0$ at infinity\\
$G(\widetilde{\mathcal{X}} | \mathcal{X})$ & Group of covering transformations of a covering projection  $\widetilde{\mathcal{X}} \to \mathcal{X}$ \cite{spanier:at}  \\
$\mathfrak{cl}\left(\mathcal U\right)\subset \mathcal X$ & The closure of subset $\mathcal U$ of topological space $\mathcal X$.   \\
$\delta_{ij}$ & Delta symbol. If $i = j$ then $\delta_{ij}=1$.  If $i \neq j$ then $\delta_{ij}=0$  \\
$\Ga(\mathcal X, E)$ & A $C(\mathcal{X})$-module of sections of a locally trivial vector bundle $E \in \mathrm{Vect}(\mathcal{X})$ \\
$H$ & Hilbert space \\
$\sH_A$ & Hilbert space over  $A$ (Definition \ref{hilb_space_a}) \\
$\mathcal{K}= \mathcal{K}(H)$ & $C^*$ - algebra of compact operators \\
$\mathcal{K}(X_A)$ & $C^*$-algebra of compact operators of a Hilbert $A$ module $X_A$ \\
$K_i(A)$ ($i = 0, 1$) & $K$ groups of $C^*$-algebra $A$\\
$\varprojlim$ & Inverse limit \\
$M(A)$  & A multiplier algebra of $C^*$-algebra $A$\\
$\mathbb{M}_n(A)$  & The $n \times n$ matrix algebra over $C^*-$ algebra $A$\\
$\mathbb{N}$  & A set of positive integer numbers\\
$\mathbb{N}^0$  & A set of nonnegative integer numbers\\
$\mathscr P\left(A\right)$ & A category of finitely generated projective modules over $A$\\
$S^n$ & The $n$-dimensional sphere\\
$SU(n)$ & Special unitary group \\



$TM$ (resp. $T^*M$) & Tangent (resp. cotangent) bundle of differentiable manifold $M$ \cite{koba_nomi:fgd}\\$U(H) \subset \mathcal{B}(H) $ & Group of unitary operators on Hilbert space $H$\\
$U(A) \subset A $ & Group of unitary operators of algebra $A$\\
$U(n) \subset GL(n, \mathbb{C}) $ & Unitary subgroup of general linear group\\
$\mathrm{Vect}(\mathcal{X})$ & A category of locally trivial vector bundles over a topological space $\mathcal X$ \cite{karoubi:k}\\ 
$\mathbb{Z}$ & Ring of integers \\

$\mathbb{Z}_n$ & Ring of integers modulo $n$ \\
$\overline{k} \in \mathbb{Z}_n$ & An element in $\mathbb{Z}_n$ represented by $k \in \mathbb{Z}$  \\
$X \backslash A$ & Difference of sets  $X \backslash A= \{x \in X \ | \ x\notin A\}$\\
$|X|$ & Cardinal number of the finite set\\ 
$f|_{A'}$& Restriction of a map $f: A\to B$ to $A'\subset A$, i.e. $f|_{A'}: A' \to B$\\ 
\hline
\end{tabular}

\break


  


\begin{defn}\label{g_cov_defn}
Let 
\begin{equation}\label{group_cov_defn}
G_1 \leftarrow G_2 \leftarrow ...
\end{equation}
 be an infinite sequence of finite groups and epimorphisms, and let $G$ be a group with epimorpisms $h_n: G \to G_n$. The sequence is said to be {\it coherent} if $\bigcap\mathrm{ker} \ h_n$ is trivial and the following diagram is commutative.

\begin{tikzpicture}\label{borel_local_comm}
  \matrix (m) [matrix of math nodes,row sep=3em,column sep=4em,minimum width=2em]
  {
       & G  &  \\
    G_n &  &  G_{n-1} \\};
  \path[-stealth]
    (m-1-2) edge node [left] {$h_n$} (m-2-1)
    (m-1-2) edge node [right] {$h_{n-1}$} (m-2-3)
    (m-2-1) edge node [left] {}  (m-2-3);
  
\end{tikzpicture}

A family $\left\{G^n\subset G\right\}_{n \in \mathbb{N}}$ is said to be a $G$-{\it covering} of the sequence \eqref{group_cov_defn} if following conditions hold:
\begin{enumerate}
\item If $m < n$ then $G^m \subset G^n$,
\item $\left|G^n\right|=\left|G_n\right|$ for any $n \in \mathbb{N}$,
\item $h_n\left(G^n\right)=G_n$,
\item $\bigcup_{n \in \mathbb{N}}G^n = G$.
\end{enumerate}
\end{defn}
\begin{defn}
Let us consider a sequence  \eqref{group_cov_defn} groups and epimorphisms. There are an inverse limit $\overline{G}=\varprojlim G_n$ \cite{spanier:at} and natural epimorphisms $h_n: \overline{G}\to G_n$. We say that  element $\overline{g}$ \textit{is represented} by the sequence $\left\{g_n \in G_n\right\}_{n \in \mathbb{N}}$ if $g_n=h_n\left(\overline{g}\right)$. We will write $\overline{g} = \mathfrak{Rep}_G\left(\left\{g_n \right\}\right)$.
\end{defn}

Henceforth  $\left\{x_{\iota}\right\}_{\iota \in I}$ means a set indexed by finite or countable set $I$ of indexes. 
\begin{defn}\label{ext_of_scalars_funct}\cite{bass}
Any homomorphisms of rings  $A \to B$ induces functor $\mathscr P(A) \to \mathscr P(B)$ of categories of finitely generated projective modules. The functor is given by $P \mapsto B \otimes_A P$ on objects, and $f \mapsto \mathrm{Id}_B \otimes f$ on morphisms. We call it the  {\it extensions of scalars} functor.
\end{defn}

\subsection{Topology}
\subsubsection{Covering projections and partition of unity}
\begin{defn}\cite{spanier:at}
	Let $\widetilde{\pi}: \widetilde{\mathcal{X}} \to \mathcal{X}$ be a continuous map. An open subset $\mathcal{U} \subset \mathcal{X}$ is said to be {\it evenly covered } by $\widetilde{\pi}$ if $\widetilde{\pi}^{-1}(\mathcal U)$ is the disjoint union of open subsets of $\widetilde{\mathcal{X}}$ each of which is mapped homeomorphicaly onto $\mathcal{U}$ by $\widetilde{\pi}$. A continuous map $\widetilde{\pi}: \widetilde{\mathcal{X}} \to \mathcal{X}$ is called a {\it covering projection} if each point $x \in \mathcal{X}$ has an open neighborhood evenly covered by $\widetilde{\pi}$. $\widetilde{\mathcal{X}}$ is called the {
		\it covering space} and $\mathcal{X}$ the {\it base space} of covering projection.
\end{defn}
\begin{defn}\label{one_one_cov_defn}
   Let $\widetilde{\pi}: \widetilde{\mathcal{X}} \to \mathcal{X}$ be a covering projection. A connected open subset $\widetilde{\mathcal{U}} \subset \widetilde{\mathcal{X}}$ is said to be a {\it one-to-one subset} with respect to $\widetilde{\pi}$ if the restriction $\widetilde{\pi}_{\widetilde{\mathcal{U}}}: \widetilde{\mathcal{U}} \to \widetilde{\pi}\left(\widetilde{\mathcal{U}}\right)$ is a homeomorphism. The family $\left\{\widetilde{\mathcal{U}}_\iota\right\}_{\iota \in I}$ of one-to-one subsets with respect to $\widetilde{\pi}$ such that $\widetilde{\mathcal{X}} = \bigcap_{\iota \in I}\widetilde{\mathcal{U}}_\iota$ is said to be a {\it one-to-one covering}  with respect to $\widetilde{\pi}$.
\end{defn}
\begin{defn}\cite{spanier:at}
	A fibration $p: \mathcal{\widetilde{X}} \to \mathcal{X}$ with unique path lifting is said to be  {\it regular} if, given any closed path $\omega$ in $\mathcal{X}$, either every lifting of $\omega$ is closed or none is closed.
\end{defn}
\begin{defn}\label{cov_proj_cov_grp}\cite{spanier:at}
	Let $p: \mathcal{\widetilde{X}} \to \mathcal{X}$ be a covering projection.  A self-equivalence is a homeomorphism $f:\mathcal{\widetilde{X}}\to\mathcal{\widetilde{X}}$ such that $p \circ f = p$). We denote this group by $G(\mathcal{\widetilde{X}}|\mathcal{X})$. This group is said to be the {\it group of covering transformations} of $p$ or the {\it covering group}.
\end{defn}
\begin{prop}\cite{spanier:at}
	If $p: \mathcal{\widetilde{X}} \to \mathcal{X}$ is a regular covering projection and $\mathcal{\widetilde{X}}$ is connected and locally path connected, then $\mathcal{X}$ is homeomorphic to space of orbits of $G(\mathcal{\widetilde{X}}|\mathcal{X})$, i.e. $\mathcal{X} \approx \mathcal{\widetilde{X}}/G(\mathcal{\widetilde{X}}|\mathcal{X})$. So $p$ is a principal bundle.
\end{prop}

\paragraph*{}In this article we consider second-countable locally compact Hausdorff spaces only \cite{munkres:topology}. So we will say a "topological space" (resp. "compact  space" ) instead "locally compact second-countable Hausdorff space" (resp. "compact second-countable Hausdorff space").

\begin{thm}\label{comp_normal}\cite{munkres:topology}
	Every compact Hausdorff space is normal.
\end{thm}
\begin{thm}\label{urysohn_lem}\cite{munkres:topology}\textbf{ Urysohn lemma.}
	Let $\mathcal X$ be a normal space, let $\mathcal A$, $\mathcal B$ be disjoint closed subsets of $\mathcal X$. Let $\left[a,b\right]$ be a closed interval in the real line. Then there exist a continuous map $f: \mathcal X \to \left[a, b\right]$ such that $f(\mathcal A)=\{a\}$ and $f(\mathcal B)=\{b\}$.
\end{thm}

\begin{thm}\cite{munkres:topology}\label{urysohn_metr_lem}\textbf{ Urysohn metrization theorem.}
	Every regular space with a countable basis is metrizable.
\end{thm}
From the Theorems \ref{comp_normal} and \ref{urysohn_lem} it follows that if $\mathcal X$ is locally compact Hausdorff space $x \in \mathcal X$, and $\mathcal B$ is closed subset of $\mathcal X$, such that $x \notin \mathcal B$ then there exist a continuous map $f: \mathcal X \to \left[a, b\right]$ such that $f(x)=a$ and $f(\mathcal B)=\{b\}$. It means that locally compact Hausdorff space is completely regular, whence $\mathcal X$ is regular (See \cite{munkres:topology}), and from the Theorem \ref{urysohn_metr_lem} it follows next corollary.
\begin{cor}\label{loc_comp_haus_metr_col}
	Every locally compact second-countable Hausdorff space is metrizable.
\end{cor}

\begin{thm}\cite{munkres:topology}
	Every metrizable space is paracompact.
\end{thm}

\begin{defn}\cite{munkres:topology}
	Let $\left\{\mathcal U_\alpha\in \mathcal X\right\}_{\alpha \in J}$ be an indexed open covering of $\mathcal{X}$. An indexed family of functions 
	\begin{equation*}
	\phi_\alpha : \mathcal X \to \left[0,1\right]
	\end{equation*}
	is said to be a {\it partition of unity }, dominated by $\left\{\mathcal{U}_\alpha \right\}_{\alpha \in J}$, if:
	\begin{enumerate}
		\item $\phi_\alpha\left(\mathcal X \backslash \mathcal U_\alpha\right)= \{0\}$
		\item The family $\left\{\mathrm{Support}\left(\phi_\alpha\right) = \mathfrak{cl}\left(\left\{x \in \mathcal X \ | \ \phi_\alpha > 0 \right\}\right)\right\}$ is locally finite.
		\item $\sum_{\alpha \in J}\phi_\alpha\left(x\right)=1$ for any $x \in \mathcal X$.
	\end{enumerate}
\end{defn}

\begin{thm}\cite{munkres:topology}
	Let $\mathcal X$ be a paracompact Hausdorff space; let $\left\{\mathcal U_\alpha\in \mathcal X\right\}_{\alpha \in J}$ be an indexed open covering of $\mathcal{X}$. Then there exists a partition of unity, dominated by $\left\{\mathcal{U}_\alpha \right\}$.  
\end{thm}

\begin{thm}\label{pavlov_troisky_thm}\cite{pavlov_troisky:cov}
	Suppose $X$ and $Y$ are compact Hausdorff connected spaces and $p : Y \to X$
	is a continuous surjection. If $C(Y )$ is a projective finitely generated Hilbert module over
	$C(X)$ with respect to the action
	\begin{equation*}
	(f\xi)(y) = f(y)\xi(p(y)), \ f \in  C(Y ), \xi \in  C(X),
	\end{equation*}
	then $p$ is a finite-fold (or a finilely listed) covering.
\end{thm}

\subsubsection{Vector bundles and projective modules}

\begin{defn}\cite{karoubi:k}
Let $\mathcal X$ be a topological space.
A {\it quasi-vector bundle with base}  $\mathcal X$ is given by
\begin{enumerate}
\item a finite dimensional $\mathbb{C}$-vector space $E_x$ for any $x \in \mathcal X$,
\item a topology on the disjoint union $E= \bigsqcup E_x$ 
\end{enumerate}
A quasi-vector bundle is denoted by $\xi = \left(E, \pi, \mathcal X\right)$ or simply $E$ if there is no risk of confusion. The space $E$ is the {\it total space} of $\xi$ and $E_x$ is the {\it fiber} of $\xi$ at the point $x$.
\end{defn}

\begin{rem}
Above definition is a specific case of discussed in \cite{karoubi:k} definition which includes real vector bundles. This article discusses complex vector bundles only.
\end{rem}
\begin{defn}\label{quasi_bundle_category} \cite{karoubi:k}
Let $\xi = \left(E, \pi, \mathcal X\right)$ and $\xi = \left(E', \pi', \mathcal X'\right)$ be quasi-vector bundles. A {\it general morphism} from $\xi$ to $\xi'$ is given by a pair $\left(f,g\right)$ of continuous maps such that
\begin{enumerate}
\item the diagram 
 \newline
 \begin{tikzpicture}
   \matrix (m) [matrix of math nodes,row sep=3em,column sep=4em,minimum width=2em]
   {
     E  & E'   \\
     \mathcal X &  \mathcal X' \\};
   \path[-stealth]
     (m-1-1) edge node [above] {$g$} (m-1-2)
     (m-2-1) edge node [above] {$f$} (m-2-2)
     (m-1-1) edge node [left]  {$\pi$} (m-2-1)
     (m-1-2) edge node [right] {$\pi'$} (m-2-2);
  \end{tikzpicture}
  
  is commutative.
  \item The map $g_x: E_x \to E'_{f(x)}$ induced by $g$ is $\mathbb{C}$-linear. General morphism can be composed by obvious way. 
\end{enumerate}
If $\xi$, $\xi'$ have the same base $\mathcal X=\mathcal X'$, a  {\it morphism} between  $\xi$ and $\xi'$ is a general morphism $\left(f,g\right)$ such that $f = \mathrm{Id}_{\mathcal X}$. Such a morphism will be simply called $g$ in the sequel. The quasi-vector bundles with the same base $\mathcal X$ are objects of a category, whose arrows we have just defined.
\end{defn}

\begin{empt}\cite{karoubi:k}
Let $\xi = \left(E, \pi, \mathcal X\right)$ be a quasi-vector bundle, and let $\mathcal X'$ be a subspace of $\mathcal X$. The triple $\left(\pi^{-1}\left(\mathcal X'\right), \pi |_{\pi^{-1}\left(\mathcal X'\right)}, \mathcal X'\right)$ defines a quasi-vector bundle $\xi'$ which is called a restriction of $\xi$ to $\mathcal X'$. We denote it by $\xi |_{\mathcal X'}$, $E |_{\mathcal X'}$ or even simply $E_{\mathcal X'}$.
\end{empt}
\begin{defn}\label{inv_image_top_defn}\cite{karoubi:k}
More generally, let $f: \mathcal X' \to\mathcal X$ be any continuous map. For any $x' \in \mathcal X'$, let $E'_{x'}=E_{f(x')}$. Then the set $E'= \bigsqcup_{x' \in \mathcal X'}E'_{x'}$ may be identified with the {\it fiber product} $\mathcal X' \times_{\mathcal X} E$, i.e. with the subset of $\mathcal X' \times E$ formed by pairs $\left(x', e\right)$ such that $f\left(x'\right)=\pi(e)$. If $\pi': E'\to X'$ is defined by $\pi'(x', e) = x'$, then it is clear that the triple $\left(E', \pi', \mathcal X'\right)$ defines a quasi-vector bundle over $\mathcal X'$, when we provide $E'$ with the topology induced by $\mathcal X' \times E$. We write $\xi'$ as $f^*(\xi)$, or $f^*(E)$: this is the {\it inverse image} (or the {\it pullback}) of $\xi$ by $f$. We have $f^*(\xi)= \xi$ for $f=\mathrm{Id}_{\mathcal X}$, and also $\left(f \circ f'\right)^*(\xi)=f'^*\left(f^*\left(\xi\right)\right)$ if $f': \mathcal X'' \to\mathcal X'$ is another continuous map. We also say the $E'$ is the pullback of $E$ by $f$.
\end{defn}
\begin{empt}\label{inv_image_functor_top}\cite{karoubi:k}
Let us now consider two quasi-vector bundles over $\mathcal{X}$ and a morphism $\al: E \to F$. If we let $E'=f^*(E)$ as in \ref{inv_image_top_defn} and  $F'=f^*(F')$ we can also define a morphism $\al'=f^*(\al)$ from $E'$ to $F'$ by the formula $\al'_{x'}=\al_{f(x')}$. If we identify $E'$ with $\mathcal X' \times_{\mathcal X} E$ and $F'$ with $\mathcal X' \times_{\mathcal X} F$, then $\al'$ is identified with  $\mathrm{Id}_{\mathcal X'} \times_{\mathcal X} \al$ which proves the continuity of the map $\al'$.

 \begin{tikzpicture}
   \matrix (m) [matrix of math nodes,row sep=3em,column sep=4em,minimum width=2em]
   {
  \  &  f^*(F) & \  & F   \\
 f^*(E)  & \ &  E &  \\ 
\  &  \mathcal X'  & \ & \mathcal X \\};
   \path[-stealth]
     (m-1-2) edge node [above] {$ $} (m-1-4)
     (m-2-1) edge node [above] {$ $} (m-2-3)
     (m-2-1) edge node [above] {$f^*(\al)$} (m-1-2)
     (m-2-1) edge node [above] {$ $} (m-3-2)
     (m-2-3) edge node [above] {$\al$} (m-1-4)
     (m-1-2) edge node [above] {$ $} (m-3-2)
     (m-1-4) edge node [above] {$ $} (m-3-4)
     (m-2-3) edge node [above] {$ $} (m-3-4)
     (m-3-2) edge node [above] {$ $} (m-3-4);
  \end{tikzpicture}

\end{empt}

\begin{prop}\cite{karoubi:k}
Let $f: \mathcal X' \to\mathcal X$ be a continuous map. Then the correspondence $E \mapsto f^*(E)$ and $\al \mapsto f^*(\al)$ induces a functor between the category of quasi-vector bundles over $\mathcal X$ and the category of quasi-vector bundles over $\mathcal X'$
\end{prop} 

\begin{empt}\cite{karoubi:k}
Let $V= \mathbb{C}^n$ a finite dimensional vector space. There is a natural structure of quasi-vector bundle on the product $\mathcal{X}  \times V$. Such bundles are called {\it trivial vector bundles}.
\end{empt}
\begin{defn}\cite{karoubi:k}
Let $\xi = \left(E, \pi, \mathcal X\right)$ be a quasi-vector bundle. Then $\xi$ is said to be {\it "locally trivial"} or a {\it "vector bundle"} if for any $x \in \mathcal{X}$ there exists a neighborhood $\mathcal U$ such that restriction $\xi |_{\mathcal{U}}$ is isomorphic to a trivial bundle.
\end{defn}
\begin{fact}
Vector bundles are in fact the objects of a full subcategory of the category of quasi-vector bundles considered in the Definition \ref{quasi_bundle_category}. We will denote this category by $\mathscr E(\mathcal X)$. If $f: \mathcal X' \to \mathcal X$ is a continuous map then functor $f^*$ defined in \ref{inv_image_top_defn} defines a functor from $\mathscr E(\mathcal X)$ to $\mathscr E(\mathcal X')$, because an inverse image of any vector bundle is also a vector bundle (See \cite{karoubi:k}).
\end{fact}
\begin{empt}\cite{karoubi:k}
Let $A =  C\left(\mathcal X\right)$ be a ring of continuous complex valued functions on a compact space $\mathcal X$. If $E$ is a vector bundle with base $\mathcal X$ then the set $\Ga\left(\mathcal X, E\right)$ of continuous sections is an $A$-module under the operation $(\la\cdot s)(x)=\la(x)s(x)$ where $\la \in A$, $s \in \Ga\left(\mathcal X, E\right)$. 
\end{empt}
\begin{thm}\cite{karoubi:k} {\it Serre - Swan theorem}.
Let $A = C\left(\mathcal X\right)$ be a ring of continuous complex valued functions on a compact space $\mathcal X$. Then the section functor $\Ga$ induces an equivalence of categories $\mathscr E(\mathcal{X}) \approx \mathscr P(A)$, where $\mathscr P(A)$ is a category of finitely generated projective $A$-modules.
\end{thm}

\begin{empt}\label{tensor_bundle}
Let $f: \mathcal X' \to \mathcal X$ a continuous map, there is an inverse image functor $f^*$ from $\mathscr E(\mathcal X) \to \mathscr E(\mathcal X')$ described in \ref{inv_image_functor_top}. If we identify $\mathscr E(\mathcal X)$ (resp. $\mathscr E(\mathcal X')$) with $\mathscr P\left(C\left(\mathcal{X}\right)\right)$ (resp. $\mathscr P\left(C\left(\mathcal{X}'\right)\right)$) then $f^*$ may be interpreted as the "extensions of scalars" functor $\mathscr P\left(C\left(\mathcal{X}\right)\right)\to \mathscr P\left(C\left(\mathcal{X}'\right)\right)$ defined by $P \mapsto C\left(\mathcal{X}'\right) \otimes P$ and $h \mapsto 1_{C\left(\mathcal X'\right)} \otimes h$ for any morphism $h$ in $\mathscr P\left(C\left(\mathcal{X}\right)\right)$. 
\end{empt}
\begin{defn}\label{gluing_sections_def}
 Let $\xi = \left(E, \pi, \mathcal X\right)$ be a vector bundle, and let $\left\{\mathcal{U}_\iota\subset\mathcal{X}\right\}_{\iota \in I}$ be a family of open subsets such that $\mathcal{X} = \bigcap_{\iota \in I} \mathcal{U}_\iota$. If $\left\{s_\iota \in \Ga\left(\mathcal{U}_\iota, E|_{\mathcal{U}_\iota}\right)\right\}_{\iota \in I}$ is a family of sections such that for any $\iota', \iota'' \in I$ following condition hold
 \begin{equation*}
 s_{\iota'}|_{\mathcal{U}_{\iota'} \bigcap \mathcal{U}_{\iota''}}=s_{\iota''}|_{\mathcal{U}_{\iota'} \bigcap \mathcal{U}_{\iota''}}
 \end{equation*}
 then there is the unique section $s \in  \Ga\left(\mathcal{X}, E\right)$ such that
  \begin{equation*}
  s|_{\mathcal{U}_{\iota}}=s_{\iota};~\forall \iota \in I.
  \end{equation*}
  The section $s$ is said to be the {\it gluing} of $\left\{s_\iota\right\}_{\iota \in I}$ and we will write $s= \mathfrak{Gluing}\left(\left\{s_\iota\right\}_{\iota \in I}\right)$.
  
\end{defn}
\begin{rem}
	The Definition \ref{gluing_sections_def} describes the gluing of  sections of the sheaf (See \cite{hartshorne:ag} for details).
\end{rem}
\begin{defn}\label{pullback_op_def}
 Let $\widetilde{\pi}:\widetilde{\mathcal X}	\to \mathcal X$ be a covering projection and let $\widetilde{\mathcal U} \subset \widetilde{\mathcal X}$ be a one-to-one subset. Let $\xi = \left(E, \pi, \mathcal X\right)$ be a vector bundle and let $\widetilde{\xi} = \left(\widetilde{E}, \pi', \widetilde{\mathcal X}\right)$ be a pullback of $\xi$. There is a natural isomorphism
 \begin{equation*}
\varphi: \Ga\left(\widetilde{\pi}\left(\widetilde{\mathcal U}\right),E|_{\widetilde{\pi}\left(\widetilde{\mathcal U}\right)}\right) \xrightarrow{\approx} \Ga\left(\widetilde{\mathcal U},\widetilde{E}|_{\widetilde{\mathcal U}}\right).
 \end{equation*}
 If $s \in \Ga\left(\widetilde{\pi}\left(\widetilde{\mathcal U}\right),E|_{\widetilde{\pi}\left(\widetilde{\mathcal U}\right)}\right)$ is a section then the section $\widetilde{s}=\varphi\left(s\right)\in \Ga\left(\widetilde{\mathcal U},\widetilde{E}|_{\widetilde{\mathcal U}}\right)$ is said to be the {\it $\widetilde{\mathcal U}$-pullback} of $s$ and we will write $\widetilde{s}= \mathfrak{pullback}_{\widetilde{\mathcal{U}}}\left(s\right)$. Elements of  $C_0\left(\mathcal X\right)$ and $C_0\left(\widetilde{\mathcal X}\right)$ can be regarded as sections of one dimensional trivial bundles, so we will use notation $\widetilde{f}= \mathfrak{pullback}_{\widetilde{\mathcal{U}}}\left(f\right)$ where $f \in C_0\left(\mathcal X\right)$ and $\widetilde{f}\in C_0\left(\widetilde{\pi}\left(\widetilde{\mathcal U}\right)\right)$. 
\end{defn}
\begin{defn}\label{pullback_global_def}
Let us consider situation of the Definition \ref{pullback_op_def}. If be $s \in \Ga\left(\mathcal X,E\right)$ is a section then following condition hold
\begin{equation*}
\mathfrak{pullback}_{\widetilde{\mathcal{U}}'}\left(s\right)|_{\widetilde{\mathcal{U}}'\bigcap \widetilde{\mathcal{U}}''} = \mathfrak{pullback}_{\widetilde{\mathcal{U}}''}\left(s\right)|_{\widetilde{\mathcal{U}}'\bigcap \widetilde{\mathcal{U}}''}
\end{equation*}
for any $\widetilde{\mathcal{U}}'$ (resp. $\widetilde{\mathcal{U}}''$) one-to-one subsets. Hence there is the unique section $\widetilde{s}=\Ga\left(\widetilde{\mathcal X},\widetilde{E}\right)$ such that 
\begin{equation*}
\widetilde{s}|_{\widetilde{\mathcal{U}}} = \mathfrak{pullback}_{\widetilde{\mathcal{U}}}\left(s\right)
\end{equation*}
for any  one-to-one subset $\widetilde{\mathcal{U}}$. The section $\widetilde{s}$ is said to be the $\widetilde{\pi}${\it -pullback} of $s$ and we will write 
\begin{equation*}
\widetilde{s} = \mathfrak{pullback}_{\widetilde{\pi}}\left(s\right).
\end{equation*}
It is clear that following condition hold
\begin{equation*}
\mathfrak{pullback}_{\widetilde{\pi}}\left(s\right) = 1_{C\left(\widetilde{\mathcal X}\right)} \otimes_{C\left(\mathcal X\right)} s.
\end{equation*}
\end{defn}
\subsection{Hilbert modules}
\paragraph{} We refer to \cite{blackadar:ko,kakariadis:corr} for the definition of Hilbert $C^*$-modules, or simply Hilbert modules. Denote by $X_A$ a right Hilbert $A$-module. The sesquilinear product on a Hilbert module $X_A$ will be denoted by $\langle \cdot,\cdot\rangle_{X_A}$. 
 For any $\xi, \zeta \in X_A$ let us define an $A$-endomorphism $\theta_{\xi, \zeta}$ given by  $\theta_{\xi, \zeta}(\eta)=\xi \langle \zeta, \eta \rangle_{X_A}$ where $\eta \in X_A$. Operator  $\theta_{\xi, \zeta}$ shall be denoted by $\xi \rangle\langle \zeta$. Norm completion of algebra generated by operators $\theta_{\xi, \zeta}$ is said to be an algebra of compact operators $\mathcal{K}(X_A)$. We suppose that there is a left action of $\mathcal{K}(X_A)$ on $X_A$ which is $A$-linear, i.e. action of  $\mathcal{K}(X_A)$ commutes with action of $A$.
 \begin{defn}\cite{kakariadis:corr}
  \textup{An \emph{$A$-$B$-correspondence $X$} is a right Hilbert $B$-module together with a $*$-homomorphism $\phi_X\colon A \rightarrow \mathscr{L}(X)$. We will denote this by $_A X_B$.}
  \end{defn}
  \begin{defn}\label{hilb_space_a}\cite{blackadar:ko}
Let $A$ be a $C^*$-algebra, let $\sH_{A}$ be the completion of the direct sum of a countable number of copies of $A$, i.e. $\sH_{A}$ consists of all sequences $\left\{a_n \in A\right\}_{n \in \mathbb{N}}$ such that $\sum_{n=1}^{\infty}a^*_na_n$ converges, with inner product
\begin{equation*}
\left\langle \left\{a_n \right\},\left\{a_n \right\}\right\rangle_{\sH_{A}}=\sum_{n=1}^{\infty}a^*_nb_n.
\end{equation*}
$\sH_{A}$ is said to be the \textit{Hilbert space over} $A$. Henceforth the Hilbert space over $A$ will by denoted by $\sH_{A}$.
  \end{defn}
  
  \paragraph*{}The sesquilinear product on a Hilbert space $H$ will be denoted by $( \cdot,\cdot)$. 
  For any $\xi, \zeta \in H$ let us define an operator $\Theta_{\xi, \zeta}\in B(H)$ given by  $\Theta_{\xi, \zeta}(\eta)= ( \zeta, \eta )\xi$ where $\eta \in H$. Operator  $\theta_{\xi, \zeta}$ shall be denoted by $\xi )( \zeta$.

 
   \subsection{Riemannian manifolds and covering projections}
   
   \begin{prop}\label{comm_cov_mani}(Proposition 5.9 \cite{koba_nomi:fgd})
   	\begin{enumerate}
   		\item Given a connected manifold $M$ there is a unique (unique up to isomorphism) universal covering manifold, which will be denoted by $\widetilde{M}$.
   		\item The universal covering manifold $\widetilde{M}$ is a principal fibre bundle over $M$ with group $\pi_1(M)$ and projection $p: \widetilde{M} \to M$, where $\pi_1(M)$ is the first homotopy group of $M$.
   		\item The isomorphism classes of covering spaces over $M$ are in 1:1 correspondence with the conjugate classes of subgroups of $\pi_1(M)$. The correspondence is given as follows. To each subgroup $H$ of $\pi_1(M)$, we associate $E=\widetilde{M}/H$. Then the covering manifold $E$ corresponding to $H$ is a fibre bundle over $M$ with fibre $\pi_1(M)/H$ associated with the principal bundle  $\widetilde{M}(M, \pi_1(M))$. If $H$ is a normal subgroup of $\pi_1(M)$, $E=\widetilde{M}/H$ is a principal fibre bundle with group $\pi_1(M)/H$ and is called a regular covering manifold of $M$.
   	\end{enumerate}
   \end{prop}
   \begin{prop}\label{smooth_part_unity_prop}\cite{brickell_clark:diff_m}
   	A differential manifold $M$ admits a (smooth) partition of unity if and only if it is paracompact. 
   \end{prop}
   \begin{empt}
   	If $\widetilde{M}$ is a covering space of Riemannian manifold $M$ then it is possible to give $\widetilde{M}$ a Riemannian structure such that $\pi: \widetilde{M} \to M$ is a local isometry (this metric is called the {\it covering metric}).  See \cite{do_carmo:rg} for details.
   \end{empt}

\subsection{Strong and/or weak extension}
   \paragraph{}In this section I follow to \cite{pedersen:ca_aut}.
   \begin{defn}\cite{pedersen:ca_aut}
    Let $A$ be a $C^*$-algebra.  The {\it strict topology} on $M(A)$ is the topology generated by seminorms $\vertiii{x}_a = \|ax\| + \|xa\|$, ($a\in A$). If $x \in M(A)$  and a sequence of partial sums $\sum_{i=1}^{n}a_i$ ($n = 1,2, ...$), ($a_i \in A$) tends to $x$ in the strict topology then we shall write
   \begin{equation*}
    x = \sum_{i=1}^{\infty}a_i.
    \end{equation*}
    \end{defn}
 \begin{defn}\cite{pedersen:ca_aut} Let $H$ be a Hilbert space. The {\it strong} topology on $B\left(H\right)$ is the locally convex vector space topology associated with the family of seminorms of the form $x \mapsto \|x\xi\|$, $x \in B(H)$, $\xi \in H$.
    \end{defn}
 \begin{defn}\cite{pedersen:ca_aut} Let $H$ be a Hilbert space. The {\it weak} topology on $B\left(H\right)$ is the locally convex vector space topology associated with the family of seminorms of the form $x \mapsto \left|\left(x\xi, \eta\right)\right|$, $x \in B(H)$, $\xi, \eta \in H$.
          \end{defn}
    
    \begin{thm}\label{vN_thm}\cite{pedersen:ca_aut}
 Let $M$ be a $C^*$-subalgebra of $B(H)$, containing the identity operator. The following conditions are equivalent:
 \begin{enumerate}
 \item $M=M''$ where $M''$ is the bicommutant of $M$.
 \item $M$ is weakly closed.
 \item $M$ is strongly closed.
 \end{enumerate}
    \end{thm}
    
    \begin{defn}
   Any $C^*$-algebra $M$ is said to be a {\it von Neumann algebra} or a {\it $W^*$- algebra} if $M$ satisfies to the conditions of the Theorem \ref{vN_thm}.
      \end{defn}

    \begin{defn}\cite{pedersen:ca_aut}
   Let  $B \subset B(H)$ be a $C^*$-algebra and  $B$ acts non-degenerately on $H$. Denote by $B''$ the strong closure of $B$ in $B(H)$. $B''$ is an unital weakly closed $C^*$-algebra . The algebra  $B''$ is said to be the {\it bicommutant},  or the {\it enveloping von Neumann algebra} or  the {\it enveloping $W^*$-algebra}  of $B$. 
    \end{defn}

     \begin{empt}
    Any separable $C^*$-algebra $A$ has a state $\tau$ which induces a faithful {\it GNS} representation  \cite{murphy}. There is a $\mathbb{C}$-valued product on $A$ given by
           \begin{equation*}
          \left(a, b\right)=\tau\left(a^*b\right).
           \end{equation*}
          This product induces a product on $A/\mathcal{I}_\tau$ where $\mathcal{I}_\tau =\left\{a \in A \ | \ \tau(a^*a)=0\right\}$. So $A/\mathcal{I}_\tau$ is a phe-Hilbert space. Let denote by $L^2\left(A, \tau\right)$ the Hilbert  completion of $A/\mathcal{I}_\tau$.  The Hilbert space  $L^2\left(A, \tau\right)$ is a space of a faithful GNS representation of $A$.
       \end{empt}
          
    \begin{empt}\label{l2_mu}
    If $\mathcal X$ is a second-countable locally compact Hausdorff space then $C_0\left(\mathcal X\right)$ is a separable algebra \cite{chun-yen:separability}. Therefore $C_0\left(\mathcal X\right)$ has a state $\tau$ such that associated with $\tau$  {\it GNS} representation  \cite{murphy} is faithful. From \cite{bogachev_measure_v2} it follows that the state $\tau$ can be represented as the following integral
    \begin{equation}\label{hilb_integral}
  \tau\left(a\right)= \int_{\mathcal X}a \ d\mu
    \end{equation}
    where $\mu$ is a positive measure. In analogy with the Riemann integration, one can define the integral of a
    bounded continuous function $a$ on $\mathcal{X}$. There is a $\mathbb{C}$ valued product on $C_0\left(\mathcal X\right)$ given by
    \begin{equation*}
   \left(a, b\right)=\tau\left(a^*b\right)= \int_{\mathcal X}a^*b \ d\mu,
    \end{equation*}
    whence $C_0\left(\mathcal X\right)$ is a phe-Hilbert space. Denote by $L^2\left(C_0\left(\mathcal X\right), \tau\right)$ or $L^2\left(\mathcal X, \mu\right)$ the Hilbert space completion of $C_0\left(\mathcal X\right)$. From  \cite{murphy,takesaki:oa_ii} it follows that $W^*$-enveloping algebra $C_0\left(\mathcal X\right)$ is isomorphic to the algebra $L^{\infty}\left(\mathcal X, \mu\right)$ (of classes of) essentially bounded complex-valued measurable functions. The  $L^{\infty}\left(\mathcal X, \mu\right)$ is a $C^*$-algebra with the pointwise-defined opreations and the essential norm $f \mapsto \|f\|_\infty$. 
    \end{empt}
 \subsection{Nonstandard analysis}\label{nsa_sec} 

 \subsubsection{Riemannian integration}\label{riemann_int_sub_sec}
\paragraph{} Nonstandard analysis operates with actual infinitesimally small parameters. This procedure enables us replace the Riemannian integration with the summation of infinitesimally small elements. Strict explanation of nonstandard analysis and its applications are contained in \cite{hermann:nonstandard}. Here is an informal explanation. Suppose that $f \in C([0,1])$ is a continuous function and we would like to define the integral
\begin{equation*}
\int_{0}^{1}f\left(x\right) dx.
\end{equation*}
Let $Q$ be a countable set given by $Q = \left\{x \in \mathbb{Q} \bigcap [0,1]~\&~\exists m,n \in \mathbb{N}^0 ~~x = \frac{m}{2^n}\right\}$. Then the integral can be represented as a sum of infinitesimally small numbers
\begin{equation*}
\int_{0}^{1}f\left(x\right) dx = \sum_{q \in Q}a^q.
\end{equation*}
Indeed infinitesimally small number $x$ is a sequence $\{x_n \in \mathbb{C}\}_{n \in \mathbb{N}} $ such that $\lim_{n \to \infty}x_n = 0$. Suppose that $a_q$ is represented by the sequence $\{a^q_n \in \mathbb{C}\}_{n \in \mathbb{N}}$ such that
 		\begin{equation*}
 		a^q_n=\left\{
 		\begin{array}{c l}
 		0 & \mathrm{den}(q) < 2^n\\
 		\frac{f(q)}{2^n} & \mathrm{den}(q) \ge 2^n
 		\end{array}\right.
 		\end{equation*}
where $\mathrm{den}$ means denominator of the irreducible fraction, i.e. $\mathrm{den}\left(\frac{x}{y}\right)=y$. It is clear that
\begin{equation*}
\int_{0}^{1}f\left(x\right) dx = \lim_{n \to \infty}\sum_{q \in Q}a^q_n = \sum_{q \in Q}a^q
\end{equation*}
and
\begin{equation*}
 \lim_{n \to \infty} a^q_n = 0.
\end{equation*}
Above equations mean that the Riemannian integral can be represented as a sum of infinitesimally small numbers.

 \subsubsection{Application to infinitely listed covering projections}
  \paragraph{} Let us consider the following sequence of covering projections
   \begin{equation}
   S^1 \xleftarrow[]{f} S^1 ...  \xleftarrow[]{f} S^1 \xleftarrow[]{f} S^1 \xleftarrow[]{f} ...\xleftarrow[]{f} .... \xleftarrow[]{} \mathbb{R}.
   \end{equation}
 Roughly speaking there is a sequence of homomorphisms  
   \begin{equation*}
   C(S^1)\to C(S^1)\to ...  \to C(S^1)  \to C(S^1)  \to ...\to .... \to C_0(\mathbb{R}).
   \end{equation*}
 Indeed there is no a natural homomorphism $C(S^1)   \to C_0(\mathbb{R})$, there is a correspondence $_{C_0(\mathbb{R})} X_{C(S^1)}$. We would like represent functions in $C_0(\mathbb{R})$ by  functions on $S^1$. If $\mathbb{R}\to S^1$ is a covering projection then any $f \in C\left(S^1\right)$ can be represented by a $2\pi$ periodic function $\widetilde{f} \in C_b(\mathbb{R})$ given by Fourier series
 \begin{equation*}
 \widetilde{f}(\xi)=\sum_{m\in \mathbb{Z}}a_{m} e^{i m \xi}.
 \end{equation*}
 If $\mathbb{R}\to\mathcal{X}_n \xrightarrow{n-\text{listed}}S^1$ then any $f_n \in C(\mathcal{X}_n)$ can be represented as a $2 \pi  n$ periodic function $\widetilde{f}_n$ given by
 \begin{equation*}
 \widetilde{f}_n(\xi)=\sum_{m\in \mathbb{Z}}a_{m} e^{\frac{ i m \xi}{n}}.
 \end{equation*}
 
 All finite listed covering projections from the sequence  give following functions
 \begin{equation}\label{four_ser}
 \widetilde{f}(\xi)=\sum_{q\in \mathbb{Q}}a_{q} e^{ i q \xi}.
 \end{equation}
 These functions cannot represent any nontrivial function in $C_0(\mathbb{R})$. But $C_0(\mathbb{R})$ can be represented by a Fourier transform
 \begin{equation*}
 f(\xi) = \int_{\mathbb{R}}\widehat{f}(x)e^{-2\pi ix\xi}
 \end{equation*}
 where $\widehat{f}\in L^1(\mathbb{R})$.
 However the Fourier transform can be regarded as the series \eqref{four_ser} with infinitesimally small coefficients. Let us consider a dependent on $n\in \mathbb{N}$  series  
 \begin{equation*}
 \sum_{q\in \mathbb{Q}}a^n_q e^{i q\xi}
 \end{equation*}
 such that
 \begin{equation*}
 a^n_{q} \to 0; \ \sum_{q\in \mathbb{Q}}a^n_q e^{ i q\xi} \to f(\xi); \text{ as }n\to \infty,
 \end{equation*}
 i.e. $a^n_{q}$ can be regarded as infinitesimally small coefficients. Any $f \in C_0\left(\mathbb{R}\right)$ is a weak (pointwise) limit of the sequence of  periodic functions $\left\{f_n \in   C_b\left(\mathbb{R}\right)\right\}_{n \in \mathbb{N}}$ given by
 \begin{equation*}
 \begin{split}
 f_n(\xi)=f(\xi)-\left(\xi - n\pi\right)\frac{f\left(n\pi\right)-f\left(-n\pi\right)}{2\pi n} -f\left(-n\pi\right),  \ \forall \xi \in \left[-n\pi, n\pi\right];\\
 f_n(\xi + 2\pi n) = f(\xi), \ \forall \xi \in \mathbb{R}.
 \end{split}
 \end{equation*}
 From periodicity of $f_n$ it follows that
 \begin{equation*}
 f_n = \sum_{q\in \mathbb{Q}}a^n_q e^{2\pi i q\xi}
 \end{equation*}
 and it is clear that $a^n_q \to 0$ as $n \to \infty$ for any $q \in \mathbb{Q}$.

\section{Spin manifolds and spectral triples}
 
 \paragraph{}
 This section contains citations of  \cite{hajac:toknotes}. 
 \subsection{Clifford algebras}
 \begin{empt}\label{sec:Cliff-alg}{\it Clifford algebras}. We start with $(V, g)$, where $V \isom \bR^n$ and $g$ is a
 	\emph{nondegenerate symmetric bilinear form}. If $q(v) = g(v,v)$,
 	then $2g(u, v) = q(u + v) - q(u) - q(v)$. Thus $g$ is determined by 
 	the corresponding ``quadratic form''~$q$.

 	\begin{defn}\cite{hajac:toknotes}
 		\label{df:Cliff-alg}
 		The \emph{\index{Clifford!algebra} Clifford algebra} $\Cl(V, g)$ is an algebra (over~$\bR$)
 		generated by the vectors $v\in V$ subject to the relations
 		$uv + vu = 2g(u, v)1$ for $u, v\in V$.
 	\end{defn}
 	
 	The existence of this algebra can be seen in two ways. First of all, 
 	let $\sT(V)$ be the tensor algebra on $V$, that is, 
 	$\sT(V) := \bigoplus_{k=0}^\infty V^{\ox n}$. Then
 	\begin{equation}
 	\Cl(V, g) := \sT(V)/\operatorname{Ideal}\langle
 	u \ox v + v \ox u - 2g(u,v)\,1 : u,v \in V\rangle.
 	\label{eq:Cliff-defn}
 	\end{equation}
 	Since the relations are \textit{not} homogeneous, the $\bZ$-grading of
 	$\sT(V)$ is lost, only a $\bZ_2$-grading remains:
 	\[
 	\Cl(V, g) = \Cl^0(V, g) \oplus \Cl^1(V, g).
 	\]
 	
 	The second option is to define $\Cl(V, g)$ as a subalgebra of
 	$\End_\bR(\La^\cz V)$ generated by all expressions
 	$c(v) = \eps(v) + \iota(v)$ for $v\in V$, where
 	\begin{align*}
 	\eps(v)\: u_1 \wyw u_k & \mapsto v \w u_1 \wyw u_k
 	\\
 	\iota(v)\: u_1 \wyw u_k & \mapsto \sum_{j=1}^k (-1)^{j-1} g(v,u_j)
 	u_1 \wyw \widehat{u_j} \wyw u_k.
 	\end{align*}
 	Note that $\eps(v)^2 = 0$, $\iota(v)^2 = 0$, and
 	$\eps(v)\iota(u) + \iota(u)\eps(v) = g(v,u)\,1$. Thus
 	\begin{align*}
 	c(v)^2 &= g(v, v)\,1  \words{for all}  v \in V,
 	\\
 	c(u)c(v) + c(v)c(u) &= 2g(u, v)\,1  \words{for all}  u,v \in V.
 	\end{align*}
 	Thus these operators on $\La^\cz V$ do provide a representation of the
 	algebra \eqref{eq:Cliff-defn}.
 	
 	Dimension count: suppose $\{e_1, \dots, e_n\}$ is an orthonormal basis
 	for $(V, g)$, i.e., $g(e_k, e_k) = \pm 1$ and $g(e_j, e_k) = 0$ for
 	$j \neq k$. Then the $c(e_j)$ anticommute and thus a basis for
 	$\Cl(V, g)$ is
 	$\{c(e_{k_1})\dots c(e_{k_r}) : 1 \leq k_1 <\cdots< k_r \leq n\}$,
 	labelled by $K = \{k_1,\dots,k_r\} \subseteq \{1,\dots,n\}$. Indeed,
 	\[
 	c(e_{k_1})\dots c(e_{k_r})\: 1 \mapsto e_{k_1} \wyw e_{k_r}
 	\equiv e_K \in \La^\cz V
 	\]
 	and these are linearly independent. Thus the dimension of the 
 	subalgebra of $\End_\bR(\La^\cz V)$ generated by all $c(v)$ is just
 	$\dim\La^\cz V = 2^n$. Now, a moment's thought shows that in the 
 	abstract presentation \eqref{eq:Cliff-defn}, the algebra $\Cl(V, g)$ 
 	is generated as a vector space by the $2^n$ products
 	$e_{k_1} e_{k_2}\dots e_{k_r}$, and these are linearly independent
 	since the operators $c(e_{k_1}) \dots c(e_{k_r})$ are linearly
 	independent in $\End_\bR(\La^\cz V)$. Therefore, this representation
 	of $\Cl(V, g)$ is faithful, and $\dim\Cl(V, g) = 2^n$.

 	The so-called ``symbol map'':
 	$$
 	\sg : a \mapsto a(1) : \Cl(V, g) \to \La^\cz V
 	$$
 	is inverted by a ``quantization map'': 
 	\begin{equation}
 	Q : u_1 \w u_2 \wyw u_r \longmapsto \frac{1}{r!} \sum_{\tau\in S_r}
 	(-1)^\tau \, c(u_{\tau(1)}) c(u_{\tau(2)}) \dots c(u_{\tau(r)}).
 	\label{eq:quantn-map}
 	\end{equation}
 	To see that it is an inverse to $\sg$, one only needs to check it
 	on the products of elements of an orthonormal basis of $(V, g)$. From now, we write $uv$ instead of $c(u)c(v)$, etc., in $\Cl(V, g)$.
 \end{empt}
 \begin{prop}\label{cl_trace_prop}\cite{hajac:toknotes}
There is an unique trace $\tau:\Cl(V, g) \to \C$ such that $\tau(1)=1$  and $\tau(a)=0$ for $a$ odd.
\end{prop}
There is a useful universality property of \index{Clifford!algebra} Clifford algebras, which is an immediate
consequence of their definition.

\begin{lem}
	\label{lm:Cliff-univ}\cite{hajac:toknotes}
	Any $\bR$-linear map $f\: V \to A$ (an $\bR$-algebra) that satisfies
	\[
	f(v)^2 = g(v,v)\,1_A \words{for all} v \in V
	\]
	extends to an unique unital $\bR$-algebra homomorphism
	$\tilde f\: \Cl(V, g) \to A$.
\end{lem}

Here are a few applications of universality that yield several
useful operations on the \index{Clifford!algebra} Clifford algebra.

\begin{enumerate}
	\item
	\textit{Grading}:
	take $A = \Cl(V, g)$ itself; the linear map $v \mapsto -v$ on~$V$
	extends to an \textit{automorphism} $\chi \in \Aut(\Cl(V,g))$
	satisfying $\chi^2 = \id_A$, given by
	\[
	\chi(v_1 \dots v_r) := (-1)^r \, v_1 \dots v_r.
	\]
	This operator gives the $\bZ_2$-grading
	\[
	\Cl(V, g) =: \Cl^0(V, g) \oplus \Cl^1(V, g).
	\]
	\item
	\textit{Reversal}:
	take $A = \Cl(V, g)^\opp$, the opposite algebra. Then the map
	$v \mapsto v$, considered as the inclusion $V \hookto A$, extends to
	an \textit{antiautomorphism} $a \mapsto a^!$ of $\Cl(V,g)$, given by
	$(v_1 v_2 \dots v_r)^! := v_r \dots v_2 v_1$.
	\item
	\textit{Complex conjugation}:
	the \emph{complexification} of $\Cl(V,g)$ is $\Cl(V, g) \ox_\bR \bC$,
	which is isomorphic to $\Cl(V^\bC, g^\bC)$ as a $\bC$-algebra. Now
	take $A$ to be $\Cl(V, g) \ox_\bR \bC$ and define
	$f\: v \mapsto \bar v : V^\bC \to V^\bC \hookto A$ (a real-linear map).
	It extends to an \textit{antilinear automorphism} of $A$. 
	\item
	\textit{Adjoint}:
	Also, $a^* := (\bar a)^!$ is an \textit{antilinear involution} on
	$\Cl(V, g) \ox_\bR \bC$.
	
	\item
	\textit{Charge conjugation}:
	$\ka(a) := \chi(\bar a) : v_1 \dots v_r
	\mapsto (-1)^r \bar v_1 \dots \bar v_r$ is an antilinear automorphism
	of $\Cl(V, g) \ox_\bR \bC$.
\end{enumerate}
{}From now on, $n = 2m$ for $n$ even, $n = 2m + 1$ for $n$ odd. We
take $\bCl(V) \isom \Cl(V,g) \ox_\bR \bC$ with $g$ always positive
definite. Suppose $\{e_1,\dots,e_n\}$ is an \textit{oriented} orthonormal
basis for $(V,g)$. If $e_k' = \sum_{j=1}^n h_{jk}\,e_j$ with
$H^t H = 1_n$, then $e_1'\dots e_n' = (\det H)\,e_1\dots e_n$, and
$\det H = \pm 1$. We restrict to the oriented case $\det H = +1$, so
the expression $e_1 e_2\dots e_n$ is independent of
$\{e_1,e_2,\dots,e_n\}$. Thus
\[
\ga := (-i)^m e_1 e_2\dots e_n
\]
is well-defined in $\bCl(V)$. Now
\[
\ga^* = i^m e_n\dots e_2 e_1 = (-i)^m(-1)^m(-1)^{n(n-1)/2}e_1 e_2\dots e_n = 
(-1)^m(-1)^{n(n-1)/2}\ga,
\]
and
$$
\frac{n(n-1)}{2} = \begin{cases}
m(2m - 1), & \text{$n$ even} \\ (2m + 1)m, & \text{$n$ odd}
\end{cases} \Biggr\} \equiv m \bmod 2,
$$
so $\ga^* = \ga$. But also
$\ga^*\ga = (e_n\dots e_2 e_1)(e_1 e_2\dots e_n) = (+1)^n = 1$, so
$\ga$ is ``unitary''. Hence $\ga^2 = 1$, so $\frac{1+\ga}{2}$,
$\frac{1-\ga}{2}$ are ``orthogonal projectors'' in $\bCl(V)$.

Since $\ga e_j = (-1)^{n-1} e_j\ga$, we get that if $n$ is odd, then
$\ga$ is \textit{central} in $\bCl(V)$; and for $n$ even, $\ga$
anticommutes with $V$, but is central in the even subalgebra
$\Cl^0(V)$. Moreover, when $n$ is even and $v \in V$, then
$\ga v \ga = -v$, so that $\ga(\cdot)\ga = \chi \in \Aut(\bCl(V))$.

\begin{prop}
	\label{pr:Cliff-centre}\cite{hajac:toknotes}
	The centre of $\bCl(V)$ is $\bC 1$ if $n$ is even; and it is
	$\bC 1 \oplus \bC\ga$ if $n$ is odd.
\end{prop}

	
	
	

 \subsection{\index{Clifford!algebra} Clifford algebra bundles}
 
\paragraph*{} A real vector bundle $E \to M$ is a \emph{Euclidean
 	bundle} if, with $\sE = \Ga(M, E^\bC)$, there is a symmetric
 $A$-bilinear form $g \: \sE \x \sE \to A = C(M)$ such that
 \begin{enumerate}
 	\item $g(s,t) \in C(M;\bR)$ when $s,t$ lie in $\Ga(M, E)$
 	---the real sections;
 	\item $g(s,s) \geq 0$ for $s \in \Ga(M, E)$, with
 	$g(s,s) = 0 \implies s = 0$.
 \end{enumerate}
 By defining $\pairing{s}{t} := g(s^*,t)$, we get a
 \emph{hermitian pairing} with values in $A$:
 \begin{itemize}
 	\item $\pairing{s}{t}$ is $A$-linear in $t$;
 	\item $\pairing{t}{s} = \overline{\pairing{s}{t}} \in A$;
 	\item $\pairing{s}{s} \geq  0$, with
 	$\pairing{s}{s} = 0 \implies s = 0$ in $\sE$;
 	\item $\pairing{s}{ta} = \pairing{s}{t}\,a$ for all $s,t \in \sE$
 	and $a \in A$.
 \end{itemize}
 These properties make $\sE$ a (right) $C^*$-module over $A$, with
 $C^*$-norm given by
 \[
 \|s\|_{\sE} := \sqrt{\|\pairing{s}{s}\|_A}  \word{for} s \in \sE.
 \]
 For each $x \in M$, we can form
 $\bCl(E_x) := \Cl(E_x,g_x) \ox_\bR \bC$. Using the linear isomorphisms
 $\sg_x \: \bCl(E_x) \to (\La^\cz E_x)^\bC$, we see that these are
 fibres of a vector bundle $\bCl(E) \to M$, isomorphic to
 $(\La^\cz E)^\bC \to M$ as $\bC$-vector bundles (but not as algebras!).
 Under $(\ka\la)(x) := \ka(x)\la(x)$, the sections of $\bCl(E)$ also
 form an algebra $\Ga(M, \bCl(E))$. It has an $A$-valued pairing
 \[
 \pairing{\ka}{\la} \: x \mapsto \tau(\ka(x)^* \la(x)).
 \]
 By defining $\|\ka\| := \sup_{x\in M} \|\ka(x)\|_{\bCl(E_x)}$, this
 becomes a C*-algebra.

 \begin{defn}\cite{hajac:toknotes}
 	A \emph{\index{Clifford!module} Clifford module} over $(M,g)$ is a finitely generated
 	projective $A$-module, with $A = C(M)$, of the form $\sE = \Ga(M, E)$
 	for $E$ a (complexified) Euclidean bundle, together with an $A$-linear
 	homomorphism $c \: B \to \Ga(M, \End E)$, where
 	$B := \Ga(M, \bCl(T^*M))$ is the \index{Clifford!algebra} Clifford algebra bundle generated by
 	$\sA^1(M)$, such that
 	\[
 	\pairing{s}{c(\ka)t} = \pairing{c(\ka^*)s}{t}
 	\words{for all} s,t \in \sE,\ \ka \in B.
 	\]
 \end{defn}


 \subsection{Riemannian geometry}
 \paragraph*{}
 Let $M$ be a \textit{compact} $\Coo$ manifold \textit{without
 	boundary}, of dimension $n$. Compactness is not crucial for some of
 our arguments (although it may be for others), but is very convenient,
 since it means that the algebras $C(M)$ and $\Coo(M)$ are
 \textit{unital}: the unit is the constant function~$1$. For
 convenience we use the function algebra $A = C(M)$ ---a commutative
 $C^*$-algebra--- at the beginning. We will change to $\sA = \Coo(M)$
 later, when the differential structure becomes important.
 Any $A$-module (or more precisely, a ``symmetric $A$-bimodule'') which
 is \textit{finitely generated and projective} is of the form
 $\sE = \Ga(M, E)$ for $E \to M$ a (complex) vector bundle. Two
 important cases are
 \begin{align*}
 \gX(M) &= \Ga(M, T_\bC M) = \text{ (continuous) vector fields on } M;
 \\
 \sA^1(M) &= \Ga(M, T^*_\bC M) = \text{ (continuous) 1-forms on } M.
 \end{align*}
 These are \textit{dual} to each other:
 $\sA^1(M) \isom \Hom_A(\gX(M), A)$, where $\Hom_A$ means ``$A$-module
 maps'' commuting with the action of~$A$ (by multiplication).

 \begin{defn}
 	\label{df:Riem-metric}\cite{hajac:toknotes}
 	A \emph{Riemannian metric} on $M$ is a symmetric bilinear form
 	\[
 	g \: \gX(M) \x \gX(M) \to C(M)
 	\]
 	such that:
 	\begin{enumerate}
 		\item $g(X, Y)$ is a real function if $X, Y$ are real vector fields;
 		\item $g$ is \emph{$C(M)$-bilinear}:\quad
 		$g(fX, Y) = g(X, fY) = f \, g(X, Y)$, \ if $f \in C(M)$;
 		\item $g(X, X) \geq $ for $X$ real, with $g(X, X) = 0\implies X = 0$
 		in $\gX(M)$.
 	\end{enumerate}
 	The second condition entails that $g$ is given by a continuous family
 	of symmetric bilinear maps $g_x \: T_x^\bC M \x T_x^\bC M \to \bC$ or
 	$g_x \: T_x M \x T_x M \to \bR$; the latter version is positive
 	definite.
 \end{defn}

 
 

 \subsection{The existence of $\Spin^{\mathbf{c}}$ structures}
 
 Suppose that $n = 2m + 1 = \dim M$ is odd. If $n$ is odd then the fibres of $B$ are
 semisimple but not simple:
 $\bCl(T_x^* M) \isom M_{2^m}(\bC) \oplus M_{2^m}(\bC)$ and we shall
 restrict to the even subalgebras, $\bCl^0(T_x^*M) \isom M_{2^m}(\bC)$,
 by demanding \textit{that $c(\ga)$ act as the identity in all cases}.
 Then we may adopt the convention that
 \[
 c(\ka) := c(\ka\ga) \quad\text{when $\ka$ is odd}.
 \]
 Notice here that $\ka\ga$ is even; and $c(\ga) = c(\ga^2) = +1$ is
 required for consistency of this rule.
 
 We take $A = C(M)$, but for $B$ we now take
 \begin{equation}
 B := \begin{cases}
 \Ga(M, \bCl(T^*M)),   & \text{if $\dim M$ is even}, \\
 \Ga(M, \bCl^0(T^*M)), & \text{if $\dim M$ is odd}.  \end{cases}
 \label{eq:full-even}
 \end{equation} 
 The fibres of these bundles are \textit{central simple algebras} of
 finite dimension $2^{2m}$ in all cases.
 
 We classify the algebras $B$ as follows. Taking
 \[
 \ul{B} := \begin{cases}
 \set{B_x = \bCl(T_x^* M) : x \in M},   & \text{if $\dim M$ is even} \\
 \set{B_x = \bCl^0(T_x^* M) : x \in M}, & \text{if $\dim M$ is odd}
 \end{cases}
 \]
 to be the collection of fibres, we can say that $\ul{B}$ is a
 ``continuous field of simple matrix algebras'', which moreover is
 \textit{locally trivial}. There is an invariant
 \[
 \delta(\ul{B}) \in \rH^3(M;\bZ)
 \]
 for such fields, found by Dixmier and Douady.  
 	The {\it Dixmier--Douady} class $\delta(\ul{B})$ is described in \cite{hajac:toknotes}. 
 Now we would like to find a bundle $S \to M$ such that there are natural isomorphisms
 \begin{equation}\label{spinor_morita}
 \End_A\left(\Ga\left(M,S\right)\right) \isom B;~ \End\left(\Ga\left(M,S\right)\right)_B \isom A.
 \end{equation}
 If $x \in M$, take $p_x \in B_x$ to be a \textit{projector of rank 
 	one}, that is,
 \[
 p_x = p_x^* = p_x^2  \words{and}  \tr p_x = 1.
 \]
 On the left ideal $S_x := B_x p_x$, we introduce a hermitian scalar
 product
 \begin{equation}
 \braket{a_x p_x}{b_x p_x} := \tr(p_x a_x^* b_x p_x ).
 \label{eq:local-pairing}
 \end{equation}
 Notice that the recipe
 \[
 \ketbra{a_x p_x}{b_x p_x} : c_x p_x \mapsto
 (a_x p_x)(b_x p_x)^*(c_x p_x) = (a_x p_x b_x^*)(c_x p_x)
 \]
 identifies $\sL(S_x)$ ---or $\sK(S_x)$ in the infinite-dimensional
 case--- with $B_x$, since the two-sided ideal
 $\spn\set{a_x p_x b_x^* : a_x, b_x \in B_x}$ equals $B_x$ by
 simplicity.

 \vspace{6pt}
 
Following proposition gives necessary and sufficient conductions  of the existence of the \textit{globally} defined  the Hilbert spaces $S_x \isom \bC^{2^m}$  such that
 $\sL(S_x) \isom B_x$, for any $x \in M$.

 

 \begin{prop}\cite{hajac:toknotes}
 	\label{pr:spinc-module}
 	Let $(M,g)$ be a compact Riemannian manifold. With $A = C(M)$ and $B$
 	the algebra of Clifford sections given by~\eqref{eq:full-even}, the
 	Dixmier--Douady class $\delta(\ul{B})$ vanishes, i.e., $\delta(\ul{B}) = 0$, if and
 	only if there is a finitely generated projective $A$-module $\sS$,
 	carrying a selfadjoint action of $B$ by $A$-linear operators, such
 	that $\End_A(S) \isom B$.
 \end{prop}

Let denote $\sS =\Ga\left(M,S\right)$ and $\sS^\sharp = \Hom_A(\sS, A)$. From the \eqref{spinor_morita} it follows that if $\sS^\sharp = \Hom_A(\sS, A)$ then $\sS \ox_A \sS^\sharp \isom B$ and
$\sS^\sharp \ox_B \sS \isom A$.
 Since $\sS^\sharp \isom \Ga(M,S^*)$ where $S^* \to M$ is the dual vector 
 bundle to $S \to M$, we can write this equivalence fibrewise:
 $\sS_x \ox_\bC \sS_x^* = \End_\bC(\sS_x) \isom B_x$ and then
 $\sS_x^* \ox_{B_x} \sS_x \isom \bC$, for $x \in M$. To proceed, we explain how $B$ acts on $\sS^\sharp = \Hom_A(\sS, A)$. The
 spinor module $\sS$ carries an $A$-valued Hermitian pairing
 \eqref{eq:local-pairing} given by the local scalar products defined in
 the construction of $\sS$, that may be written
 \begin{equation}
 \pairing{\psi}{\phi} \ : \ x \mapsto \langle\psi_x|\phi_x\rangle,
 \word{for} x \in M.
 \label{eq:spinor-pairing}
 \end{equation}
 We can identify elements of $\sS^\sharp$ with ``bra-vectors'' $\bra{\psi}$
 using this pairing, namely, we define $\bra{\psi}$ to be the map
 $\phi \mapsto \pairing{\psi}{\phi} \in A$. Since $A$ is unital, there
 is a ``Riesz theorem'' for $A$-modules showing that all elements of
 $\sS^\sharp$ are of this form. Now the left $B$-action is defined by
 \[
 b\,\bra{\psi} := \bra{\psi} \circ \chi(b^!).
 \]
 Recall that $b \mapsto \chi(b!)$ is a linear antiautomorphism of~$B$.
We also require triviality  of the  second Stiefel–Whitney class  $\ka(B)  = w_2(TM) = w_2(T^*M)$ of the tangent (or cotangent) bundle described in \cite{hajac:toknotes}. The condition $\ka(B) = 0$ is hold if
 \[
 \sS^\sharp \isom \sS  \quad\text{as $B$-$A$-bimodules}.
 \]
 We now reformulate this condition in terms of a certain antilinear 
 operator~$C$; later on, in the context of \index{spectral triple} spectral triples, we shall 
 rename it to~$J$.

 \begin{prop}
 	\label{pr:charge-conj}\cite{hajac:toknotes}
 	There is a $B$-$A$-bimodule isomorphism $\sS^\sharp \isom \sS$ if and only
 	if there is an \textit{antilinear} endomorphism $C$ of $\sS$ such that
 	\begin{enumerate}
 		\item[(a)]
 		$C(\psi\,a) = C(\psi) \,\bar a$\quad
 		for $\psi \in \sS$, $a \in A$;
 		\item[(b)]
 		$C(b\,\psi) = \chi(\bar b)\, C(\psi)$\quad
 		for $\psi \in \sS$, $b \in B$;
 		\item[(c)]
 		$C$ is \emph{antiunitary} in the sense that
 		$\pairing{C\phi}{C\psi} = \pairing{\psi}{\phi} \in A$,\quad
 		for $\phi, \psi \in \sS$;
 		\item[(d)]
 		$C^2 = \pm 1$ on $\sS$ whenever $M$ is connected.
 	\end{enumerate}
 \end{prop}

 The antilinear operator $C\: \sS \to \sS$, which becomes an
 \textit{antiunitary} operator on a suitable Hilbert-space completion
 of $\sS$, is called the \emph{charge conjugation}. It exists if and
 only if $\ka(B) = 0$.
 What, then, are $\mathrm{Spin}^{\mathbf{c}}$ and Spin structures on $M$? We choose on $M$
 a metric (without losing generality), and also an \textit{orientation}
 $\eps$, which organizes the action of $B$, in that a change
 $\eps \mapsto -\eps$ induces $c(\ga) \mapsto -c(\ga)$, which either
 \begin{enumerate}
 	\item[(i)] reverses the $\bZ_2$-grading of $\sS = \sS^+ \oplus \sS^-$,
 	in the even case; or
 	\item[(ii)] changes the action on $\sS$ of each $c(\al)$ to $-c(\al)$,
 	for $\al \in \sA^1(M)$, in the odd case ---recall that
 	$c(\al) := c(\al\ga)$ in the odd case.
 \end{enumerate}

 \begin{defn}
 	\label{df:spin-structure}\cite{hajac:toknotes}
 	Let $(M,\eps)$ be a compact boundaryless orientable manifold, 
 	together with a chosen orientation~$\eps$. Let $A = C(M)$ and let $B$ 
 	be specified as before (in terms of a fixed but arbitrary Riemannian 
 	metric on~$M$). If $\dl(B) = 0$ in $\rH^3(M;\bZ)$, a 
 	\emph{$\mathrm{Spin}^{\mathbf{c}}$ structure} on $(M,\eps)$ is an isomorphism
 	class $[\sS]$ of equivalence $B$-$A$-bimodules.
 	If $\dl(B) = 0$ and if $\ka(B) = 0$ in $\rH^2(M;\bZ_2)$, a pair
 	$(\sS,C)$ give data for a spin structure, when $\sS$ is an equivalence
 	$B$-$A$-bimodule such that $\sS^\sharp \isom \sS$, and $C$ is a charge
 	conjugation operator on~$\sS$. A \emph{Spin structure} on $(M,\eps)$
 	is an isomorphism class of such pairs. A vector bundle $S \to M$ such that $\sS = \Gamma(M, S)$ is said to be the \emph{spinor bundle}.
 \end{defn}

 \subsection{The Spin connection}
 
 \paragraph*{}We now leave the topological level and introduce differential
 structure. Thus we replace $A = C(M)$ by $\sA = \Coo(M)$, and
 continuous sections $\Ga_{\mathrm{cont}}$ by smooth sections
 $\Ga_{\mathrm{smooth}}$. Thus $\sS = \Ga_{\mathrm{smooth}}(M, S)$ will henceforth denote
 the $\sA$-module of \textit{smooth} spinors.
 
 Our treatment of \index{Morita equivalence} Morita equivalence of \emph{unital} algebras passes
 without change to the smooth level. We can go back with the functor
 $- \ox_{\Coo(M)} C(M)$, if desired.

 \begin{defn}\cite{hajac:toknotes}
 	A \emph{connection} on a (finitely generated projective)
 	$\sA$-module $\sE = \Ga(M, E)$ is a $\bC$-linear map
 	$\nb \: \sE \to \sA^1(M) \ox_{\sA} \sE = \Ga(M, T^*M \ox E)
 	\equiv \sA^1(M, E)$, satisfying the Leibniz rule
 	\[
 	\nb(fs) = df \ox s + f\,\nb s.
 	\]
 	It extends to an \textit{odd derivation} of degree +1 on
 	$\sA^\cz(M) \ox_{\sA} \sE = \Ga(M, \La^\cz T^*M\ox E) \equiv \sA^\cz(M,E)$
 	with grading
 	inherited from that of $\sA^\cz(M)$, leaving $\sE$ trivially graded, so
 	that $\nb(\om \w \sg) = d\om \w \sg + (-1)^{|\om|} \om \w \nb\sg$
 	for $\om \in \sA^\cz(M)$, $\sg \in \sA^\cz(M, E)$. 
 \end{defn}


 

 \begin{defn}
 	\label{df:Herm-conn}\cite{hajac:toknotes}
 	If $\sE$  an $\sA$-module equipped with an $\sA$-valued Hermitian 
 	pairing, we say that a \emph{connection} $\nb$ on~$\sE$ is 
 	\emph{Hermitian} if
 	\begin{align*}
 	\pairing{\nb s}{t} + \pairing{s}{\nb t} &= d\,\pairing{s}{t},
 	\word{or, in other words,}
 	\\
 	\pairing{\nb_X s}{t} + \pairing{s}{\nb_X t} &= X\,\pairing{s}{t},
 	\word{for any \emph{real}} X \in \gX(M). 
 	\end{align*}
 \end{defn}

 If $\nb, \nb'$ are connections on $\sE$, then $\nb' - \nb$ is an
 $\sA$-module map: $(\nb' - \nb)(fs) = f(\nb' - \nb)s$, so that
 \textit{locally}, over $U \subset M$ for which $E\bigr|_U \to U$ is
 trivial, we can write
 \[
 \nb = d + \al, \word{where} \al \in \sA^1(U, \End E).
 \]

 \begin{fact}\cite{hajac:toknotes}
 	On $\gX(M) = \Ga(M, TM)$ there is, for each Riemannian metric $g$,
 	a \emph{unique torsion-free connection that is compatible with $g$}:
 	\[
 	g(\nb X, Y) + g(X, \nb Y) = d(g(X,Y)) \word{for} X, Y \in \gX(M),
 	\word{or}
 	\]
 	\[
 	g(\nb_Z X, Y) + g(X, \nb_Z Y) = Z(g(X,Y)) \word{for} X, Y, Z \in \gX(M).
 	\]
 \end{fact}

 The explicit formula for this connection is
 \begin{align}
 2g(\nb_X Y, Z) &= X(g(Y, Z)) + Y(g(X, Z)) - Z(g(X, Y)) 
 \nonumber \\
 &\qquad + g(Y, [Z,X]) + g(Z, [X,Y]) - g(X, [Y,Z]).
 \label{eq:Levi-Civita}
 \end{align}
 It is called \emph{Levi-Civita connection} associated to~$g$. (The 
 proof of existence consists in showing that the right hand side of 
 this expression is $\sA$-linear in $Y$ and~$Z$, and obeys a Leibniz 
 rule with respect to~$X$, so it gives a connection; and uniqueness 
 is obtained by checking that metric compatibility and torsion freedom 
 make the right hand side automatic.)
 
 The \emph{dual connection} on $\sA^1(M)$ will also be called the
 ``Levi-Civita connection''. At the risk of some confusion, we shall 
 use the same symbol $\nb$ for both of these Levi-Civita connections.

 \begin{defn}
 	\label{df:spin-conn}\cite{hajac:toknotes}
 	On a spinor module $\sS = \Ga(M, S)$, a
 	\emph{spin$^{\mathbf{c}}$-connection} is any \emph{Hermitian}
 	connection $\nb^S \: \sS \to \sA^1(M) \ox_{\sA} \sS$ which is
 	compatible with the action of~$B$ in the following way:
 	\begin{align}
 	\nb^S(c(\al)\psi) &= c(\nb\al) \psi + c(\al) \nb^S\psi
 	\word{for} \al \in \sA^1(M), \ \psi \in \sS; \word{or}
 	\nonumber \\
 	\nb_X^S(c(\al)\psi) &= c(\nb_X\al) \psi + c(\al) \nb_X^S\psi
 	\word{for} \al \in \sA^1(M), \ \psi \in \sS, \ X \in \gX(M),
 	\label{eq:spin-conn}
 	\end{align}
 	where $\nb \al$ and $\nb_X \al$ refer to the \emph{Levi-Civita
 		connection} on $\sA^1(M)$.
 	
 	If $(\sS, C)$ are data for a spin structure, we say $\nb^S$ is a
 	\emph{spin connection} if, moreover, each $\nb_X \: \sS \to \sS$
 	commutes with $C$ whenever $X$ is real.
 \end{defn}

 \[
 \nb^S(c(\al)\psi) =  c(\nb \al) \psi + c(\al) \nb^S\psi.
 \]

 \begin{prop}\cite{hajac:toknotes}
 	If $(\sS, C)$ are data for a spin structure on $M$, then there is a unique
 	Hermitian spin connection $\nb^S \: \sS \to \sA^1(M) \ox_{\sA} \sS$,
 	such that
 	\[
 	\nb^S(c(\al)\psi) = c(\nb\al) \psi + c(\al) \nb^S\psi,
 	\word{for} \al \in \sA^1(M),\ \psi \in \sS,
 	\]
 	and such that $\nb^S_X C = C \nb^S_X$ for $X \in \gX(M)$ real. 
 \end{prop}

 \subsection{Dirac operators}\label{comm_dir_op}
 \label{ch:DOSG_3}
 
\paragraph*{} Suppose we are given a compact oriented (boundaryless) Riemannian
 manifold $(M,\eps)$ and a spinor module with charge conjugation
 $(\sS,C)$, together with a Riemannian metric $g$, so that the Clifford
 action $c \: \sB \to \End_{\sA}(\sS)$ has been specified. We can also
 write it as $\hat c \in \Hom_{\sA}(\sB \ox_{\sA} \sS, \sS)$ by setting
 $\hat c(\ka \ox \psi) := c(\ka)\,\psi$.

 \begin{defn}
 	\label{df:Dirac-defn}\cite{hajac:toknotes}
 	Using the inclusion $\sA^1(M) \hookto \sB$ ---where in the odd
 	dimensional case this is given by $c(\al) := c(\al\ga)$, as before---
 	we can form the composition
 	\begin{equation}
 	\Dslash := -i\,\hat c \circ \nb^S
 	\label{eq:Dirac-defn}
 	\end{equation}
 	where
 	\[
 	\sS \xrightarrow{\nb^S} \sA^1(M)\ox_{\sA}\sS \xrightarrow{\hat c} \sS,
 	\]
 	so that $\Dslash \: \sS \to \sS$ is $\bC$-linear. This is the
 	\emph{Dirac operator} associated to $(\sS, C)$ and~$g$.
 \end{defn}

 The $(-i)$ is included in the definition to make $\Dslash$ symmetric
 (instead of skewsymmetric) as an operator on a Hilbert space, because
 we have chosen $g$ to be positive definite, that is,
 $\ga^\al \ga^\bt + \ga^\bt \ga^\al = +2\,\dl^{\al\bt}$. Historically,
 $\Dslash$ was introduced as $-i\ga^\mu \dl_\mu = \ga^\mu p_\mu$ where 
 the $p_\mu$ are components of a $4$-momentum, but in the Minkowskian
 signature.
 
 Using local (coordinate or orthonormal) bases for $\gX(M)$ and
 $\sA^1(M)$, we get nicer formulas:
 \begin{equation}
 \Dslash\psi = -i\,\hat c(\nb^S\psi) = -i\,c(dx^j) \nb^S_{\del_j} \psi
 = -i\,\ga^\al \nb^S_{E_\al} \psi.
 \label{eq:Dirac-local}
 \end{equation}
 The essential algebraic property of $\Dslash$ is the
 \textit{commutation relation}:
 \begin{equation}
 [\Dslash, a] = -i\,c(da), \words{for all} a \in \sA = \Coo(M).
 \label{eq:comm-Dirac}
 \end{equation}
 Indeed,
 \begin{align*}
 [\Dslash, a]\,\psi
 &= -i\,\hat c(\nb^S(a\psi)) + ia\,\hat c(\nb^S\psi)
 \\
 &= -i\,\hat c(\nb^S(a\psi) - a\,\nb^S\psi)
 \\
 &= -i\,\hat c(da \ox\psi) = -i\,c(da)\,\psi, \word{for} \psi \in \sS.
 \end{align*}
 As an operator, we can make sense of $[\Dslash, a]$ by conferring on
 $\sS$ the structure of a Hilbert space: if we write
 $\det g := \det[g_{ij}]$ for short, then
 \[
 \nu_g := \sqrt{\det g} \,dx^1 \w dx^2 \wyw dx^n \in \sA^n(M)
 \]
 is the Riemannian volume form (for the given orientation $\eps$ and
 metric~$g$). In the notation, we assume that all local charts are
 consistent with the given orientation, which just means that
 $\det[g_{ij}] > 0$ in any local chart. The scalar product on $\sS$ is
 then given by
 \begin{equation}\label{comm_hilb_space}
 (\phi,\psi) \stackrel{\mathrm{def}}{=} \int_M \pairing{\phi}{\psi} \,\nu_g
 \word{for} \phi,\psi \in \sS.
 \end{equation}
 On completion in the norm $\|\psi\| \stackrel{\mathrm{def}}{=} \sqrt{\left(\psi,\psi\right)}$, we
 get the Hilbert space $\H := L^2(M,S)$ of $L^2$-spinors on~$M$. In the even case, $B = \Ga(M,\bCl(T^*M))$ contains the operator $\Ga = c(\ga)$ which
 extends to a selfadjoint unitary operator on $H$. It is known that $\left(C^{\infty}(M), L^2(M, S),\slashed D, C, \Ga\right)$ (resp. $\left(C^{\infty}(M), L^2(M, S),\slashed D, C \right)$) is a spectral triple in case of even (resp. odd) dimension \cite{hajac:toknotes}.

 \subsection{Definition of spectral triples}
  \begin{defn}
  \label{df:spec-triple}\cite{hajac:toknotes}
  A (unital) {\bf \index{spectral triple} spectral triple} $(\mathscr{A}, H, D)$ consists of:
  \begin{itemize}
  \item
  an \emph{algebra} $\mathscr{A}$ with an involution $a \mapsto a^*$, equipped
  with a faithful representation on:
  \item
  a \emph{Hilbert space} $\H$; and also
  \item
  a \emph{selfadjoint operator} $D$ on $\mathcal{H}$, with dense domain
  $\Dom D \subset \H$, such that $a(\Dom D) \subseteq \Dom D$ for all 
  $a \in \mathcal{A}$,
  \end{itemize}
  satisfying the following two conditions:
  \begin{itemize}
  \item
  the operator $[D,a]$, defined initially on $\Dom D$, extends to a
  \emph{bounded operator} on $\H$, for each $a \in \mathcal{A}$;
  \item
  $D$ has \emph{compact resolvent}: $(D - \lambda)^{-1}$ is compact, when
  $\lambda \notin \spec(D)$.
  \end{itemize}
  \end{defn}

  For now, and until further notice, all \index{spectral triple} spectral triples will be
  defined over unital algebras. The compact-resolvent condition must be 
  modified if $\mathcal{A}$ is nonunital: as well as enlarging $\mathcal{A}$ to a unital 
  algebra, we require only that the products $a(D - \lambda)^{-1}$, for 
  $a \in \mathcal{A}$ and $\lambda \notin \spec(D)$, be compact operators.
  
  \subsection{The Dixmier trace}
  \paragraph{} The Dixmier trace is the noncommutative analogue of integral over a manifold.
 \begin{empt}\cite{varilly:noncom} The algebra $\mathcal{K}$ of compact operators on a separable, infnite-dimensional Hilbert space   contains the ideal $\mathcal{L}^1$ of traceclass operators, on which  $\|T\|_1 = \mathrm{Tr}|T|$ is a norm not to  be confused with the operator norm $\|T\|$. Let $\sigma_n(T)$ be such that 
   \begin{equation*}
  \sigma_n(T)= \sup\left\{\|TP_n\|_1 \ | \ P_n \ \text{is a projector of rank} \ n \right\}
  \end{equation*}
  
  There is a formula \cite{connes_moscovici:local_index}, coming from real  interpolation theory of Banach spaces:
  \begin{equation*}
  \sigma_n(T)= \left\{\inf \{\|R\|_1+n\|S\| \ | \ R,S \in \mathcal{K}, \ R + S = T \right\}.
  \end{equation*}
  If $T \in \mathcal{K}$ is a compact operator then $\sigma_n$ can be defined as
  \begin{equation*}
  \sigma_n = \sum_{i =1}^n \lambda_i
  \end{equation*}
  where $\{\lambda_i\}_{i \in N}$ is a decreasing ordered set of the operator $\left(T^*T\right)^{1/2}$ eigenvalues, i.e. $\lambda_1 \ge \lambda_2 \ge ... \ge \lambda_n \ge ...$.   We can think of $\sigma_n(T)$ as the trace of $|T|$ with a {\it cutoff} at the scale $n$.  This scale does not have to be an integer; for any scale $ \lambda > 0$, we can {\it define}
  \begin{equation*}
  \sigma_{\lambda}(T) = \inf\left\{\|R\|_1+\lambda \|S\| \ | \ R,S \in \mathcal{K}, \ R + S = T \right\}.
 \end{equation*}
 If $0 < \lambda \le 1$, then $\sigma_{\lambda}(T) = \lambda\|T\|$. If $\lambda = n + t$ with $0 \le t < 1$, one checks that
  \begin{equation}\label{sltn}
 \sigma_{\lambda}(T)=(1-t)\sigma(T)+t \sigma_{n+1}(T),
 \end{equation}
  so $\lambda \mapsto \sigma_{\lambda}(T)$ is a piecewise linear, increasing, concave function on $(0,1)$.
  
  \paragraph*{} Each $\sigma_{\lambda}$ is a norm by \eqref{sltn}, and so satisfies the triangle inequality. It is proved in \cite{varilly:noncom} that for positive compact operators, there is a triangle inequality in the opposite direction:
 \begin{equation}\label{sigma_uneq}
  \sigma_{\lambda}(A) + \sigma_{\mu}(B) \le \sigma_{\lambda + \mu}(A + B); \ \text{if} \ A,B > 0.
 \end{equation}
  It suffices to check this for integral values $\lambda=m$, $\mu =n$. If $P_m$, $P_n$ are projectors of respective ranks $m$, $n$, and if $P = P_m \vee P_n$ is the projector with range $P_mH + P_nH$, then
  \begin{equation*}
  \|AP_m\|_1 + \|BP_n\|_1 = Tr(P_mAP_m) + Tr(P_nBP_n) \le Tr(P(A + B)P) \le \|(A + B)P\|_1,
 \end{equation*}
  and \eqref{sigma_uneq} follows by taking supremum over $P_m$, $P_n$. Thus we have a sandwich of norms:
  \begin{equation}\label{norm_sandwich}
  \sigma_{\lambda}(A + B) \le \sigma_{\lambda}(A) + \sigma_{\lambda}(B) \le \sigma_{2\lambda}(A + B) \ \text{if} \ A,B \ge 0.
  \end{equation}
  \end{empt}
  \begin{empt}\cite{varilly:noncom} {\it The Dixmier ideal}. The {\itfirst-order infinitesimals} can now be defined precisely as the following normed ideal of compact operators:
  \begin{equation*}
  \mathcal{L}^{1+}=\left\{T \in \mathcal{K} \ | \ \|T\|_{1+}=\sup_{\lambda > e} \frac{\sigma_{\lambda}(T)}{\log\lambda}< \infty\right\},
  \end{equation*}
  
  that obviously includes the traceclass operators $\mathcal{L}^1$. (On the other hand, if $p > 1$ we have $\mathcal{L}^{1+}\subset\mathcal{L}^p$, where the latter is the ideal of those $T$ such that $\mathrm{Tr}|T|^p < 1$, for which $\sigma_{\lambda}(T) = O(\lambda^{1-1/p})$.) If $T \in \mathcal{L}^{1+}$, the function $\lambda \mapsto \sigma_{\lambda}(T)/\log\lambda$ is continuous and bounded on the interval  $[e,\infty)$, i.e., it lies in the $C^*$-algebra $C_b[e,\infty)$. We can then form the following Ces\`{a}ro mean of this function:
 \begin{equation*}
 \tau_{\lambda}(T)=\frac{1}{\log \lambda}\int_{e}^{\lambda}\frac{\sigma_u(T)}{\log u}\frac{du}{u}.
  \end{equation*}
  
  Then $\lambda \mapsto \tau_{\lambda}(T)$ lies in $C_b[e, \infty)$ also, with upper bound $\|T\|_{1+}$. From \eqref{norm_sandwich} we can derive   that
  \begin{equation*}
  \tau_{\lambda}(A)+\tau_{\lambda}(B)-\tau_{\lambda}(A+B)\le \left(\|A\|_{1+}+\|B\|_{1+}\right)\log 2 \frac{\log \log \lambda}{\log \lambda},
  \end{equation*}
  so that $\tau_{\lambda}$ is "asymptotically additive" on positive elements of $\mathcal{L}^{1+}$.
  \paragraph{} We get a true additive functional in two more steps. Firstly, let $\dot \tau(A)$ be the class of $\lambda \mapsto \tau_{\lambda}(A)$ in the quotient $C^*$-algebra $\mathcal{B} = C_b[e, \infty)/C_0[e, \infty)$. Then $\dot \tau$ is an additive,   positive-homogeneous map from the positive cone of $\mathcal{L}^{1+}$ into $\mathcal{B}$, and $\dot \tau(UAU^{-1}) = \dot \tau(A)$ for  any unitary $U$; therefore it extends to a linear map $\dot \tau: \mathcal{L}^{1+} \to \mathcal{B}$ such that $\dot \tau (ST) =\dot \tau (TS)$ for $T \in \mathcal{L}^{1+}$ and any $S$.
  \paragraph{} Secondly, we follow $\dot \tau$ with any state (i.e., normalized positive linear form) $\omega:\mathcal{B} \to \mathbb{C}$.  The composition is a {\it Dixmier trace}:
  \begin{equation*}
  \mathrm{Tr}_{\omega}(T) = \omega(\dot \tau (T)).
  \end{equation*}
 \end{empt}
 \begin{empt}{\it The noncommutative integral.} Unfortunately, the $C^*$-algebra $\mathcal{B}$ is not separable and there is no way to {\it exhibit} any particular state. This problem can be finessed by noticing that a function $f\in C_b[e, \infty)$ has a limit $\lim_{\lambda \to \infty} f(\lambda) = c$ if and only if $\omega(f) = c$ does not  depend on $\omega$. Let us say that an operator $T\in \mathcal{L}^{1+}$ is {\it measurable} if the function $\lambda \mapsto \tau_{\lambda}(T)$ converges as $\lambda \to \infty$, in which case any $\mathrm{Tr}_{\omega}(T)$ equals its limit. We denote by $\ncint T$ the common value of the Dixmier traces:
  \begin{equation*}
  \ncint T = \lim_{\lambda \to \infty} \tau_{\lambda}(T) \ \text{if this limit exists}.
  \end{equation*}
  
  We call this value the {\it noncommutative integral} of $T$.
  \paragraph{}Note that if $T\in \mathcal{K}$ and $\sigma_n(T)/ \log n$ converges as $n \to \infty$, then $T$ lies in $\mathcal{L}^{1+}$ and is   measurable. 
  \end{empt}

  \begin{exm}\label{comm_integ_exm}{\it Commutative case}.
  Let $M$ be a compact spin-manifold, and let $g$ be the Riemannian metric \cite{varilly:noncom}. There is the Riemannian volume form $\Omega$ given by
  \begin{equation*}
  \Omega = \sqrt{\mathrm{det}g(x)}dx^1 \wedge...\wedge dx^n.
  \end{equation*}
  It is proven in \cite{varilly:noncom} that for any $a \in C(M)$ following equation hold
  
  \[
      \int_{M} a \Omega=
  \begin{cases}
      m!(2\pi)^m \ncint a \slashed D^{-2m} ,& \text{if} \ \mathrm{dim}M = 2m \ \text{is even},\\
      (2m+1)!!\pi^{m+1} \ncint a |\slashed D|^{-2m-1}             &  \text{if} \ \mathrm{dim}M = 2m + 1 \ \text{is odd}.
  \end{cases}
  \]

  \end{exm}

 \subsection{Regularity of  spectral triples}
   
   \paragraph{} The arguments of the previous section are not applicable to determine
   whether $[|D|, a]$ is bounded, in the case $r = 1$. This must be
   formulated as an assumption. In fact, we shall ask for much more:
   we want each element $a \in \sA$, and each bounded operator $[D,a]$ 
   too, to lie in the {\it smooth domain} of the following derivation.

   \begin{notn}
   We denote by $\dl$ the derivation on $B(\sH)$ given by taking the 
   commutator with $|D|$. It is an unbounded derivation, whose domain is
   \[
   \Dom \dl := \set{T \in B(\sH) :
   T(\Dom|D|) \subseteq \Dom|D|,\ [|D|, T] \text{ is bounded }}.
   \]
   We write $\dl(T) := [|D|, T]$ for $T \in \Dom\dl$.
   \end{notn}

   \begin{defn}\cite{hajac:toknotes}
   \label{df:spt-regular}
   A \index{spectral triple} spectral triple $(\sA, \sH, D)$ is called {\it regular}, if for
   each $a \in \sA$, the operators $a$ and $[D,a]$ lie in
   $\bigcap_{k\in\bN} \Dom \dl^k$.
   \end{defn}

    \begin{cor}\cite{hajac:toknotes}
   \label{cr:comm-spt}
   The standard commutative example $(\Coo(M), L^2(M, S), \Dslash)$
   is a regular \index{spectral triple} spectral triple.
   \end{cor}

   \subsection{Pre-C*-algebras}
   
  \paragraph*{} If any spectral triple $(\sA, \sH, D)$, the algebra $\sA$ is a
   (unital) $*$-algebra of bounded operators acting on a Hilbert space
   $\sH$ [or, if one wishes to regard $\sA$ abstractly, a faithful
   representation $\pi \: \sA \to B(\sH)$ is given]. Let $A$ be the
   norm closure of $\sA$ [or of $\pi(\sA)$] in $B(\sH)$: it is a
   C*-algebra in which $\sA$ is a dense $*$-subalgebra.
   
   {\it A priori}, the only functional calculus available for $\sA$
   is the holomorphic one:
   \begin{equation}
   f(a) := \inv{2\pi i} \oint_\Ga f(\la) (\la 1 - a)^{-1} \,d\la,
   \label{eq:Dunford-int}
   \end{equation} 
   where $\Ga$ is a contour in $\bC$ winding (once positively) around
   $\spec(a)$, and $\spec(a)$ means the spectrum of $a$ in the
   C*-algebra $A$. To ensure that $a \in \sA$ implies $f(a) \in \sA$,
   we need the following property:
   
   If $a \in \sA$ has an inverse $a^{-1} \in A$, then in fact $a^{-1}$
   lies in $\sA$ (briefly: $\sA \cap A^\x = \sA^\x$, where $\sA^\x$ is
   the group of invertible elements of~$A$). If this condition holds,
   then $\inv{2\pi i} \oint_\Ga f(\la) (\la 1 - a)^{-1} \,d\la$ is a
   limit of Riemannian sums lying in $\sA$. To ensure convergence in~$\sA$
   (they do converge in~$A$), we need only ask that $\sA$ be complete in
   some topology that is finer than the $C^*$-norm topology.

   \begin{defn}
   \label{df:pre-Cstar}\cite{hajac:toknotes}
   A {\boldmath{\it pre-$C^*$-algebra}} is a subalgebra of $\sA$ of a
   C*-algebra $A$, which is stable under the holomorphic functional
   calculus of~$A$.
   \end{defn}

   If $\sA$ is a nonunital algebra, we can always adjoin a unit in the
   usual way, and work with $\Tilde\sA := \bC \oplus \sA$ whose unit is
   $(1,0)$, and with its $C^*$-completion $\Tilde A := \bC \oplus A$.
   Since the multiplication rule in $\Tilde\sA$ is
   $(\la,a)(\mu,b) := (\la\mu, \mu a + \la b + ab)$, we see that
   $1 + a := (1,a)$ is invertible in $\Tilde\sA$, with inverse $(1,b)$,
   if and only if $a + b + ab = 0$.

   \begin{lem}\cite{hajac:toknotes}
   \label{lm:preC-matrix}
   If $\sA$ is a unital, Fr\'echet pre-$C^*$-algebra, then so also is
   $M_n(\sA) = M_n(\bC) \ox \sA$.
   \end{lem}
   
   \begin{lem}
   \label{lm:sS-preC}\cite{hajac:toknotes}
   The Schwartz algebra $\sS(\bR^n)$ is a nonunital pre-$C^*$-algebra.
   \end{lem}

  We state, without proof, two important facts about Fr\'echet
   pre-$C^*$-algebras.

   \begin{fact}\cite{hajac:toknotes}
   If $\sA$ is a Fr\'echet pre-$C^*$-algebra and $A$ is
   its $C^*$-completion, then $\rK_j(\sA) = \rK_j(A)$ for $j = 0,1$. More 
   precisely, if $i\: \sA \to A$ is the (continuous, dense) inclusion,
   then $i_* \: \rK_j(\sA) \to \rK_j(A)$ is an surjective isomorphism,
   for $j = 0$ or~$1$.
   \end{fact}

   This invariance of $\rK$-theory was proved by Bost~\cite{bost:oka}.
   For~$\rK_0$, the spectral invariance plays the main role. For~$\rK_1$,
   one must first formulate a topological $\rK_1$-theory is a category of
   ``good'' locally convex algebras (thus whose invertible elements form 
   an open subset and for which inversion is continuous), and it is 
   known that Fr\'echet pre-$C^*$-algebras are ``good'' in this sense.

   \begin{fact}\cite{hajac:toknotes}
   If $(\sA, \sH, D)$ is a regular  spectral triple, we can confer on
   $\sA$ the topology given by the seminorms
   \begin{equation}
   q_k(a) := \|\dl^k(a)\|, \quad  q_k'(a) := \|\dl^k([D,a])\|,
   \word{for each} k \in \bN.
   \label{eq:delta-top}
   \end{equation}
   The completion $\sA_{\dl}$ of $\sA$ is then a Fr\'echet
   pre-$C^*$-algebra, and $(\sA_{\dl}, \sH, D)$ is again a regular spectral
   triple.
   \end{fact}

   These properties of the completed  spectral triple are due to 
   Rennie~\cite{rennie:smooth_nonunital}. We now discuss another result of Rennie, namely 
   that such completed algebras of regular spectral triples are endowed 
   with a {\it $\Coo$ functional calculus}.

   \begin{prop}\cite{hajac:toknotes}
   \label{pr:Coo-funccalc}
   If $(\sA, \sH, D)$ is a regular spectral triple, for which $\sA$ is
   complete in the Fr\'echet topology determined by the
   seminorms~\eqref{eq:delta-top}, then $\sA$ admits a $\Coo$-functional
   calculus. Namely, if $a = a^* \in \sA$, and if $f\: \bR \to \bC$ is a
   compactly supported smooth function whose support includes a
   neighbourhood of $\spec(a)$, then the following element $f(a)$ lies
   in~$\sA$:
   \begin{equation}
   f(a) := \inv{2\pi} \int_{\bR} \hat f(s) \exp(isa) \,ds.
   \label{eq:Coo-funccalc}
   \end{equation}
   \end{prop}
   
   
 Before showing how this smooth functional calculus can yield useful
   results, we pause for a couple of technical lemmas on approximation of
   idempotents and projectors, in Fr\'echet pre-$C^*$-algebras. The first
   is an adaptation of a proposition of~\cite{bost:oka}

   \begin{lem}
   \label{lm:almost-idemp}\cite{hajac:toknotes}
   Let $\sA$ be an unital Fr\'echet pre-$C^*$-algebra, with $C^*$-norm
   $\|\cdot\|$. Then for each $\eps$ with $0 < \eps < \frac{1}{8}$, we
   can find $\dl \leq \eps$ such that, for each $v \in \sA$ with
   $\|v - v^2\| < \dl$ and $\|1 - 2v\| < 1 + \dl$, there is an idempotent
   $e = e^2 \in \sA$ such that $\|e - v\| < \eps$.
   \end{lem}
   
  Lemma~\ref{lm:almost-idemp} says that in a unital Fr\'echet
   pre-$C^*$-algebra $\sA$, an ``almost idempotent'' $v \in \sA$ that is
   not far from being a projector (since $\|1 - 2v\|$ is close to~$1$)
   can be retracted to a genuine idempotent in~$\sA$. The next Lemma says
   that projectors in the $C^*$-completion of~$\sA$ can be approximated
   by projectors lying in~$\sA$.

   \begin{lem}
   \label{lm:proj-approx}\cite{hajac:toknotes}
   Let $\sA$ be an unital Fr\'echet pre-$C^*$-algebra, whose
   $C^*$-completion is~$A$. If $\tilde{q} = \tilde{q}^2 = \tilde{q}^*$ is
   a projector in $A$, then for any $\eps > 0$, we can find a projector
   $q = q^2 = q^* \in \sA$ such that $\|q - \tilde{q}\| < \eps$.
   \end{lem}

 \subsection{Real spectral triples}
   
   Recall that a spin structure on an oriented compact manifold
   $(M,\eps)$ is represented by a pair $(\sS,C)$, where $\sS$ is a
   $\sB$-$\sA$-bimodule and $C \: \sS\to \sS$ is an antilinear
   map such that $C(\psi\,a) = C(\psi)\,\bar a$ for $a \in \sA$;
   $C(b\,\psi) = \chi(\bar b)\,C(\psi)$ for $b \in \sB$; and, by choosing
   a metric $g$ on $M$, which determines a Hermitian pairing on~$\sS$, we
   can also require that
   $(C\phi|C\psi) = (\phi|\psi) \in \sA$ for
   $\phi,\psi \in \sS$. $\sS$ may be completed to a Hilbert space
   $\sH = L^2(M,S)$, with scalar product
   $\langle\phi|\psi\rangle = \int_M(\phi|\psi),\nu_g$. It is clear
   that $C$ extends to a bounded antilinear operator on $\sH$ such that
   $\braket{C\phi}{C\psi} = \braket{\psi}{\phi}$ by integration with
   respect to $\nu_g$, so that (the extended version of) $C$ is
   \textit{antiunitary} on $\sH$. There are two tables of signs
   
     \[
     \begin{array}[t]{|c|cccc|}
     \hline
     n \bmod 8               & 0 & 2 & 4 & 6 \rule[-5pt]{0pt}{17pt} \\
     \hline
     C^2 = \pm 1             & + & - & - & + \rule[-5pt]{0pt}{17pt} \\
     C\Dslash = \pm\Dslash C & + & + & + & + \rule[-5pt]{0pt}{17pt} \\
     C\Ga = \pm\Ga C         & + & - & + & - \rule[-5pt]{0pt}{17pt} \\
     \hline
     \end{array}
     \qquad\qquad
     \begin{array}[t]{|c|cccc|}
     \hline
     n \bmod 8               & 1 & 3 & 5 & 7 \rule[-5pt]{0pt}{17pt} \\
     \hline
     C^2 = \pm 1             & + & - & - & + \rule[-5pt]{0pt}{17pt} \\
     C\Dslash = \pm\Dslash C & - & + & - & + \rule[-5pt]{0pt}{17pt} \\
     \hline
     \end{array}
     \]
     
     \vspace{6pt}
   
   where $n$ is a dimension of spin manifold (See \cite{hajac:toknotes} for details).
   
   \begin{defn}
   \label{df:spt-real}\cite{hajac:toknotes}
   A {\it real  spectral triple} is a  spectral triple $(\sA, \sH, D)$,
   together with an antiunitary operator $J\:\sH\to\sH$ such that
   $J(\Dom D)\subset \Dom D$, and $[a, Jb^*J^{-1}] = 0$ for all
   $a, b \in \sA$.
   \end{defn}

   \begin{defn}
   \label{df:spt-even}\cite{hajac:toknotes}
   A  spectral triple $(\sA, \sH, D)$ is {\it even} if there is
   a selfadjoint unitary operator $\Ga$ on $\sH$ such that $a\Ga = \Ga a$
   for all $a \in \sA$, $\Ga(\Dom D) = \Dom D$, and $D\Ga = -\Ga D$. If
   no such $\bZ_2$-grading operator $\Ga$ is given, we say that
   the spectral triple is {\it odd}.
   \end{defn}

   We have seen that in the standard commutative example, the {\it even}
   case arises when the auxiliary algebra $\sB$ contains a natural
   $\bZ_2$-grading operator, and this happens exactly when {\it the
   manifold dimension is even}. Now, the manifold dimension is determined
   by the spectral growth of the Dirac operator, and this spectral
   version of dimension may be used for noncommutative  spectral triples,
   too. To make this more precise, we must look more closely at spectral
   growth.
   
 \subsection{Geometric conditions on \index{spectral triple} spectral triples}
 \label{sec:geo-cond}
 
 We begin by listing a set of requirements on a \index{spectral triple} spectral triple
 $(\sA, \sH, D)$, whose algebra $\sA$ is unital but not necessarily
 commutative, such that $(\sA, \sH, D)$ provides a ``spin geometry''
 generalization of our ``standard commutative example''
 $(\Coo(M), L^2(M,S), \Dslash)$. Again we shall assume, for 
 convenience, that $D$ is invertible.

 \begin{cond}[Spectral dimension]
 There is an \emph{integer} $n \in \{1,2,\dots\}$, called the spectral
 dimension of $(\sA, \sH, D)$, such that $|D|^{-1} \in \sL^{n+}(\sH)$,
 and $0 < \mathrm{Tr}_{\om}(|D|^{-n}) < \infty$ for any Dixmier trace~$\mathrm{Tr}_{\om}$.
 
 When $n$ is \emph{even}, the \index{spectral triple} spectral triple $(\sA, \sH, D)$ is also 
 even: that is, there exists a selfadjoint unitary operator
 $\Ga \in B(\sH)$ such that $\Ga(\Dom D) = \Dom D$, satisfying
 $a\Ga = \Ga a$ for all $a \in \sA$, and $D\Ga = -\Ga D$.
 \end{cond}

\begin{cond}[Regularity]
 For each $a \in \sA$, the bounded operators $a$ and $[D,a]$ lie in the
 smooth domain $\bigcap_{k\geq 1} \Dom \dl^k$ of the derivation
 $\dl \: T \mapsto [|D|, T]$.
 
 Moreover, $\sA$ is complete in the topology given by the seminorms
 $q_k \: a \mapsto \|\dl^k(a)\|$ and
 $q'_k \: a \mapsto \|\dl^k([D,a])\|$. This ensures that $\sA$ is a
 Fr\'echet pre-$C^*$-algebra.
 \end{cond}

 \begin{cond}[Finiteness]\label{fin_cond}
 The subspace of smooth vectors
 $\sH^\infty := \bigcap_{k\in\bN} \Dom D^k$ is a \emph{finitely
 generated projective} left $\sA$-module.
 
 This is equivalent to saying that, for some $N \in \bN$, there is a 
 projector $p = p^2 = p^*$ in~$M_N(\sA)$ such that 
 $\sH^\infty \isom \sA^N p$ as left $\sA$-modules. 
 \end{cond}

 \begin{cond}[Real structure]
 There is an antiunitary operator $J : \sH \to \sH$ satisfying
 $J^2 = \pm 1$, $JDJ^{-1} = \pm D$, and $J\Ga = \pm \Ga J$ in the even
 case, where the signs depend only on $n \bmod 8$ (and thus are given
 by the table of signs for the standard commutative examples).
 \[
    \begin{array}[t]{|c|cccc|}
    \hline
    n \bmod 8               & 0 & 2 & 4 & 6 \rule[-5pt]{0pt}{17pt} \\
    \hline
    J^2 = \pm 1             & + & - & - & + \rule[-5pt]{0pt}{17pt} \\
    JD = \pm D J & + & + & + & + \rule[-5pt]{0pt}{17pt} \\
    J\Ga = \pm\Ga J         & + & - & + & - \rule[-5pt]{0pt}{17pt} \\
    \hline
    \end{array}
    \qquad\qquad
    \begin{array}[t]{|c|cccc|}
    \hline
    n \bmod 8               & 1 & 3 & 5 & 7 \rule[-5pt]{0pt}{17pt} \\
    \hline
    J^2 = \pm 1             & + & - & - & + \rule[-5pt]{0pt}{17pt} \\
    JD = \pm D J & - & + & - & + \rule[-5pt]{0pt}{17pt} \\
    \hline
    \end{array}
    \]
 Moreover, $b \mapsto J b^* J^{-1}$ is an antirepresentation of $\sA$
 on~$\sH$ (that is, a representation of the opposite algebra
 $\sA^\opp$), which commutes with the given representation of~$\sA$:
 $$
 [a, J b^* J^{-1}] = 0,  \word{for all} a,b \in \sA.
 $$
 \end{cond}

 \begin{cond}[First order]
 For each $a,b \in \sA$, the following relation holds:
\begin{equation}\label{fist_order}
 [[D, a], J b^* J^{-1}] = 0,  \word{for all} a,b \in \sA.
\end{equation}
 This generalizes, to the noncommutative context, the condition that
 $D$ be a first-order differential operator.
 
 Since 
 \[
 [[D, a], Jb^*J^{-1}]
 = [[D, Jb^*J^{-1}], a] + [D, \underbrace{[a, Jb^*J^{-1}]}_{=0}],
 \]
 this is equivalent to the condition that $[a, [D, Jb^*J^{-1}]] = 0$.
 \end{cond}

 \begin{cond}[Orientation]
 There is a Hochschild $n$-cycle
 $$
 \cc = \tsum_j (a_j^0 \ox b_j^0) \ox a_j^1 \oxyox a_j^n 
 \in Z_n(\sA, \sA \ox \sA^\opp),
 $$
 such that
 \begin{equation}
 \pi_D(\cc)
 \equiv \tsum_j a_j^0 (J b_j^{0*} J^{-1}) \,[D,a_j^1] \dots [D,a_j^n]
 = \begin{cases} \Ga, &\text{if $n$ is even}, \\
                   1, &\text{if $n$ is odd}. \end{cases}
 \label{eq:vol-cond}
 \end{equation}
 \end{cond}

 In many examples, including the noncommutative examples we shall meet
 in the next two sections, one can often take $b_j^0 = 1$, so that
 $\cc$ may be replaced, for convenience, by the cycle
 $\sum_j a_j^0 \ox a_j^1 \oxyox a_j^n \in Z_n(\sA, \sA)$. In the
 commutative case, where $\sA^\opp = \sA$, this identification may be
 justified: the product map $m\: \sA \ox \sA \to \sA$ is a
 homomorphism.
 
 The data set $(\sA, \sH, D; \Ga \text{ or } 1, J, \cc)$ satisfying
 these six conditions constitute a ``noncommutative spin geometry''. In
 the fundamental paper where these conditions were first laid out
 \cite{connes:grav}, Connes added one more nondegeneracy condition
 (Poincar\'e duality in $K$-theory) as a requirement. We shall not go
 into this matter here.
 
   \section{Topological constructions}\label{inf_to}
  
  \paragraph{} This section is concerned with a topological construction of an infinitely listed covering projection from finitely listed ones. 
  \begin{empt}\label{top_constr}
  	Let $\mathcal{X}$ be a  second-countable \cite{munkres:topology} locally compact connected Hausdorff space, and let
  	
  	\begin{equation}\label{top_inv_limit}
  	\mathcal{X} = \mathcal{X}_0 \xleftarrow{}... \xleftarrow{} \mathcal{X}_n \xleftarrow{} ... 
  	\end{equation}
  	be a sequence of finitely listed covering projections. Let $\left\{G_n = G\left(\mathcal{X}_n | \mathcal{X}\right)\right\}_{n \in \mathbb{N}}$ be groups of covering transformations. Let $\widehat{\mathcal{X}} = \varprojlim \mathcal{X}_n$, $\widehat{G} =\varprojlim G_n$ be inverse limits. The group $\widehat{G}$ has a topology defined by subgroups of finite index \cite{milne:etale}. Let $\overline{G}$ be a discrete group algebraically isomorphic to  $\widehat{G}$.  For any  $x \in \widehat{\mathcal{X}}$ let us define a map $\phi_x: \overline{G} \to \widehat{\mathcal{X}}$ be given by $g \mapsto gx$. We define  a topological space $\overline{\mathcal{X}}$ as the set $\widehat{\mathcal{X}}$ with the final topology \cite{bourbaki_sp:gt} such that the identical map $\mathrm{Id} : \widehat{\mathcal{X}} \to \overline{\mathcal{X}}$, and maps $\varphi_x$ for any $x \in \widehat{\mathcal{X}}$ are continuous. Action of $\overline{G}$ on $\overline{\mathcal{X}}$ is free and properly discontinuous, so there is a natural regular covering projection $\overline{\pi}:\overline{\mathcal{X}} \to \mathcal{X}$. Let $\widetilde{\mathcal{X}}$ be a connected component of $\overline{\mathcal{X}}$, and let $G \subset \overline{G}$ be a maximal subgroup such that
  	\begin{equation}\label{comm_normal}
  	G\widetilde{\mathcal{X}}=\widetilde{\mathcal{X}}.
  	\end{equation}

  	For any $g \in \overline{G}$ a subgroup $gGg^{-1}$ satisfies to \eqref{comm_normal}, i.e. $gGg^{-1}=G$, whence $G$ is a normal subgroup. Any connected component of a covering space is also a covering space, therefore  $\widetilde{\pi}=\overline{\pi}|_{\widetilde{\mathcal{X}}}:\widetilde{\mathcal{X}} \to \mathcal{X}$ is a covering projection and $\mathcal{X}\approx\widetilde{\mathcal{X}}/G$, i.e. $\widetilde{\mathcal{X}} \to \mathcal{X}$ is a regular covering projection.  
  	\begin{defn}
  		The space $\overline{\mathcal{X}}$ is said to be a {\it disconnected covering space} of the sequence \eqref{top_inv_limit}, and the space $\widetilde{\mathcal{X}}$ is said to be  a {\it connected covering space} of the sequence \eqref{top_inv_limit}.
  	\end{defn}
  \end{empt}

  \begin{defn}\label{fund_dom}\cite{chavel:riemann} Let $p: \widetilde{\mathcal X} \to \mathcal{X}$ be a (topological) covering projection.
  	A {\it fundamental domain} of $G$ is an open connected set $\mathcal{D} \subset \widetilde{\mathcal{X}}$ such that
  	\begin{enumerate}
  		\item[(a)] 
  		
  		\begin{equation}
  		p\left(\mathfrak{cl}\left(\mathcal{D}\right)\right)=\mathcal{X}
  		\end{equation}
  		\item[(b)] 
  		\begin{equation}
  		g \mathcal{D} \bigcap \mathcal{D} = \emptyset \ \text{ for any nontrivial } g\in G.
  		\end{equation}
  		
  	\end{enumerate}

  \end{defn}
  \begin{rem}
  The condition (a) of the Definition \ref{fund_dom} is equivalent to 
  \begin{equation*}
  \widetilde{\mathcal{X}} = \bigcup_{g \in G}g ~\mathfrak{cl}\left(\mathcal{D}\right).
  \end{equation*}
  \end{rem}
  \begin{rem}\label{cut_loci_rem}\cite{chavel:riemann}
  	When considering a Riemannian covering $\widetilde{M}\to M$ one may construct a fundamental by following way. For any point $x \in M$ there is the \textit{cut loci}, which is an open set $\Om_x \subset M$ such that $M \backslash \Om_x$ is a set of measure 0 \cite{chavel:riemann}. It means that 
  $1_{\Om_x}\in L^\infty\left(M\right)$ and
  \begin{equation}
 1_{\Om_x} = 1_M.
  \end{equation}
From \cite{chavel:riemann} it follows that for any $\widetilde{x}\in p^{-1}\left(x\right)$ there is the natural connected open subset $\widetilde{\Om}_{\widetilde{x}}$ which is mapped homeomorphically on $\Om_x$ and 
 \begin{equation}
 \sum_{g \in G\left(\widetilde{M}|M\right)}g ~1_{\widetilde{\Om}_{\widetilde{x}}} = 1_{\widetilde{M}}.
 \end{equation}
  The definition of the cut loci is contained in \cite{chavel:riemann}.


  \end{rem}
 
  \begin{defn}\label{fund_domu}
  	Let $\widetilde{\mathcal{X}} \to \mathcal{X}$ be a topological covering projection. A finite or countable family $\left\{\mathcal{U}_{\iota}\subset\mathcal{X}\right\}_{\iota \in I}$ of connected relatively compact open sets, such that

  	\begin{enumerate}
  		\item
  		\begin{equation}\label{inf_union}
  		\mathcal{U}_i \text{ is evenly covered by } \pi^{-1}\left(\mathcal{U}_{\iota}\right),
  		\end{equation}
  		\item
  		\begin{equation}\label{inf_com}
  		\bigcup_{\iota \in I}  \mathcal{U}_{\iota} = \mathcal{X}  \text{ is a locally finite covering},
  		\end{equation}
  		\item
  		\begin{equation}\label{inf_comq}
  		\bigcup_{\iota \in I \backslash \{\iota_0\} } \mathcal{U}_{\iota} \neq \mathcal{X}; \ \forall \iota_0 \in I.
  		\end{equation}
  	\end{enumerate}
  	is said to be a {\it fundamental covering } of $\widetilde{\mathcal{X}} \to \mathcal{X}$.  Let us select a single connected subset $\widetilde{\mathcal{U}}_{\iota}\subset \widetilde{\mathcal{X}}$ which is mapped homeomorphically onto $\mathcal{U}_{\iota}$ , and we require that the union $\bigcup_{\iota \in I}\widetilde{\mathcal{U}}_{\iota}$ is a connected set. The family $\left\{\widetilde{\mathcal{U}}_{\iota}\subset\widetilde{\mathcal{X}}\right\}_{\iota \in I}$ is said to be a {\it basis of the fundamental covering}.
  \end{defn}
  \begin{defn}\label{lift_desc_defn}
  	Let $\widetilde{\pi}: \widetilde{\mathcal X} \to \mathcal X$ be a covering projection and let $\widetilde{\mathcal U}\subset \widetilde{\mathcal X}$ be a connected open set which is mapped homeomorphicaly onto $\mathcal U= \widetilde{\pi}\left(\widetilde{\mathcal U}\right)$. If $\varphi \in C_0\left(\mathcal{X}\right)$ is such that $\varphi\left(\mathcal X \backslash \mathcal U\right)= \{0\}$ then a function $\widetilde{\varphi} \in C_0\left(\widetilde{\mathcal X}\right)$ given by
  	\begin{equation*}
  	\widetilde{\varphi}\left(\widetilde{x}\right)=\left\{
  	\begin{array}{c l}
  	\varphi\left(\widetilde{\pi}\left(\widetilde{x}\right)\right) & \widetilde{x} \in \widetilde{\mathcal{U}}  \\
  	0 & \widetilde{x} \notin \widetilde{\mathcal{U}}
  	\end{array}\right.
  	\end{equation*}
  	is said to be the $\widetilde{\mathcal{U}}$-{\it lift } of $\varphi$. Otherwise if $\widetilde{\varphi} \in C_0\left(\widetilde{\mathcal{X}}\right)$  is such that $\widetilde{\varphi} \left(\widetilde{\mathcal X} \backslash \widetilde{\mathcal U}\right)=  \{0\}$  then a function $\varphi \in C_0\left(\mathcal X\right)$ given by
  	\begin{equation*}
  	\varphi\left(x\right)=\left\{
  	\begin{array}{c l}
  	\widetilde{\varphi}\left(\widetilde{x}\right) & x \in \mathcal{U},  \  \widetilde{x} \in \widetilde{\mathcal U}, \  \widetilde{\pi}\left(\widetilde{x}\right)=x   \\
  	0 & x \notin \mathcal{U}      \end{array}\right.
  	\end{equation*}
  	is said to be the {\it descent} of $\widetilde{\varphi}$ on $\mathcal X$.
  \end{defn}
  \begin{defn}
  	Let $\left\{\widetilde{\mathcal{U}}_{\iota}\subset\widetilde{\mathcal{X}}\right\}_{\iota \in I}$ be a basis of the fundamental covering of $\widetilde{\pi}:\widetilde{\mathcal{X}}\to \mathcal{X}$. Let $\sum_{\iota \in I}a_{\iota} = 1_{C_b\left(\mathcal X\right)}$ be a partition of unity  dominated by $\left\{\widetilde{\pi}\left(\widetilde{\mathcal U}_\iota\right)\right\}_{\iota \in I}$ , and let $\widetilde{a}_{\iota}\in C_0\left(\widetilde{\mathcal X}\right)$ be the $\widetilde{\mathcal{U}}_{\iota}$-lift of $a_\iota$ for any $\iota \in I$.  
  	A partition of unity given by
  	\begin{equation}\label{part_unity_g}
  	1_{C_b\left(\widetilde{\mathcal X}\right)} = \sum_{g \in G} \sum_{\iota \in I}  g\widetilde{a}_\iota = \sum_{\left(g,\iota\right)\in G \times I}\widetilde{a}_{\left(g,\iota\right)}
  	\end{equation}
  	where $G = G\left(\widetilde{\mathcal X} \ | \ \mathcal X \right)$ and $\widetilde{a}_{\left(g,\iota\right)} = ga_\iota$ is said to be {\it dominated} by  $\sum_{\iota \in I}a_{\iota} = 1_{C_b\left(\mathcal X\right)}$. We also say that \eqref{part_unity_g} is the partition of unity is {\it dominated} by $\left\{\widetilde{\mathcal{U}}_{\iota}\right\}_{\iota \in I}$.
  \end{defn}
  \begin{empt}
  	Let $\phi_{[1,0]}: [0,1]\to [0,1]$ be a continuous function such that
  	\begin{equation*}
  	\phi_{[1,0]}\left(x\right) = \left\{
  	\begin{array}{c l}
  	1 & x \le \frac{1}{2} \\
  	0 & x = 1.
  	\end{array}\right.
  	\end{equation*}
  	Let $\widetilde{\pi}: \widetilde{\mathcal X} \to \mathcal X$ be a covering projection such that $\widetilde{\mathcal X}$ is a locally compact metric space. For any $\widetilde{x} \in \widetilde{\mathcal X}$ there is $r > 0$ such that the set $\widetilde{\mathcal U}=\left\{\widetilde{y} \in \widetilde{\mathcal X} \ | \ \mathfrak{dist}\left(\widetilde{y}, \widetilde{x}\right) \le r  \right\}$ is mapped homeomorphicaly on $\widetilde{\pi}\left(\widetilde{\mathcal U}\right)$. There is a function $\phi^{\widetilde{x}}_{[1,0]}\in C_0\left(\widetilde{\mathcal{X}}\right)$ given by
  	\begin{equation}
  	\phi^{\widetilde{x}}_{[1,0]}\left(\widetilde{y}\right) = \left\{
  	\begin{array}{c l}
  	\phi_{[1,0]}\left(\frac{\mathfrak{dist}\left(\widetilde{y}, \widetilde{x}\right)}{r}\right) & \widetilde{y} \in \widetilde{\mathcal U} \\
  	0 & \widetilde{y} \notin \widetilde{\mathcal U} .
  	\end{array}\right.
  	\end{equation}
  	There is an open neighborhood $\widetilde{\mathcal V}=\left\{\widetilde{y} \in \widetilde{\mathcal X} \ | \ \mathfrak{dist}\left(\widetilde{y}, \widetilde{x}\right) < \frac{r}{2}  \right\}$ of $\widetilde{x}$ such that  $\phi^{\widetilde{x}}_{[1,0]}\left(\widetilde{\mathcal V}\right)= \{1\}$.
  \end{empt}
  \begin{defn}\label{test_func_defn}
  	A quadruple $\left(\widetilde{x}, \phi^{\widetilde{x}}_{[1,0]}, \widetilde{\mathcal U}, \widetilde{\mathcal V} \right)$ is said to be a {\it  test function subordinated to} $\widetilde{\pi}$.
  \end{defn}
  
  \begin{empt}\label{smooth_test}
  	If $M$ is a $C^\infty$-manifold then we can define smooth test functions. Let us select a $C^\infty$ function $\phi_{[1,0]}: [0,1]\to [0,1]$ such that following conditions hold
  	Let $\phi_{[1,0]}: [0,1]\to [0,1]$ be a continuous function such that
  	\begin{equation*}
  	\phi_{[1,0]}\left(x\right) = \left\{
  	\begin{array}{c l}
  	1 & x \le \frac{1}{2} \\
  	\\
  	0 & x \ge \frac{3}{4}.
  	\end{array}\right.
  	\end{equation*}	
  	If $\widetilde{\pi}:\widetilde{M}\to M$ is a covering projection then any point $\widetilde{x} \in  \widetilde{M}$ has an open neighborhood $\widetilde{x} \in \widetilde{U}\subset\widetilde{M}$ such that
  	\begin{enumerate}
  		\item There is an infinitely differentiable diffeomorphism $\psi : \widetilde{U} \to \mathbb{R}^n$.
  		\item The restriction $\widetilde{\pi}|_{\widetilde{U}}: \widetilde{U} \to \widetilde{\pi}\left(\widetilde{U}\right)$ is a homeomorphism.
  	\end{enumerate}
  	Suppose that $\psi$ is such that $\psi\left(\widetilde{x}\right) = 0 \in \mathbb{R}^n$, and let $\|\cdot\|_{\mathbb{R}^n}$ be a norm on $\mathbb{R}^n$ given by
  	\begin{equation*}
  	\left\|	\left(\begin{array}{c l}
  	x_1 \\
  	...	\\
  	x_n
  	\end{array}\right)\right\|_{\mathbb{R}^n} = \sqrt{x_1^2 + ... + x_n^2}.
  	\end{equation*}
  	Let us introduce a $C^\infty$ function $\phi^{\widetilde{x}}_{[1,0]}: \widetilde{M} \to \mathbb{R}$ given by
  	\begin{equation}
  	\phi^{\widetilde{x}}_{[1,0]}\left(\widetilde{y}\right) = \left\{
  	\begin{array}{c l}
  	\phi_{[1,0]}\left(\left\|\psi\left(\widetilde{y}\right)\right\|_{\mathbb{R}^n}\right) & \widetilde{y} \in \widetilde{ U} \\
  	0 & \widetilde{y} \notin \widetilde{ U} .
  	\end{array}\right.
  	\end{equation}
  	If $\widetilde{V} = \left\{\widetilde{y}\in \widetilde{U} \ | \ \left\|\psi\left(\widetilde{y}\right)\right\|_{\mathbb{R}^n} < \frac{1}{2}\right\}$ then $\phi^{\widetilde{x}}_{[1,0]}\left(\widetilde{V}\right)	= \{1\}$.
  	
  \end{empt}
  \begin{defn} In the situation \ref{smooth_test} 
  	a quadruple $\left(\widetilde{x}, \phi^{\widetilde{x}}_{[1,0]}, \widetilde{\mathcal U}, \widetilde{\mathcal V} \right)$ is said to be a {\it $C^\infty$ test function subordinated} to $\widetilde{\pi}$.
  \end{defn}
  \begin{defn}\label{gluin}
  	Let $\mathcal{X}$, $\mathcal{Y}$ be Hausdorff spaces, and let $\left\{\mathcal{U}_\iota\subset \mathcal{X}\right\}_{\iota \in I}$ by a family of open sets such that $\mathcal{X}  = \bigcup_{\iota \in I} \mathcal{U}_{\iota}$. A family of continuous maps  $\varphi^{\mathcal{U}_{\iota}}: \mathcal{U}_{\iota} \to \mathcal{Y}$ is said to be {\it coherent} if $\varphi^{\mathcal{U}_{\iota}}|_{\mathcal{U}_{\iota} \bigcap \mathcal{U}_{\iota'}}=\varphi^{\mathcal{U}_{\iota'}}|_{\mathcal{U}_{\iota} \bigcap \mathcal{U}_{\iota'}}$ For any coherent sequence there is the unique continuous map $\varphi: \mathcal{X} \to \mathcal{Y}$ such that $\varphi|_{{\mathcal{U}_{\iota}}} = \varphi^{\mathcal{U}_{\iota}}$;  $\forall \iota \in I$. The map $\varphi$ is said to be the {\it gluing} of $\left\{\varphi^{\mathcal{U}_{\iota}}\right\}_{\iota \in I}$. Let denote by $\varphi = \mathfrak{Gluing}\left(\left\{\varphi^{\mathcal{U}_{\iota}}\right\}_{\iota \in I}\right)$ the gluing of $\left\{\varphi^{\mathcal{U}_{\iota}}\right\}_{\iota \in I}$.
  \end{defn}

 \section{Noncommutative covering projections}
 \paragraph{} In this section I follow to the \cite{ivankov:inv_lim}.
 \subsection{Finite case}
 
 \begin{defn}\cite{ivankov:inv_lim}
 	If $A$ is a $C^*$-algebra then an action of a group $G$ is said to be {\it involutive } if $ga^* = \left(ga\right)^*$ for any $a \in A$ and $g\in G$.
 \end{defn}
 \begin{defn}\label{hilb_product_defn}\cite{ivankov:inv_lim}
 	If  $B \subset A$ is an inclusion of $C^*$-algebras, $G$ is a finite group with an involutive action on $A$ such that $B = A^G$, then there is a $B$-valued sesquilinear product on $A$ given by 
 	\begin{equation*}
 	\langle a, b \rangle_{A} = \frac{1}{|G|} \sum_{g \in G} g(a^*b)
 	\end{equation*}
 	The structure of Hilbert $B$-module $A$ is said to be the {\it induced by} $G$-action. 
 \end{defn}
 \begin{rem}
The Definition \ref{hilb_product_defn} complies with the Theorem \ref{pavlov_troisky_thm}.
 \end{rem}
 
 \begin{defn}\label{fin_def}\cite{ivankov:inv_lim}
 	Let $\pi:A \to \widetilde{A}$ be an injective *-homomorphism of $C^*$-algebras, and let $G$ be a finite group such that following conditions hold:
 	\begin{enumerate}
 		\item There is an involutive continuous action of $G$ on $\widetilde{A}$ such that $\widetilde{A}^G=A$; 
 		\item $\widetilde{A}\subset \mathcal{K}\left(\widetilde{A}_A\right)$ where structure of  Hilbert $A$-module  $\widetilde{A}_A$ is induced by $G$-action;
 		\item There is a finite or countable set $I$ and indexed by $I$ subsets $\{a_{\iota}\}_{\iota \in I}$, $\{b_{\iota}\}_{\iota \in I} \subset \widetilde{A}$ such that
 		\begin{equation}\label{can_nc}
 		\sum_{\iota \in I} a_{\iota}(gb_\iota)=\left\{
 		\begin{array}{c l}
 		1_{M(\widetilde{A})} & g \in  G \text{ is trivial}\\
 		0 & g \in  G \text{ is not trivial}
 		\end{array}\right..
 		\end{equation}
 		where the sum of the series means the strict convergence \cite{blackadar:ko}.
 		
 	\end{enumerate}
 	Then $\pi$ is said to be a {\it finite noncommutative covering projection}, $G$ is said to be the {\it covering transformation group}. Denote by $G(\widetilde{A}|A)=G$. The algebra $\widetilde{A}$ is said to be the {
 		\it covering algebra}, and $A$ is called the {\it base algebra} of the covering projection. A triple $\left(A, \widetilde{A}, G\right)$ is also  said to be a {\it noncommutative finite covering projection}.
 \end{defn}
 \begin{rem}
  The article \cite{ivankov:inv_lim} contains the comprehensive foundation of the Definition  \ref{fin_def}.
 \end{rem}
 
 \begin{defn}
 	Let $\left(A, \widetilde{A}, G\right)$ be a noncommutative finite covering projection.  Algebra  $\widetilde{A}$  is a  finitely generated projective Hilbert $A$-modules with induced by $G$-action sesquilinear product given by
 	\begin{equation}\label{fin_form_a}
 	\langle a, b \rangle_{\widetilde{A}} = \frac{1}{|G|} \sum_{g \in G} g(a^*b)
 	\end{equation}
 	We say that the structure of Hilbert $A$-module is {\it induced by the covering projection} $\left(A, \widetilde{A}, G\right)$. Henceforth we shall consider $\widetilde{A}$ as a right $A$-module. 
 \end{defn}
 \begin{empt}\label{dir_sum_constr}
 	Let $\left(A, \widetilde{A}, G\right)$ be a noncommutative finite covering projection. If $\widetilde{a} \in \widetilde{A}$ then $\widetilde{a} = a + p$ where $a = \frac{1}{|G|}\sum_{g \in G}g\widetilde{a}$ and $p = \widetilde{a}-a$. It is clear that $\sum_{g \in G}gp = 0$ and for any $G$-invariant $b \in \widetilde{A}$ we have 
 	\begin{equation*}
 	\langle b, p \rangle_{\widetilde{A}} = \frac{1}{|G|}\widetilde{b}\sum_{g \in G}gp = 0.
 	\end{equation*}
 	Otherwise the set of $G$-invariant elements is just a sublagebra $A \subset \widetilde{A}$.
 	So  $\widetilde{A}_A$ can be decomposed into the direct orthogonal sum, i.e.
 	\begin{equation}\label{hilb_mod_direct_sum}
 	\begin{split}
 	\widetilde{A}_A= A \oplus P; \  
 	A \perp P \ , \text{i.e. } \langle \widetilde{b}, p \rangle_{\widetilde{A}}= 0; \text{ for any } a \in A; \ p \in P. 
 	\end{split}
 	\end{equation}
 \end{empt}

 \begin{exm}\label{circle_fin}{\it Finite covering projections of the circle $S^1$.} 
 	There is the universal covering projection $\widetilde{\pi}: \mathbb{R}\to S^1$. 
 	Let $\widetilde{\mathcal{U}}_1,\ \widetilde{\mathcal{U}}_2 \subset \mathbb{R}$ be such that
 	\begin{equation}\label{ur}
 	\widetilde{\mathcal{U}}_1 = (-\pi -1/2,  1/2), \ \widetilde{\mathcal{U}}_2 = (-1/2, \pi +  1/2). 
 	\end{equation} For any $i \in \{1,2\}$ the set $\mathcal{U}_i = \widetilde{\pi}(\widetilde{\mathcal{U}}_i) \subset S^1$ is open, connected and evenly covered. Since $S^1=\mathcal{U}_1\bigcup\mathcal{U}_2$ there is  a partition of unity $a_1, a_2$ dominated by $\{\mathcal{U}_i\}_{i\in \{1,2\}}$ \cite{munkres:topology}, i.e.  $a_i: S^1 \to [0,1]$ are such that
 	\begin{equation*}
 	\left\{
 	\begin{array}{c l}
 	a_i(x)>0 & x \in \mathcal{U}_i \\
 	a_i(x)=0 & x \notin \mathcal{U}_i
 	\end{array}\right.; \ i \in \{1,2\}.
 	\end{equation*}
 	and $a_1+a_2 = 1_{C(S^1)}$. From the Proposition \ref{smooth_part_unity_prop} is follows that we can select smooth partition of unity, i.e. $a_1, a_2 \in \Coo(S^1)$. If $e_1, e_2\in \Coo(S^1)$ are given by
 	\begin{equation}\label{e_1_e_2}
 	e_i = \sqrt{a_i}; \ i=1,2;
 	\end{equation}
 	then
 	\begin{equation*}
 	\left(e_1\right)^2+\left(e_2\right)^2= 1_{C_0(S^1)}.
 	\end{equation*}
 	If $\widetilde{e}_i \in C_0(\mathbb{R})$ are given by
 	\begin{equation}\label{e_tilde}
 	\widetilde{e}_i(\widetilde{x})=\left\{
 	\begin{array}{c l}
 	e_i(\widetilde{\pi}(\widetilde{x}))>0 & \widetilde{x} \in \widetilde{\mathcal{U}}_i \\
 	0 & \widetilde{x} \notin \widetilde{\mathcal{U}}_i
 	\end{array}\right.; \ i \in \{1,2\}
 	\end{equation}
 	then there is a pointwise (=weak) convergence of the following series
 	\begin{equation*}
 	\sum_{i=1,2; \ n \in \mathbb{Z}} n \cdot \widetilde{e}^{2}_i = 1_{C_b\left(\mathbb{R}\right)=M\left(C_0\left(\mathbb{R}\right)\right)}, 
 	\end{equation*}
 	where $n \cdot -$ means the natural action of $G(\mathbb{R}|S^1) \approx \mathbb{Z}$ on $C_0(\mathbb{R})$.
 	Let $\overline{\pi}^n: \mathcal{X}_n\to S^1$ be an $n$-listed covering projection then $G(\mathcal{X}_n|S^1)\approx \mathbb{Z}_n$. It is well known that $\mathcal{X}_n\approx S^1$ but we use the  $\mathcal{X}_n$ notion for clarity. There is a sequence of covering projections $\mathbb{R}\xrightarrow{\pi^n}\mathcal{X}_n \to S^1$. If $\mathcal{U}^n_i = \pi^n\left(\widetilde{\mathcal{U}}_i\right)$ then $\mathcal{U}^n_i \bigcap g\mathcal{U}^n_i = \emptyset$ for any nontrivial $g \in G(\mathcal{X}_n, S^1)$. If $e^n_i \in C(\mathcal{X}_n)$ is given by
 	\begin{equation}\label{e_n_i}
 	e^n_i(\pi^n(\widetilde{x}))=\left\{
 	\begin{array}{c l}
 	\widetilde{e}_i(\widetilde{x}) & \pi^n(\widetilde{x}) \in \mathcal{U}^n_i  \\
 	0 & \pi^n(\widetilde{x}) \notin \mathcal{U}^n_i
 	\end{array}\right.; \ i \in \{1,2\}
 	\end{equation}
 	then
 	\begin{equation*}
 	\begin{split}
 	\sum_{i \in \{1,2\}; \ g \in G(\mathcal{X}_n, S^1)} g\left(e^n_i\right)^2 = 1_{C_0(\mathcal{X}_n)}; \\ e^n_i(ge^n_i) = 0; \ \text{ for any notrivial } g \in G(\mathcal{X}_n|S^1).
 	\end{split}
 	\end{equation*}
 	If $I_n = G(\mathcal{X}_n|S^1) \times \{1,2\}$
 	and
 	\begin{equation}\label{iota_def}
 	e^n_{\iota} = ge_i^n; \text{ where } \iota = (g,i) \in I_n
 	\end{equation}
 	then
 	\begin{equation}\label{circ_sum}
 	\sum_{\iota \in I_n} e^n_{\iota}(ge^n_\iota)=\left\{
 	\begin{array}{c l}
 	1_{C_0(\mathcal{X}_n)} & g \in  G(\mathcal{X}_n|S^1) \text{ is trivial}\\
 	0 & g \in  G(\mathcal{X}_n|S^1) \text{ is not trivial}
 	\end{array}\right..
 	\end{equation}
 	
 	So a natural *-homomorphism $\pi:C(S^1)\to C(\mathcal{X}_n)$ satisfies the condition 3 of the Definition \ref{fin_def}. Otherwise $C(\mathcal{X}_n)\approx C(S^1)^n$ as $C(S^1)$-module, i.e. $C(\mathcal{X}_n)$ is a finitely generated projective left and right $C(S^1)$-module. So a triple $\left(C_0(S^1),C_0(\mathcal{X}_n), \mathbb{Z}_n\right)$ is a finite noncommutative covering projection.
 \end{exm}
 \begin{exm}\label{fin_lem}{\it Finite covering projections of locally compact spaces.} The Example \ref{circle_fin} can be generalized.
 	Let $\pi:\widetilde{\mathcal{X}}\to \mathcal{X}$ be a topological finitely listed covering projection such that $\mathcal{X}$ is a second-countable locally compact Hausdorff space. Suppose that both $\widetilde{\mathcal{X}}$ and $\mathcal{X}$ are connected spaces. There is an involutive continuous action of the covering  transformation group  $G = G\left(\widetilde{\mathcal X}| \mathcal X\right)$ on $C_0\left(\widetilde{\mathcal X}\right)$ arising from the action of $G$ on $\widetilde{\mathcal X}$,  and $C_0\left(\mathcal X\right)=C_0\left(\widetilde{\mathcal X}\right)^G$, i.e. condition 1 of the Definition \ref{fin_def} hold. Let $\left\{\widetilde{\mathcal{U}}_{\iota}\subset\widetilde{\mathcal{X}}\right\}_{\iota \in I}$  be a {\it basis of the fundamental covering}, and let $1_{C_b\left(\widetilde{\mathcal X}\right)} = \sum_{g \in G} \sum_{\iota \in I}  g\widetilde{a}_\iota = \sum_{\left(g,\iota\right)\in G \times I}\widetilde{a}_{\left(g,\iota\right)}$ be a partition of unity dominated by $\left\{\widetilde{\mathcal{U}}_{\iota}\right\}_{\iota \in I}$. If  $\widetilde{e}_{\left(g,\iota\right)}\in C_0\left(\widetilde{\mathcal X}\right)$ is given by
 	\begin{equation}\label{a_iota}
 	\widetilde{e}_{\left(g,\iota\right)} = \sqrt{g \widetilde{a}_\iota}.
 	\end{equation}
 	then 
 	\begin{equation*}
 	\sum_{\left(g',\iota\right) \in G \times I} \widetilde{e}_{\left(g',\iota\right)}\left(g\widetilde{e}_{\left(g',\iota\right)}\right)=\left\{
 	\begin{array}{c l}
 	1_{M\left(C_0\left(\widetilde{\mathcal{X}}\right)\right)} & g \in  G \text{ is trivial}\\
 	0 & g \in  G \text{ is not trivial}
 	\end{array}\right..
 	\end{equation*}
 	So condition 3 of the Definition \ref{fin_def} hold. If $\varphi \in C_0\left(\widetilde{\mathcal{X}}\right)$ is any function then
 	\begin{equation*}
 	\varphi = \sum_{\left(g,\iota\right) \in G \times I} \varphi \widetilde{e}_{\left(g,\iota\right)}\widetilde{e}_{\left(g,\iota\right)}=\varphi \sqrt{\widetilde{e}_{\left(g,\iota\right)}}\sqrt{\widetilde{e}_{\left(g,\iota\right)}}\widetilde{e}_{\left(g,\iota\right)}
 	\end{equation*}
 	and from above equation and definition of $\widetilde{e}_{\left(g,\iota\right)}$ it follows that
 	\begin{equation*}
 	\varphi = \sum_{\iota \in I} \left(\left\langle \varphi, \sqrt{\widetilde{e}_{\left(g,\iota\right)}} \right\rangle_{C_0\left(\widetilde{\mathcal X}\right)} \sqrt{\widetilde{e}_{\left(g,\iota\right)}} \right) \left\rangle \right\langle \widetilde{e}_{\left(g,\iota\right)},
 	\end{equation*}
 	i.e. $\varphi$ is an infinite sum of rank one operators. The sum is norm convergent, whence $\varphi\in \mathcal{K}\left(C_0\left(\widetilde{\mathcal{X}}\right)_{C_0\left(\mathcal X\right)}\right)$, and  condition 2 of the Definition \ref{fin_def} hold. So all conditions of the definition \ref{fin_def} hold, whence the triple
 	$\left(C_0\left(\mathcal{X}\right),C_0\left(\widetilde{\mathcal{X}}\right), G\left(\widetilde{\mathcal X}| \mathcal X\right)\right)$ is a finite noncommutative covering projection.
 	
 \end{exm}

 \begin{exm}\label{nt_fin_cov}{\it A finite covering projection of a noncommutative torus}.
 	A noncommutative torus \cite{varilly:noncom} $A_{\theta}$ is an universal unital $C^*$-algebra generated by two unitary elements ($u, v \in U(A_{\theta})$) with one relation given by
 	\begin{equation}\label{nt_commutation}
 	uv = e^{2\pi i\theta}vu, \ (\theta \in \mathbb{R} \backslash \mathbb{Q}).
 	\end{equation}
 	Let $\overline{\pi} : A_{\theta} \to A_{\theta'}$ be a *-homomorphism such that:
 	\begin{itemize}
 		\item There are $m, n, k \in \mathbb{N}$ such that $\theta'=\frac{\theta + 2 \pi k}{mn}$;
 		\item $A_{\theta'}$ is generated by $u_m,v_n \in U(A_{\theta'})$ and $\overline{\pi}$ is given by
 		\begin{equation}\label{fin_nt}
 		u \mapsto u_m^m; \ v \mapsto v_n^n.
 		\end{equation}
 	\end{itemize}
 	
 	There is a continuous involutive action of $G=\mathbb{Z}_m\times\mathbb{Z}_n$ on $A_{\theta'}$ given by
 	\begin{equation*}
 	\left(\overline{p}, \overline{q}\right) u_m = u_m e^{\frac{2\pi i p}{m}}, \ \left(\overline{p}, \overline{q}\right) v_n = v_n e^{\frac{2\pi i q}{n}}; \ \forall \left(\overline{p}, \overline{q}\right) \in G =\mathbb{Z}_m\times\mathbb{Z}_n,
 	\end{equation*}
 	and $A_{\theta}=A_{\theta'}^{G}$, i.e. the condition 1 of the Definition \ref{fin_def} hold.
 	If $\{e_{\iota}^m\}_{\iota \in I_m}$ and $\{e_{\iota}^n\}_{\iota \in I_n}$ are given by \eqref{iota_def} then from \eqref{circ_sum} it follows that
 	\begin{equation}\label{circ_sum_nt}
 	\begin{split}
 	\sum_{\iota \in I_m \ }  e^m_{\iota}(u_m) \left(\overline{k} e^{m}_{\iota}(u_m)\right) \left( \text{resp. }\sum_{\iota \in I_n }  e^n_{\iota}(v_n) \left(\overline{k} e^{n}_{\iota}(v_n)\right) \right)= \\ =\left\{
 	\begin{array}{c l}
 	1_{A_{\theta'}} & \overline{k}=\overline{0} \\
 	0 & \overline{k}\neq\overline{0}
 	\end{array}\right. \text{ where } \overline{k} \in \mathbb{Z}_m \text{ (resp. } \overline{k} \in \mathbb{Z}_n\text{)} .
 	\end{split}
 	\end{equation}
 	If $I = I_m \times I_n$ and $\{e_{\iota}\in A_{\theta'}\}_{\iota \in I}$, $\{e'_{\iota}\in A_{\theta'}\}_{\iota \in I}$ are given by
 	\begin{equation*}\label{nt_e_iota}
 	e_{\iota}= e^m_{\iota_1}(u_m) e^n_{\iota_2}(v_n); \ e'_{\iota}=  e^n_{\iota_2}(v_n)e^m_{\iota_1}(u_m);  \ \iota_1 \in I_m,\ \iota_2 \in I_n,\ \iota = (\iota_1, \iota_2) \in I 
 	\end{equation*}
 	then from \eqref{circ_sum_nt} it follows that
 	\begin{equation*}
 	\sum_{\iota \in I}e'_{\iota}e_{\iota} = \sum_{\iota_2 \in I_n} \left(e^n_{\iota_2}(v_n) \left(\sum_{\iota_1 \in I_m}e^{m}_k(u_m)e^{m}_k(u_m)\right)e^n_{\iota_2}(v_n)\right) = \sum_{\iota_2 \in I_n} \left(e^n_{\iota_2}(v_n) 1e^n_{\iota_2}(v_n)\right)=1.
 	\end{equation*} 
 	If $g = (\overline{p}, \overline{q}) \in \mathbb{Z}_m\times\mathbb{Z}_n$ is such that $\overline{p}\neq \overline{0}$ then from \eqref{circ_sum_nt} it follows that
 	\begin{equation*}
 	\sum_{\iota \in I}e'_{\iota}(ge_{\iota}) = \sum_{\iota_2 \in I_n} \left(e^n_{\iota_2}(v_n) \left(\sum_{\iota_1 \in I_m}e^{m}_{\iota_1}(u_m)\left(\overline{p}e^m_{\iota_1}(u_n)\right)\right)\left(\overline{q}e^n_{\iota_2}(v_n)\right)\right) =
 	\end{equation*}
 	\begin{equation*}
 	= \sum_{\iota_2 \in I_n} e^n_{\iota_2}(v_n) 0\left(\overline{q}e^n_{\iota_2}(v_n)\right)=0.
 	\end{equation*} 
 	If $g = (\overline{p}, \overline{q}) \in \mathbb{Z}_m\times\mathbb{Z}_n$ is such that $\overline{p}= \overline{0}$ and $\overline{q}\neq \overline{0}$ then from \eqref{circ_sum_nt} it follows that
 	\begin{equation*}
 	\sum_{\iota \in I}e'_{\iota}(ge_{\iota}) = \sum_{\iota_2 \in I_n} \left(e^n_{\iota_2}(v_n) \left(\sum_{\iota_1 \in I_m}e^{m}_{\iota_1}(u_m)e^{m}_{\iota_1}k(u_m)\right)\left(\overline{q}e^n_{\iota_2}(v_n)\right)\right) =
 	\end{equation*}
 	\begin{equation*}
 	= \sum_{\iota_2 \in I_n} e^n_{\iota_2}(v_n) 1\left(\overline{q}e^n_{\iota_2}(v_n)\right)= \sum_{\iota_2 \in I_n} e^n_{\iota_2}(v_n) \left(\overline{q}e^n_{\iota_2}(v_n)\right)=0.
 	\end{equation*}
 	From above equations it follows that 
 	\begin{equation}\label{torus_comp}
 	\sum_{\iota\in I, \ g \in  \mathbb{Z}_m\times\mathbb{Z}_n} e'_{\iota}(ge_{\iota})=\left\{
 	\begin{array}{c l}
 	1 & g \in  \mathbb{Z}_m\times\mathbb{Z}_n \text{ is trivial}\\
 	0 & g \in  \mathbb{Z}_m\times\mathbb{Z}_n \text{ is not trivial}
 	\end{array}\right.,
 	\end{equation}
 	i.e. condition 3 of the Definition \ref{fin_def} hold.
 	Otherwise $A_{\theta'} \approx A^{mn}_{\theta}$ as right and left $A_{\theta}$ module, i.e. condition 2 of the Definition \ref{fin_def} hold. So the triple $\left(A_{\theta}, A_{\theta'}, \mathbb{Z}_m\times\mathbb{Z}_n\right)$ is a  noncommutative finite covering projection.

 \end{exm}

 \begin{exm}\label{boring_example}{\it Boring example}.
 	Let $A$ (resp. $G$) be {\it any} unital $C^*$-algebra (resp. finite group.). Let $\widetilde{A}=\oplus_{g\in G}A_g$ where $A_g\approx A$ for any $g \in G$. Let $1_{A_g} \in G$ be the unity of $A_g$. Then $\widetilde{A}$ is a finitely generated left and right $A$-module. Action of $G$ is given by $g_1A_{g_2}=A_{g_1g_2}$. We have
 	\begin{equation*}
 	\sum_{g\in G} 1_{A_g}1_{A_g} = 1_{\widetilde{A}}  ,
 	\end{equation*}
 	\begin{equation}
 	\sum_{i =1, ..., n} 1_{A_g}(g1_{A_g}) = 0 \ \forall g\in G \ (g \ \mathrm{is \ nontrivial}).
 	\end{equation}
 	So a triple  $\left(A, \oplus_{g\in G}A_g, G\right)$ is a finite noncommutative covering projection. This example is boring since it does not reflect properties of $A$ and this projection can be constructed for any finite group.

 \end{exm}
 
 \begin{defn}
 	A ring is said to be {\it irreducible } if it is not a direct sum of more than one nontrivial ring. A finite covering projection $(A, \widetilde{A}, G)$ is said to be {\it irreducible} if both $A$ and $\widetilde{A}$ are irreducible. Otherwise  $(A, \widetilde{A}, G)$ is said to be {\it reducible}.
 \end{defn}
 \begin{rem}
 	Any reducible finite covering projection is boring. Covering projections from examples \ref{circle_fin} - \ref{nt_fin_cov} are irreducible. 
 \end{rem}
 
 \begin{defn}
 	Let
 	\begin{equation*}
 	A=A_0 \xrightarrow{\pi^1} A_1 \xrightarrow{\pi^2} ... \xrightarrow{\pi^n} A_n \xrightarrow{\pi^{n+1}} ...
 	\end{equation*}
 	be a finite or countable sequence of $C^*$-algebras and *-homomorphisms such that for any $i>0$ there is a finite noncommutative covering projection $\left(A_{i-1}, A_i, G_i\right)$. The sequence is said to be {\it composable} if following conditions hold:
 	\begin{enumerate}
 		\item Any composition $\pi_{i_1}\circ ...\circ\pi_{i_0+1}\circ\pi_{i_0}:A_{i_0}\to A_{i_1}$ corresponds to the noncommutative covering projection $\left(A_{i_0}, A_{i_1}, G\left(A_{i_1}|A_{i_0}\right)\right)$;
 		\item There is the natural exact sequence of covering transformation groups
 		\begin{equation*}
 		\{e\}\to G\left(A_{i+2}|A_{i+1}\right) \xrightarrow{\iota} G\left(A_{i+2}|A_{i}\right)\xrightarrow{\pi}G\left(A_{i+1}|A_{i}\right)\to\{e\}
 		\end{equation*}
 		for any $i>0$.
 	\end{enumerate}

 \end{defn}
 \begin{exm}
 	Let
 	\begin{equation*}
 	\mathcal{X} = \mathcal{X}_0 \xleftarrow{}... \xleftarrow{} \mathcal{X}_n \xleftarrow{} ... 
 	\end{equation*}
 	be a finite or countable sequence of second-countable locally compact Hausdorff spaces and regular finitely listed topological covering projections. From the Example \ref{fin_lem} it follows that for any $i$ there is a finite noncommutative covering projection $\left(C\left(\mathcal X_{i-1}\right), C\left(\mathcal X_i\right), G_i\right)$. Since a composition of two regular finitely listed topological covering projections is also regular finitely listed topological covering projection the sequence
 	\begin{equation*}
 	C\left(\mathcal X\right)=C\left(\mathcal X_0\right) \xrightarrow{} C\left(\mathcal X_1\right) \xrightarrow{} ... \xrightarrow{} C\left(\mathcal X_n\right) \xrightarrow{} ...
 	\end{equation*}
 	is composable.
 \end{exm}
 \subsection{Infinite case}\label{bas_constr}
 \paragraph{}
 This section is concerned with a noncommutative generalization of the described in the Section \ref{inf_to} construction. Let
 \begin{equation}\label{seq}
 A=A_0 \xrightarrow{\pi^1} A_1 \xrightarrow{\pi^2} ... \xrightarrow{\pi^n} A_n \xrightarrow{\pi^{n+1}} ...
 \end{equation}
 
 be a sequence of *-homomorphisms which correspond to irreducible noncommutative finite covering projections, and suppose that \eqref{seq} is composable. Let $G_n = G(A_{n}|A)$  be  covering transformations groups, where $n \in \mathbb{N}$. For any $n \in \mathbb{N}$ there is the  natural group epimorphism $h_n:\overline{G} = \varprojlim G_m \to G_n$. 
 \begin{defn}
 	A sequence \eqref{seq} is said to be {\it irreducible} if $A_n$ is irreducible for any $n \in \mathbb{N}^0$.
 \end{defn}
 \begin{empt}
 	
 	Algebras  $\{A_n\}_{n \in \mathbb{N}^0}$  are  finitely generated projective Hilbert $A$-modules with sesquilinear product given by \eqref{fin_form_a}, i.e.
 	
 	\begin{equation}\label{fin_form}
 	\langle a, b \rangle_{A_n} = \frac{1}{|G_n|} \sum_{g \in G_n} g(a^*b)
 	\end{equation}
 	From \eqref{hilb_mod_direct_sum} it follows that there $A_n$ can be decomposed to a direct sum of orthogonal Hilbert $A$-modules
 	\begin{equation}\label{a_p}
 	A_n = A_{n-1} \oplus P_n,
 	\end{equation}
 	i.e. $\langle a, p \rangle_{A_n}=0$ for any $a\in A_{n-1}$, $p \in P_{n}$.
 \end{empt}
 \begin{defn}\label{coh_defn}
 	A sequence $\{a_n \in A_n\}_{n \in \mathbb{N}^0}$  such that following conditions hold:
 	
 	\begin{equation}\label{seq_cond}
 	a_{n+1} = \frac{a_n}{\left|G\left(A_{n+1}|A_n\right)\right|} + p_{n + 1}, \ p_{n + 1} \in P_{n+1};
 	\end{equation}

 	\begin{equation}\label{norm_conv_cond}
 	\text{A sequence } \left\{\left\langle a_n, a_n \right\rangle_{A_n}\in A\right\}_{n \in \mathbb{N}} \text{ is norm  convergent as } n \to \infty
 	\end{equation}
 	
 	is said to be {\it coherent}.
 \end{defn}
 \begin{rem}
 	The condition \eqref{seq_cond} is equivalent to
 	\begin{equation}\label{seq_cond_sum}
 	a_{n} = \sum_{g \in G\left(A_{n+1}|A_n\right)}ga_{n+1}.
 	\end{equation}
 \end{rem}

 \begin{rem}\label{inf_small_rem1}
 	Informally $\frac{a_n}{|G(A_{n+1}|A_n)|}$ means an infinitesimally small coefficient, because $\frac{a_n}{|G(A_{n+N}|A_n)|} \to 0$ as $N \to \infty$. If we denote $P_0 = A$ then $A_n = \bigoplus_{i \in \{0,...,n\}}P_i$, where $P_i$ is finitely generated projective $A$ module. Otherwise for any $n \in\mathbb{N}^0$ there is an inclusion $P_n \subset\sH_{A}$ into the Hilbert space over $A$ (Definition \ref{hilb_product_defn}). From the Definition \ref{coh_defn} it follows that
 	\begin{equation*}
 	a_n = \sum_{k = 0}^{n} \frac{p_{k}}{\left|G \left(A_n|A_k\right)\right|}
 	\end{equation*}
 	where $p_i \in \sH_A$. For any $k \in \mathbb{N}^0$ there is an infinitesimally small element $p^k \in \sH_A$ given by the sequence $\left\{p^k_n\in A\right\}_{n \in \mathbb{N}}$ where 
 	 		\begin{equation}\label{p_nk_eqn}
 	 		p^k_n=\left\{
 	 		\begin{array}{c l}
 	 		0 & n<k\\
 	 		\frac{p_{k}}{\left|G \left(A_n|A_k\right)\right|} & n \ge k
 	 		\end{array}\right..
 	 		\end{equation}
 Otherwise from \eqref{p_nk_eqn} it follows that 
 \begin{equation*}
\left\langle a_n, a_n \right\rangle_{A_n} = \sum_{k = 0}^{\infty}\left(p^k_n\right)^*p^k_n,
 \end{equation*}	 		
\begin{equation*}
\lim_{n \to \infty}\left\langle a_n, a_n \right\rangle_{A_n} = \lim_{n \to \infty}\sum_{k = 0}^{\infty}\left(p^k_n\right)^*p^k_n
\end{equation*}
or
\begin{equation*}
\lim_{n \to \infty}\left\langle a_n, a_n \right\rangle_{A_n} = \sum_{k = 0}^{\infty}\left(p^k\right)^*p^k.
\end{equation*}
where $p^k$ is infinitesimally small for any $k \in \mathbb{N}$. As well as in the Section \ref{nsa_sec} there is a representation by the sum of infinitesimally small elements.	 		
 \end{rem}
 \begin{empt}\label{product}
 	There is an action of $\overline{G}=\varprojlim G_n$ on the linear space $X'$ of coherent sequences arising from actions of $G_n$ on $A_n$. There is the unique sesquilinear $A$-valued product $\langle \cdot , \cdot \rangle_{X'}$ on $X'$ such that $\langle \{a_n\} , \{b_n\} \rangle_{X'} = \mathrm{lim}_{n\to \infty} \langle a_n, b_n \rangle_{A_n}$. If $\mathcal{I} =\left\{x \in X' \ | \ \langle x, x\rangle_{X'} = 0 \right\}$ and  $X''=X'/\mathcal{I} $ then there is a norm on $X''$ given by $\|x\|=\sqrt{\langle x, x\rangle_{X'}}$. Let $\overline{X}_A$ be the norm completion of $X''$.  There is the unique sesquilinear $A$-valued product $\langle \cdot , \cdot \rangle_{\overline{X}_A}$ on $\overline{X}_A$ arising from  $\langle \cdot , \cdot \rangle_{X'}$, so $\overline{X}_A$  is a Hilbert $A$-module. There is an  action  $\overline{G}$ on  $\overline{X}_A$ arising from the action of $\overline{G}$ on $X'$. 
 \end{empt}
 \begin{rem}
 	As well as in the Remark \ref{inf_small_rem1} the inner product can be represented by the sum
 	\begin{equation*}
 	\lim_{n \to \infty}\left\langle a_n, b_n \right\rangle_{A_n} = \sum_{k = 0}^{\infty}\left(p^k\right)^*q^k.
 	\end{equation*}
 	where $p^k, q^k \in \sH_A$ are infinitesimally small for any $k \in \mathbb{N}^0$.
 	
 \end{rem}
 \begin{defn}
 	Any coherent sequence $\{a_n \in A_n\}_{n \in \mathbb{N}^0}$ naturally gives the unique element $\xi \in \overline{X}_A$. We say that $\xi$ is {\it represented by }  $\{a_n \in A_n\}_{n \in \mathbb{N}^0}$, and we will write $\xi = \mathfrak{Rep}\left(\left\{a_n \in A_n\right\}_{n \in \mathbb{N}^0}\right)$. We say that the sequence $\{a_n \in A_n\}_{n \in \mathbb{N}^0}$ is a {\it representative} of $\xi$.
 \end{defn}
 \begin{empt}\label{irr_text}
 	Let $\Lambda= \{a_n \in A_n\}_{n \in \mathbb{N}^0}$ be a coherent sequence. For any $N > 0$ and $b_N \in A_N$ we will define a coherent sequence $b_N\Lambda= \{c_n \in A_n\}_{n \in \mathbb{N}^0}$ given by
 	\begin{equation*}
 	c_n =\left\{
 	\begin{array}{c l}
 	\sum_{g \in G\left(A_N |A_n\right)} \frac{1}{\left|G\left(A_N |A_n\right)\right|}g\left(b_Na_n\right) & n < N\\
 	b_Na_n & n \ge N
 	\end{array}\right. .
 	\end{equation*}
 	whence there is the left action of $A_N$ on  $\overline{X}_A$ arising from the action of $A_N$ on coherent sequences. Similarly there is the right action of $A_N$ on $\overline{X}_A$. 
 	Let $\mathcal{K}\left(\overline{X}_A\right)$ be a $C^*$-algebra of compact operators with left action on $X_A$. For any $N \in \mathbb{N}$ the left (resp. right) action of $A_N$ on $\overline{X}_A$ induces the left (resp. right) action of   $A_N$ on $\mathcal{K}\left(\overline{X}_A\right)$.  If $A \to B\left(H\right)$ is a faithful representation then $\overline{X}_A \otimes_AH$ is a pre-Hilbert space with the scalar product given by
 	\begin{equation*}
 	\left(\xi \otimes x, \eta \otimes y\right)= \left(x, \left\langle \xi, \eta\right\rangle_{\overline{X}_A} y\right).
 	\end{equation*}
 	If $\overline{H}$ is the Hilbert completion of $\overline{X}_A \otimes H$ then $\overline{H}$ is a Hilbert space. For any $n \in \mathbb{N}$ there is the action of $A_n$ on $\overline{H}$ arising from the action of $A_n$ on $\overline{X}_A$. Similarly $\mathcal{K}\left(\overline{X}_A\right)$ acts on $\overline{H}$ and there are natural inclusions $\bigcup_{n\in \mathbb{N}}A_n \subset B\left(\overline{H}\right)$ and $\mathcal{K}\left(\overline{X}_A\right) \subset B\left(\overline{H}\right)$. Moreover there is the natural inclusion of enveloping $W^*$-algebra $\left(\bigcup_{n\in \mathbb{N}}A_n\right)'' \subset B\left(\overline{H}\right)$. 
 	There is a $C^*$-algebra $\overline{A}$ given by
 	\begin{equation*}
 	\overline{A} = \mathcal{K}\left(\overline{X}_A\right)\bigcap \left(\bigcup_{n\in \mathbb{N}}A_n\right)'' \subset  B\left(\overline{H}\right)
 	\end{equation*} There is the natural left (resp. right) action of $A_n$ on both $\mathcal{K}\left(\overline{X}_A\right)$ and $\left(\bigcup_{n\in \mathbb{N}}A_n\right)''$, so there is the left (resp. right) action of $A_n$ of $\overline{A}$.
 	For any $n \in \mathbb{N}$ there is the natural action of $\overline{G}$ on $A_n$ given by $\overline{g}a = h_n(g)a_n$, which induces natural actions of $\overline{G}$ on both $\overline{X}_A$ and $\overline{A}$, such that $g\left(a\xi\right) = \left(ga\right)\left(g\xi\right)$ for any $a \in \overline{A}$ and $\xi \in \overline{X}_A$. According to the Zorn's lemma \cite{spanier:at} there is a maximal irreducible subalgebra $\widetilde{A} \subset \overline{A}$. Let $G\subset \overline{G}$ is a maximal subgroup such that $G\widetilde{A}=\widetilde{A}$. If $g \notin G$ then $\widetilde{A} \bigcap g \widetilde{A} = \{0\}$. From $g^{-1}Gg = G$ it follows that $G$ is a normal subgroup and for any $n \in \mathbb{N}$ there is a homomorphism $h_n|_G : G \to G_n$.
 \end{empt}

  \begin{defn}
  The sequence \eqref{seq} is said to be \textit{faithful} if for any  $n \in \mathbb{N}$ following conditions hold:
  \begin{enumerate}
  \item[(a)] For any  $n \in \mathbb{N}$ the restriction  $h_n|_G$ is a group epimorphism, i.e. $h_n\left(G\right) = G_n$.
  
   \item[(b)] The natural left and right actions of $A_n$ on $\overline{A}$ are faithful.
   
   \end{enumerate}
  \end{defn}
 
 \begin{lem}\label{dir_sum_lem}
 	If the sequence \eqref{sec:Cliff-alg} is faithful  and $J= \overline{\overline{G}/G}$ is a set of representatives of $\overline{G}/G$ then
 	\begin{enumerate}
 		\item[(a)] $\overline{A} = \bigoplus_{\overline{g} \in J} \overline{g}\widetilde{A}$,
 		\item[(b)] Left and right actions of $A_n$ on $\widetilde{A}$ are faithful for any $n \in \mathbb{N}$.
 	\end{enumerate}
 \end{lem}
 \begin{proof}
 a) The algebra $\overline{A}$ is invariant with respect to the $\overline{G}$-action, i.e. $\overline{g} \ \overline{a} \in \overline{A}$ for any $\overline{g} \in \overline{G}$ and $\overline{a} \in \overline{A}$, or $\overline{G} \ \overline{A}=\overline{A}$. If $g \in G$ then $g\widetilde{A}=\widetilde{A}$ for any maximal irreducible subalgebra $\widetilde{A} \subset \overline{A}$. Otherwise if $\overline{g} \in \overline{G} \backslash G$ then $\overline{g}$ transposes irreducible subalgebras, so  $\overline{A} = \bigoplus_{\overline{g} \in J} \overline{g}\widetilde{A}$.
 	\newline
 	b) Consider the right action of $A_n$ on $\overline{A}$. 
 	Let $\mathcal{I} \subset A_n$ be the annulator of $X_A$, i.e.  $\mathcal{I}$ is the maximal ideal such that $\widetilde{A}\mathcal{I} = \{0\}$. Since action of $A_n$ on $\overline{A} = \bigoplus_{\overline{g} \in J}\overline{g}\widetilde{A}$ is faithful we have $\bigcap_{\overline{g} \in J} \mathcal{I}^{\overline{g}}=\{0\}$, where $\mathcal{I}^{\overline{g}}=h_n\left(\overline{g}\right)\mathcal{I}$ is the annulator of $\overline{g} \widetilde{A}$. However since $h_n\left(G\right)=G_n$ an element $h_n\left(\overline{g}\right)$ is trivial for any $\overline{g} \in J$  and $\mathcal{I}^{\overline{g}}= h_n\left(\overline{g}\right) \mathcal{I} = \mathcal{I}$, whence $\mathcal{I}= \bigcap_{\overline{g} \in J} \mathcal{I}^{\overline{g}}=\{0\}$ and the action of $A_n$ on $\widetilde{A}$ is faithful. Similarly we can proof that the left action of $A_n$ on $\widetilde{A}$ is faithful.
 \end{proof}

 \begin{defn}\label{main_defn}
 	Let \eqref{seq} be a composable faithful sequence of irreducible $C^*$-algebras. A Hilbert $A$-module $\overline{X}_A$ is said to be the {\it disconnected module} of the sequence \eqref{seq}.  The $\overline{G}$ is said to be a {\it disconnected group} of the sequence \eqref{seq}. The algebra $\overline{A}= \mathcal{K}\left(\overline{X}_A\right)\bigcap \left(\bigcup_{n\in \mathbb{N}}A_n\right)''$ is said to be the {\it disconnected covering algebra} of the sequence \eqref{seq}.  If  $\widetilde{A} \subset \overline{A}$ is a maximal irreducible subalgebra and $X_A = \widetilde{A}\otimes_{\overline{A}}\overline{X}_A$, then $\widetilde{A}$ is said to be a {\it connected covering algebra} of the sequence \eqref{seq} and $X_A\subset\overline{X}_A$ is said to be a {\it connected module} of sequence \eqref{seq}. A maximal subgroup $G \subset \overline{G}$ such that $G\widetilde{A}=\widetilde{A}$ (or $G  X_A = X_A$) is said a {\it  covering transformation group} of the sequence \eqref{seq}. The group $G$ is a normal subgroup of $\overline{G}$. $X_A$ is a $\widetilde{A}$-$A$ correspondence, i.e. $X_A = _{\widetilde{A}}X_A$. The quadruple $\left(A, \widetilde{A}, _{\widetilde{A}}X_A, G\right)$ is said to be a {\it noncommutative infinite covering projection} of the sequence \eqref{seq}. $A$  is said to be the {\it base algebra of the sequence} \eqref{seq}.
 \end{defn}
 \begin{rem}
 	From the Lemma \ref{dir_sum_lem} all irreducible subalgebras of $\overline{A}$ are isomorphic. Similarly we can say about $G$ and $_{\widetilde{A}}X_A$. So  $\widetilde{A}$, $G$ and $_{\widetilde{A}}X_A$ from the Definition \ref{main_defn} are unique up to isomorphisms.
 \end{rem}
 \begin{lem}\label{sum_lem}
 	Let $\Lambda = \left\{e_{n} \in A_n\right\}_{n \in \mathbb{N}^0}$, $\Lambda' = \left\{e'_{n} \in A_n\right\}_{n \in \mathbb{N}^0}$ be  coherent  sequences such that 
 	\begin{equation*}
 	e_{n} = \sum_{g \in G(A_{n+1}|A_n)}g e_{n+1}; \ e'_{n} = \sum_{g \in G(A_{n+1}|A_n)}g e'_{n+1}; \ e'_ne^*_{n} = \sum_{g \in G(A_{n+1}|A_n)}g e'_{n+1}e^*_{n+1}.
 	\end{equation*}

 	If $\xi = \mathfrak{Rep}(\Lambda)$, $\xi' = \mathfrak{Rep}(\Lambda')$ then following series 
 	
 	\begin{equation*}\label{infsum}
 	\widetilde{a} = \sum_{g \in G} g \xi'\rangle \langle g \xi
 	\end{equation*}
 	is strictly convergent and $\langle \eta e'_0e_0^*, \zeta \rangle_{\overline{X}_A} =\langle \eta, \widetilde{a}\zeta \rangle_{\overline{X}_A}$  for any $\eta,\zeta \in \overline{X}_A$.
 \end{lem}
 \begin{proof}
 	Let $\left\{G^k \subset \overline{G}\right\}_{k\in\mathbb{N}}$ be a $\overline{G}$-covering of the sequence
 	\begin{equation*}\
 	G\left(A_1, A\right) \leftarrow G\left(A_2, A\right) \leftarrow ...
 	\end{equation*}
 	where $\overline{G}=\varprojlim  G\left(A_n, A\right)$.
 	If $\eta, \zeta \in \overline{X}_A$ are given by $\eta = \mathfrak{Rep}\left(\left\{b_n \in A_n\right\}_{n \in \mathbb{N}^0}\right), \ \zeta = \mathfrak{Rep}\left(\left\{c_n \in A_n\right\}_{n \in \mathbb{N}^0}\right)$ then from \ref{product}  it  follows that
 	\begin{equation*}
 	\langle \eta, \xi' \rangle_{\overline{X}_A} = \lim_{n\to \infty} c^*_n e'_n; \ \langle \xi, \zeta \rangle_{\overline{X}_A} = \lim_{n\to \infty} e^*_nb_n; \ 
 	\end{equation*}
 	\begin{equation*}
 	\langle \eta, \xi' \rangle_{\overline{X}_A} \langle\xi, \zeta \rangle_{\overline{X}_A}=  \langle \eta, \left( \xi' \rangle \langle\xi\right) \zeta \rangle_{\overline{X}_A}= \lim_{n\to \infty} c^*_ne'_ne^*_nb_n;
 	\end{equation*}
 	If $a_m \in \mathcal{K}\left(\overline{X}_A\right)$ is given by
 	\begin{equation*}
 	a_m = \sum_{g \in G^m}g \xi'\rangle \langle g \xi.
 	\end{equation*}
 	then 
 	\begin{equation*}
 	\langle \eta, a_m \zeta \rangle_{\overline{X}_A} = \lim_{n \to \infty} c^*_n \left(\sum_{g \in G^m}(h_m(g)(e'_me^*_m))\right)b_n=\lim_{n \to \infty} c^*_n \left(\sum_{g \in G\left(A_m, A\right)}g\left(e'_me^*_m\right)\right)b_n,
 	\end{equation*}
 	whence
 	\begin{equation*}
 	\lim_{m \to \infty}\langle \eta, a_m \zeta \rangle_{\overline{X}_A} =\lim_{m \to \infty}\lim_{n \to \infty} c^*_n \left(\sum_{g \in G\left(A_m, A\right)}g\left(e'_me^*_m\right)\right)b_n = 
 	\end{equation*}
 	\begin{equation*}
 	= \lim_{n \to \infty} c^*_n e'_0e^*_0b_n=\langle \eta e'_0e^*_0, \zeta \rangle_{\overline{X}_A}. 
 	\end{equation*}
 	Form the above equation it follows that the sequence $\left\{a_m\right\}_{m \in \mathbb{N}^0}$ is strictly convergent as $m\to \infty$ and $\langle \eta,\left( \lim_{m \to \infty} a_m\right) \zeta \rangle_{\overline{X}_A}=\langle \eta e'_0e^*_0,\zeta \rangle_{\overline{X}_A}$. Otherwise $\lim_{m \to \infty} a_m = \widetilde{a}$ in the sense of the strict convergence.
 \end{proof}
 \begin{cor}\label{cor_xi1}
 	Let $I$ be a finite or countable set and $\xi_{\iota}=\mathfrak{Rep}\left(\left\{e_{\iota n} \in A_n\right\}_{n \in \mathbb{N}^0}\right)$, $ \xi'_{\iota}=\mathfrak{Rep}\left(\left\{e'_{\iota n} \in A_n\right\}_{n \in \mathbb{N}^0}\right)\in \overline{X}_A$ satisfy conditions of the Lemma \ref{sum_lem}. If $\xi$ and $\xi'$ satisfy following condition
 	\begin{equation}\label{alg}
 	\sum_{\iota} e'_{0,\iota}e^*_{0,\iota} = 1_{M(A)}, \ \forall n \in \mathbb{N}^0
 	\end{equation}
 	in sense of strict topology and then
 	\begin{equation*}
 	\sum_{\iota \in I, g \in \overline{G}} g \xi'_{\iota}\rangle \langle g \xi_{\iota} = 1_{M\left(\mathcal{K}\left(\overline{X}_A\right)\right)}
 	\end{equation*}
 	in sense of strict topology.
 \end{cor}
 \begin{proof}
 	From lemma \ref{sum_lem} it follows that for any $\eta, \zeta \in \overline{X}_A$ following condition hold.
 	\begin{equation}\label{comp}
 	\left\langle \eta, \left(\sum_{g \in G} g \xi_{\iota}\rangle \langle g \xi'_{\iota}\right) \zeta\rangle_{\overline{X}_A}=\langle \eta e^*_{0, \iota}e'_{0, \iota} \ , \ \zeta\right\rangle_{\overline{X}_A}.
 	\end{equation}
 	From follows \eqref{alg}, \eqref{comp} it follows that for any $\eta, \zeta \in \overline{X}_A$ following condition  hold
 	\begin{equation*}
 	\left\langle \eta , \left( \sum_{\iota \in I, \ g \in \overline{G}}g \xi_{\iota}\rangle \langle g \xi_{\iota}\right) \zeta \right\rangle_{\overline{X}_A} =\left\langle \eta \sum_{\iota \in I} e^*_{0\iota}e'_{0,\iota} \ , \ \zeta \rangle_{\overline{X}_A} = \langle \eta , \zeta \right\rangle_{\overline{X}_A} \ ,
 	\end{equation*}
 	i.e.
 	\begin{equation*}
 	\sum_{\iota \in I, g \in \overline{G}} g \xi'_{\iota}\rangle \langle g \xi_{\iota} = 1_{M\left(\mathcal{K}\left(\overline{X}_A\right)\right)}.
 	\end{equation*}
 \end{proof}
 \begin{cor}\label{cor_xi2}
 	In the situation of the corollary \ref{cor_xi1} a linear span of $\left\{g \xi'_{\iota}a\right\}_{g \in \overline{G}, \ \iota \in I, \ a \in A}$ is a dense subspace of $\overline{X}_A$.
 \end{cor}
 \begin{proof}
 	Follows from the Corollary \ref{cor_xi1}.
 \end{proof}
 
\section{Covering projections of spectral triples}
\begin{empt}\label{s_repr}
Let  $\left(\A, H, D\right)$ be a spectral triple.  Similarly to \cite{bram:atricle} we  define a representation of $\pi^1:\A \to B(H^2)$ given by
\begin{equation}
\nonumber \pi^1(a) =  \begin{pmatrix} a & 0\\
[D,a] & a\end{pmatrix}.
\end{equation}
 We can inductively construct  representations $\pi^s: \A \to B\left(H^{2^s}\right)$ for any $s \in \mathbb{N}$. If $\pi^s$ is already constructed then  $\pi^{s+1}: \A \to B\left(H^{2^{s+1}}\right)$ is given by
  \begin{equation}\label{s_diff_repr_equ}
  \pi^{s+1}(a) =  \begin{pmatrix}  \pi^{s}(a) & 0 \\ \left[D,\pi^s(a)\right] &  \pi^s(a)\end{pmatrix}
  \end{equation}
  where we assume diagonal action of $D$ on $H^{2^s}$, i.e.
 \begin{equation*}
 D \begin{pmatrix} x_1\\ ... \\ x_{2^s}
 \end{pmatrix}= \begin{pmatrix} D x_1\\ ... \\ D x_{2^s}
 \end{pmatrix}; \ x_1,..., x_{2^s}\in H.
 \end{equation*}
 \end{empt}

\begin{defn}\label{coh_spec_triple_defn}
Let  
\begin{equation}\label{sp_tr_sec_eqn}
\left\{\left(\A_n, H_n, D_n\right)\right\}_{n\in \mathbb{N}^0}
\end{equation}
 be a sequence of spectral triples and $\left(\A, H, D\right) = \left(\A_0, H_0, D_0\right)$. The sequence is said to be {\it  coherent} if following conditions hold:
\begin{enumerate}
\item There is a sequence of injective *-homomorphisms
\begin{equation}\label{triple_seq}
\A = \A_0 \to \A_1 \to ... \to \A_n \to ... \ ,
\end{equation} 
\item For any $n \in \mathbb{N}$ there is a finite noncommutuative covering projection $\left(A_{n-1}, A_n, G\left(A_n, A_{n-1}\right)\right)$ where $A_n$ is the $C^*$-completion of $\A_n$ and the *-homomorphism $A_{n-1}\to A_n$ is induced by the inclusion $\A_{n-1}\to \A_n$.
\item The sequence of finite noncommutative covering projections 
\begin{equation}\label{inf_seq_trp}
A = A_0 \to A_1 \to ... \to A_n \to ... \ ,
\end{equation} 
is composable.
\item If $g \in G\left(A_n,A\right)$ then $g\A_n = \A_n$ , and $\A_n^{G\left(A_n, A_{m}\right)}= \A_m$ for any $m, n \in \mathbb{N}^0$ such that $n > m$.
\item The $\A_n$-module $\sH^\infty_n= \bigcap_{k\in\bN} \Dom D_n^k$ is given by $\sH^\infty_n = \A_n \otimes_{\A} \sH^\infty$ where $\sH^\infty=\bigcap_{k\in\bN} \Dom D^k\subset H$. So $H_n = A_n \otimes H$ and the scalar  product on  $H_n$ is given by 
\begin{equation}\label{hilb_scalar_correct}
\left(a \otimes \xi, b \otimes \eta\right)= \left(\xi, \langle a,b\rangle_{A_n}\eta\right); \ \forall a, b \in A_n, \ \forall \xi, \eta \in H_0.
\end{equation}
From the above expressions it follows that
\begin{itemize}
\item The space $\sH^\infty_n$ is given by $\sH^\infty_n = \A_n \otimes_{\A_m} \sH^\infty_m$ for any $n > m$,
\item Since $\sH^\infty_m = \A_m \otimes_{\A_m}\sH^\infty_m$ and $\A_m \subset \A_n$ there is the natural inclusion $\sH^\infty_m \subset \sH^\infty_n$ and the action of $G\left(A_n, A_m\right)$ on $\sH^\infty_n$ for any $n > m$.
\item If $g \in G\left(A_n,A\right)$ then $g\sH^\infty_n = \sH^\infty_n$ , and $\left(\sH^\infty_n\right)^{G\left(A_n, A_{m}\right)}= \sH^\infty_m$ for any $m, n \in \mathbb{N}^0$ such that $n > m$
\end{itemize}
\item  For any $g \in G\left(A_n,A\right)$ and $\xi \in \sH^\infty_n$  following conditions hold:
\begin{equation*}
g \left(D_n\xi\right)= D_n\left(g\xi\right); \ \forall \xi \in \sH^\infty_n,
\end{equation*}
\begin{equation*}
D_n|_{\sH^\infty_m}= D_m;~\forall n > m.
\end{equation*}

\end{enumerate}

\end{defn}
\begin{rem}
	From the condition 6 of the Definition \ref{coh_spec_triple_defn} it follows that if $n > m$ then
	\begin{equation*}
	D_n \left(1_{A_n} \otimes \xi\right) = 1_{A_n} \otimes D_m \xi; \ \ \forall \xi \in \mathrm{Dom}\left(D_m\right);~\forall n > m.
	\end{equation*}
where tensor product means that $H_n = A_n \otimes_{A_m} H_m$. From this property it follows that $\mathrm{Dom}\left(D_m\right)\subset \mathrm{Dom}\left(D_n\right)$ and $D_n|_{\mathrm{Dom}\left(D_m\right)} = D_m$	
\end{rem}

\begin{empt}

	Let denote $\left(\A, D, H\right) = \left(\A_0, D_0, H_0\right)$ and $\sH^\infty = \sH^\infty_0$. If $\sH = A\otimes_{\A}\sH^\infty$ then from \cite{bass} it follows that $\sH$ is a projective finitely generated $A$-module. From \cite{frank:frames} it follows that $\sH$ is a finitely generated Hilbert $A$-module. So  $\sH_n = A_n \otimes_A \sH$ is a finitely generated Hilbert $A$-module with $A$ valued product given by
	\begin{equation*}
	\left\langle\xi \otimes a, \eta \otimes b\right\rangle_{\sH_n} = \left\langle\xi, \left\langle a, b \right\rangle_{A_n}\eta \right\rangle_{\sH}
	\end{equation*}
\end{empt}
\begin{empt}\label{sp_tr_cov_constr}	
 On the algebraic tensor product $\overline{X}_A \otimes_A H$ there is a $\mathbb{C}$-valued product $\left(\cdot, \cdot\right)$  given by
	
	\begin{equation}\label{hilbert_cov_product_cov}
\left(\mu \otimes \xi, \nu \otimes\eta\right)= \left( \xi, \left\langle \mu, \nu \right\rangle_{\overline{X}_A} \eta\right).
	\end{equation} 
Denote by $\overline{H}$ the Hilbert completion of the $\overline{X}_A \otimes_A H$ and denote by  $\widetilde{H}$ the Hilbert completion of $_{\widetilde{A}}X_A \otimes_A H$. From $_{\widetilde{A}}X_A \subset \overline{X}_A$ it follows the inclusion $\widetilde{H}\subset\overline{H}$.
\end{empt}

\begin{defn}\label{hilbert_coh}
	A sequence $\left\{\xi_n \in H_n\right\}_{n \in \mathbb{N}^0}$ is said to be {\it coherent} in $\overline{H}$ (or $\overline{H}$-\textit{coherent}) if following conditions hold:
	\begin{enumerate}
		\item $\xi_n = 
		\sum_{g \in G\left(A_{n+1}\ | \ A_n\right)}g\xi_{n+1}$
		\item The sequence $\left\{\left(\xi_n, \xi_n\right)\in \mathbb{R}\right\}_{n \in \mathbb{N}}$ is  convergent. 	\end{enumerate}
\end{defn}

\begin{empt}\label{h_sec_constr}
	If $\left\{a_n \in A_n \right\}_{n \in \mathbb{N}^0}$ is a coherent sequence then for any $\xi \in H$ the sequence $\left\{a_n \otimes \xi \in H_n \right\}_{n \in \mathbb{N}^0}$ is coherent in $\overline{H}$. So any $\overline{H}$-coherent sequence $\left\{\xi_n \in H_n \right\}_{n \in \mathbb{N}^0}$ corresponds  to a functional on $\overline{X}_A \otimes_A H$ given by
	\begin{equation*}
	\left\{a_n \otimes \xi  \right\} \mapsto \lim_{n \to \infty}\left(a_n \otimes \xi, \xi_n\right).
	\end{equation*}
	The functional can be uniquely extended to $\overline{H}$ and from the Riesz representation theorem it follows the existence  of the unique $\overline{\xi} \in \widetilde{H}$ which corresponds to the functional.

\end{empt}
\begin{defn}\label{h_repr_defn}
	In the situation \ref{h_sec_constr} we say that $\overline{\xi}$ is \textit{$\overline{H}$-represented} by the sequence $\left\{\xi_n  \right\}$ and we will write $\overline{\xi} = \mathfrak{Rep}_{H}\left(\left\{\xi_n  \right\}\right)$.
\end{defn}
\begin{empt}
 	Now we would like to define the unbounded Dirac operator $\widetilde{D}$ on $\widetilde{H}$. It is naturally to define $\widetilde{D}$ on coherent sequences  $\left\{\xi_n \in H_n\right\}_{n \in \mathbb{N}^0}$ such that
 	\begin{equation*}
 	\widetilde{D}~\mathfrak{Rep}_{H}\left(\left\{\xi_n  \right\}\right) = \mathfrak{Rep}_{H}\left(\left\{D_n\xi_n  \right\}\right).
 	\end{equation*} 
 	and then obtain closure of this operator.
 	But the sequence  $\left\{D_n\xi_n  \right\}$ should not be always coherent, and the general definition of $	\widetilde{D}$ can be very difficult. However the situation can be simplified in  the special case of local covering projections which are described below.
 \end{empt}
\begin{empt}\label{local_constr}
 Let $\left(\A, H, D\right)$ (resp. $\left(\widetilde{\A}, \widetilde{H}, \widetilde{D}\right)$) be a spectral triple, let $A$ (resp. $\widetilde{A}$) be the $C^*$-completion of $\A$ (resp. $\widetilde{\A}$). Suppose that there is a finite noncommutative covering projection $\left(A, \widetilde{A}, G\right)$ such that $G\widetilde{\A} = \widetilde{\A}$. Suppose that there are a subspace $\widehat{H}\subset\widetilde{H}$ such that $\widetilde{H}=\bigoplus_{g \in G}g\widehat{H}$, and there is an isomorphism of Hilbert spaces $\varphi: \widehat{H} \to H$ given by  $\xi \mapsto \sum_{g \in G} g \xi$.  
\end{empt}
\begin{defn}\label{local_defn}
	Let us consider the situation \ref{local_constr}, and let $\mathfrak{A}= \left(\A, H, D\right)$, $\widetilde{\mathfrak{A}}=\left(\widetilde{\A}, \widetilde{H}, \widetilde{D}\right)$, $\mathfrak{B}=\left(A, \widetilde{A}, G\right)$. The triple $\left(\mathfrak{A}, \widetilde{\mathfrak{A}}, \mathfrak{B}\right)$ is said to be a \textit{local covering projection of spectral triples} if following conditions hold:
	 	\begin{enumerate}
	 		\item[(a)] $\mathrm{Dom}~\widetilde{D} \bigcap \widehat{H} = \varphi^{-1}\left(\mathrm{Dom}~D \right)$.
	 		\item[(b)] $\widetilde{D}\left(\mathrm{Dom}~\widetilde{D} \bigcap \widehat{H}\right) \subset  \widehat{H}$.
	 		\item[(c)] $D\left(\varphi\left(\xi\right)\right)=\varphi\left(\widetilde{D}\xi\right)$ for any $\xi \in \widehat{H} \bigcap \mathrm{Dom}~D$.
	 	\end{enumerate}
\end{defn}
\begin{rem}
	The meaning of the "local" term is explained in the Remark \ref{locality_explanation_rem}. 
\end{rem}

\begin{lem}
	If $\left(\mathfrak{A}, \widetilde{\mathfrak{A}}, \mathfrak{B}\right)$ is  local covering projection of spectral triples then
	\begin{equation*}
	\ncint \widetilde{D} = \left|G\right|\ncint D.
	\end{equation*}
\end{lem}\label{ncint_lem}
	\begin{proof}
	Let $\widehat{D} = \widetilde{D}|_{\widehat{H}}$. Operator $\widehat{D}$ can be regarded as operator $\mathrm{Dom}~\widetilde{D} \to \widetilde{H}$ and as  $\widehat{H}\bigcap\mathrm{Dom}~\widetilde{D} \to \widehat{H}$. From the diagram
	\newline
 \begin{tikzpicture}
 \matrix (m) [matrix of math nodes,row sep=3em,column sep=4em,minimum width=2em]
 {
 	\widehat{H} \bigcap \mathrm{Dom}~\widetilde{D}    & \widehat{H}   \\
 	\mathrm{Dom}~D   & H    \\};
 \path[-stealth]
 (m-1-1) edge node [above] {$\widehat{D}$} (m-1-2)
 (m-2-1) edge node [above] {$D$} (m-2-2)
 (m-1-1) edge node [left]  {$\varphi$} (m-2-1)
 (m-1-2) edge node [right] {$\varphi$} (m-2-2);
 \end{tikzpicture}
 \newline
 it follows that operator $\widehat{D}$ is measurable and $\ncint \widehat{D} = \ncint D$. If $g\widehat{D}$ is given by $\xi \mapsto g \widehat{D} g^{-1}\xi$ then for any $g \in G$ we have $\ncint \widehat{D}= \ncint g\widehat{D}$. From $\widetilde{D}= \sum_{g \in G} g \widehat{D}$ it follows that
 \begin{equation*}
 \ncint \widetilde{D}= \ncint \sum_{g \in G} g \widehat{D} = \sum_{g \in G} \ncint g \widehat{D} = \left|G\right| \ncint D.
 \end{equation*}
	\end{proof}
\begin{rem}
	It is well known that if $\widetilde{M}\to M$ is an $m$-fold covering of Riemannian manifold  then
	\begin{equation}\label{comm_cov_int_eqn}
\int_{\widetilde{M}}	\sqrt{\mathrm{det}g(x)}dx^1 \wedge...\wedge dx^n=m\int_{M}	\sqrt{\mathrm{det}g(x)}dx^1 \wedge...\wedge dx^n
	\end{equation}
 Otherwise from the example \ref{comm_integ_exm} that in the noncommutative case the noncommutative integral of the Dirac operator $\slashed D$ is proportional to $\int_{M}	\sqrt{\mathrm{det}g(x)}dx^1 \wedge...\wedge dx^n$. Thus the Lemma \ref{ncint_lem} is the noncommutative generalization of the Equation \eqref{comm_cov_int_eqn}.
\end{rem}
\begin{empt}\label{local_sp_tr_constr}
Suppose that the coherent  sequence of spectral triples \eqref{sp_tr_sec_eqn} is such that if $\mathfrak{A}=\left(\A, H, D\right)$, $\mathfrak{A}_n=\left(\A_n, H_n, D_n\right)$ and $\mathfrak{B}_n = \left(A, A_n, G\left(A_n| A\right)\right)$ then the triple $\left(\mathfrak{A}, \mathfrak{A}_n, \mathfrak{B}_n\right)$ is a local covering projection of spectral triples for any $n \in \mathbb{N}$. Let $\widehat{H}^n$ be such that $H^n = \bigoplus_{g \in G\left(A_n,A\right)}g \widehat{H}^n$ and $\varphi^n : \widehat{H}^n \to H$ is the isomorphism given by $\xi \mapsto \sum_{g \in G\left(A_n,A\right)}g\xi$.
\end{empt}
  \begin{defn}\label{local_sec_thiples_defn}
  Let us consider the situation \ref{local_sp_tr_constr}. The  coherent  sequence of spectral triples \eqref{sp_tr_sec_eqn} is said to be \textit{local} if for any $n \in \mathbb{N}$ and $\xi \in \widehat{H}^{n}$ following conditions hold:
  \begin{enumerate}
 \item[(a)] 
  \begin{equation*}
\sum_{g \in G\left(A_{n}|A_{n-1}\right)} ~g\xi \in \widehat{H}^{n -1}.
  \end{equation*}
 \item[(b)] If $\left\{\xi_n \in H_n\right\}_{n \in \mathbb{N}}$ is such that $\xi_n \in \widehat{H}^n$ then $\mathfrak{Rep}_{\overline{H}}\left(\left\{\xi_n \in H_n\right\}\right)\in \widetilde{H}$.
 
  \end{enumerate}
    \end{defn}
    \begin{empt}\label{h_tilde_constr}
   If  coherent  sequence of spectral triples \eqref{sp_tr_sec_eqn} is local then for any $n \in \mathbb{N}$ there is the natural isomorphism $\psi_n:\widehat{H}^{n} \to \widehat{H}^{n-1}$ given by
  \begin{equation*}
\psi_n(\xi)=\sum_{g \in G\left(A_{n}|A_{n-1}\right)} ~g\xi
  \end{equation*}
  and a following condition holds
  \begin{equation*}
 \varphi_n = \psi_1 \circ ... \circ \psi_n.
  \end{equation*} 
 An $\overline{H}$-coherent sequence $\left\{\xi_n = H_n\right\}$ is said to be {\it special} if $\xi_n \in \widehat{H}^n$ for any $n \in \mathbb{N}$. If the sequence is special then there is $\xi \in H$ such that $\xi_n= \varphi^{-1}_n\left(\xi\right)$. If $\Xi$ is the set of special sequences then denote by $\widehat{\overline{H}}$ the Hilbert completion of the $\mathbb{C}$-linear span of $\mathfrak{Rep}_{\overline{H}}\left(\Xi\right)$. There is the unique isomorphism $\widehat{\varphi}:\widehat{\overline{H}} \to H$ which is the extension of the given by $\mathfrak{Rep}_{\overline{H}}\left(\left\{\xi_n\right\}\right) \mapsto \xi_0$ map.
    \end{empt}
\begin{defn}
   	Suppose that the coherent  sequence \eqref{sp_tr_sec_eqn} is local.	A $\overline{H}$-coherent sequence $\left\{\xi_n \in H_n\right\}_{n \in \mathbb{N}^0}$ is said to be {\it special} if $\xi_n \in \widehat{H}^n$ for any $n \in \mathbb{N}$. 
   \end{defn}
    \begin{lem}\label{local_direct_sum_lem}
 	Suppose that the coherent  sequence \eqref{sp_tr_sec_eqn} is local. Let $G_n = G\left(A_n|A\right)$ and $\overline{G}=\varprojlim G_n$. If $H=H_0$ is a separable Hilbert space then following condition hold
  \begin{equation*}
\overline{H} = \bigoplus_{g \in \overline{G}} g \widehat{\overline{H}}.
  \end{equation*}
    \end{lem}

  	 \begin{proof}
  	 Let $\left\{e_\iota\right\}_{\iota \in I}\subset H$ be a countable orthonormal basis of $H$, and let $\widehat{H}_\iota = \mathbb{C} e_\iota  \subset H$ be a generated by $e_\iota$ one dimensional subspace. We say that a $\widetilde{H}$-coherent sequence $\left\{\xi_n\right\}$ is $\iota$-\textit{special} if $\xi_n \in H^\iota_n = \bigoplus_{g \in G_n}\mathbb{C}e^n_\iota \subset H_n$ where $e^n_\iota = \varphi^{-1}_n\left(e_\iota\right)$. Any $\overline{H}$ coherent $\Lambda$ sequence can be decomposed $\Lambda = \sum_{\iota \in I}\Lambda_\iota$ into $\iota$-special sequences, so $\overline{H}$ can be decomposed $\overline{H}= \bigoplus_{\iota \in I} \overline{H}_\iota$ where $\overline{H}_\iota$ is generated by $\iota$-special $\overline{H}$-coherent sequences.	 If $\xi, \eta, \zeta \in \overline{H}_\iota$ are given by $\xi = \mathfrak{Rep}_{\overline{H}}\left(\left\{\xi_n \in H_n\right\}_{n \in \mathbb{N}^0}\right),~\eta = \mathfrak{Rep}_{\overline{H}}\left(\left\{\eta_n \in H_n\right\}_{n \in \mathbb{N}^0}\right), \ \zeta = \mathfrak{Rep}_{\overline{H}}\left(\left\{\zeta_n \in H_n\right\}_{n \in \mathbb{N}^0}\right)$ then 
  	 	\begin{equation*}
  	 		\left(\eta, \xi \right) = \lim_{n\to \infty}\left(\eta_n, \xi_n \right); 
  	 	\end{equation*}
  	 	\begin{equation*}
  	 			\left( \eta, \xi \right) 	\left(\xi, \zeta \right)=  \langle \eta, \left( \xi \right) 	\left(\xi\right) \zeta \rangle_{\overline{X}_A}= \lim_{n\to \infty} \left(\eta_n,\xi_n\right)\left(\xi_n,\zeta_n\right);
  	 	\end{equation*}
  	 	 	If $\left\{G^k \subset \overline{G}\right\}_{k\in\mathbb{N}}$ is a $\overline{G}$-covering of the sequence $\{G_1 \leftarrow G_2 \leftarrow ...\}$, $\overline{e}_\iota = \mathfrak{Rep}_{\overline{H}}\left(\left\{e^n_\iota \right\}\right)$ $a_m \in B\left(\overline{H}_\iota\right)$ be given by
  	 	  \begin{equation*}
  	 	  a_m = \sum_{g \in G^m} g\overline{e}_\iota \left)\right(g\overline{e}_\iota.
  	 	  \end{equation*}
  	 	then 
  	 	\begin{equation*}
  	 		\left( \eta, a_m \zeta \right) = \lim_{n \to \infty} \left(\eta_n, \left(\sum_{g \in G^m} h_n(g)\overline{e}_\iota\left)\right(h_n(g)\overline{e}_\iota\right)\zeta_n\right)
  	 	\end{equation*}
  	
    	 	whence
  	 	\begin{equation*}
  	 		\lim_{m \to \infty}\langle \eta, a_m \zeta \rangle_{\overline{X}_A} =\lim_{m \to \infty} \lim_{n \to \infty} \left(\eta_n, \left(\sum_{g \in G^m} h_n(g)\overline{e}_\iota\left)\right(h_n(g)\overline{e}_\iota\right)\zeta_n\right) = 
  	 	\end{equation*}
  	 	\begin{equation*}
  	 		= \lim_{n \to \infty} \left(\eta_n, \left(\sum_{g \in G^n} h_n(g)\overline{e}_\iota\left)\right(h_n(g)\overline{e}_\iota\right)\zeta_n\right)= \lim_{n \to \infty} \left(\eta_n, \left(\sum_{g \in G_n} ge^n_\iota\left)\right(ge^n_\iota\right)\zeta_n\right)= 
  	 	\end{equation*}
  	 	\begin{equation*}
  	 	\lim_{n \to \infty} \left(\eta_n, \zeta_n\right) = \left(\eta,\zeta\right)
  	 	\end{equation*}
  	 	Form the above equation it follows that the sequence $\left\{a_m\right\}_{m \in \mathbb{N}^0}$ is weakly convergent as $m\to \infty$ and $\lim_{m \to \infty} a_m = 1_{B\left(\overline{H}_\iota\right)}$, i.e.
  	 	\begin{equation*}
  	 	\sum_{g \in \overline{G}} g \overline{e}_\iota \left)\right( g \overline{e}_\iota = 1_{B\left(\overline{H}_\iota\right)}.
  	 	\end{equation*}
  	 \end{proof}
  	 So $\left\{g\overline{e}_\iota\right\}_{g \in \overline{G}}\subset \overline{H}_\iota$ is an orthonormal basis of  $\overline{H}_\iota$. From $\overline{H}= \bigoplus_{\iota \in I} \overline{H}_\iota$ it follows that $\left\{g\overline{e}_\iota\right\}_{g \in \overline{G},~\iota \in I}\subset \overline{H}$ is an orthonormal basis of  $\overline{H}_\iota$, therefore
  	 \begin{equation*}
  	 \overline{H} = \bigoplus_{g \in \overline{G},~\iota \in I}g\overline{H}_\iota= \bigoplus_{g \in \overline{G}}g \widehat{\overline{H}}.
  	 \end{equation*}
 \begin{cor}
 If $G$ is a covering transformation group of the sequence \eqref{inf_seq_trp} then $\widetilde{H} = \bigoplus_{g \in G}g\widehat{\overline{H}}$.
 \end{cor}
 \begin{proof}
 Follows from the condition (b) of the Definition \ref{local_sec_thiples_defn} and the Lemma \ref{local_direct_sum_lem}.
 \end{proof}
\begin{empt}\label{local_dirac_constr}
If a coherent  sequence of spectral triples \eqref{sp_tr_sec_eqn} is local  then there is a densely defined unbounded operator $\widetilde{D}$ on $\widetilde{H}$ given by
\begin{equation*}
\mathrm{Dom}~\widetilde{D}= \bigoplus_{g \in G}g \widehat{\varphi}^{-1}\left(\mathrm{Dom}~D\right)
\end{equation*}
\begin{equation*}
\widetilde{D}\left( \sum_{g \in G}g \xi_g\right) = \sum_{g \in G} g \widehat{\varphi}^{-1}\left(D\widehat{\varphi}\left(\xi_g\right)\right)
\end{equation*}
where $\xi_g \in \widehat{\overline{H}}$ for any $g \in G$.
\end{empt}	
\begin{defn}
If  coherent  sequence of spectral triples \eqref{sp_tr_sec_eqn} is local then defined in \ref{local_dirac_constr} operator $\widetilde{D}$ is said to be the \textit{Dirac operator} of the sequence \eqref{sp_tr_sec_eqn}. 
\end{defn}
\begin{thm}\label{sa_thm}\cite{reed_simon:mp_1}
Let $T$ be an unbounded symmetric operator on a Hilbert space $H$. Then the following are equivalent:
\begin{enumerate}
\item[(a)] $T$ is  self-adjoint,
\item[(b)] $T$ is closed and $\mathrm{ker}\left(T^*\pm i\right)= \{0\}$,
\item[(c)] $\mathrm{ran}\left(T\pm i\right)=H$.
\end{enumerate}
\end{thm}
\begin{cor}
If  coherent  sequence of spectral triples \eqref{sp_tr_sec_eqn} is local the Dirac operator is self-adjoint.
\end{cor}
\begin{proof}
Since Dirac operator $D$ of the spectral triple $\left(\A, H, D\right)$ is self-adjoint it satisfies to the Theorem \ref{sa_thm}. Form \ref{local_dirac_constr} it follows that $\widetilde{D}$ is a direct sum of infinite copies of $D$. So $\widetilde{D}$ satisfies to the Theorem \ref{sa_thm}. 
\end{proof}
\begin{empt}
Let $\left\{a_n \in A_n\right\}_{n \in \mathbb{N}^0}$ be a coherent sequence such that $a_n \in \A_n$. For any $s, n \in \mathbb{N}$ there is the $s$-times differentiable representation $\pi^s_n: \A_n \to B\left(H^{2^s}_n\right)$.
\end{empt}
\begin{defn}
	A sequence $\left\{a_n\right\}_{n \in \mathbb{N}}$ such that $a_n \in \A_n$
	is said to be {\it $s$-times differentiable} if the sequence
		\begin{equation*}
		\left\{\frac{1}{\left|G\left(A_{n}\ | \ A_0\right)\right|}\sum_{g \in G\left(A_{n} \ | \ A_0\right)}g\left(\left(\pi^s\left(a_{n}\right)\right)^*\pi^s\left(a_{n}\right)\right)\in B\left(H^{2^s}\right)\right\}_{n \in \mathbb{N}}
		\end{equation*}
		is norm convergent as $n \to \infty$.
	For any $s$-times differentiable  coherent sequence we will define a norm $\left\|\cdot\right\|_s$ given by
	\begin{equation*}
	\begin{split}
	\left\|\left\{a_n \in \A_n\right\}_{n \in \mathbb{N}}\right\|_s= \sqrt{\lim_{n \to \infty}\frac{1}{\left|G\left(A_{n}\ | \ A_0\right)\right|}\sum_{g \in G\left(A_{n} \ | \ A_0\right)}g\left(\left(\pi^s\left(a_{n}\right)\right)^*\pi^s\left(a_{n}\right)\right)}.
	\end{split}
	\end{equation*}
\end{defn} 
\begin{defn}\label{smooth_coh_defn}
	A coherent sequence $\left\{a_n \in \A_n\right\}_{n \in \mathbb{N}}$ is said to be {\it smooth} if it is $s$-times differentiable for any $s \in \mathbb{N}$.
\end{defn}
\begin{empt}\label{smooth_smooth}
	If both $\left\{a_n\right\}_{n \in \mathbb{N}}$, $\left\{b_n\right\}_{n \in \mathbb{N}}$ are smooth coherent sequences then $\left\langle\mathfrak{Rep}\left(\left\{a_n\right\}\right), \mathfrak{Rep}\left(\left\{b_n\right\}\right)\right\rangle_{X_A}\in \A$. If $\left\{a_n\right\}_{n \in \mathbb{N}}$ is a smooth coherent sequence and $b \in \A$ then the coherent sequence $\left\{a_n\right\}b$ given by 
	\begin{equation*}
\left\{a_n\right\}b= \left\{a_nb\right\}
	\end{equation*}
	is smooth.
\end{empt}\
\begin{defn}\label{smooth_mod_defn}
	Let $\Xi$ be a set of all smooth coherent sequences and $X' \subset \overline{X}_A$ be a $\mathbb{C}$-linear span of $\mathfrak{Rep}\left(\Xi\right)$. The completion of $X'$ with respect to seminorms $\left\|\cdot\right\|_s$  is said to be the {\it disconnected smooth} module of the coherent sequence $\left\{\left(\A_n, D_n, H_n\right)\right\}_{n\in \mathbb{N}^0}$. The disconnected smooth module will be denoted by $\overline{X}^\infty_{\A}$. There is the Fr\'echet topology on $\overline{X}_{\A}^\infty$ induced by seminorms $\left\|\cdot\right\|_s$. 	There is the natural inclusion $\overline{X}_{\A}^\infty \subset  \overline{X}_A$. The intersection $\overline{X}^\infty_{\A} \bigcap ~ _{\widetilde{A}}X_A$ of the disconnected smooth module and connected module (Definition \ref{main_defn}) is said to be the \textit{connected smooth module} which will be denoted by $X^\infty_{\A}$.
\end{defn}

\begin{defn}\label{comp_smooth_defn}
	An element $\kappa \in \mathcal{K}\left(_{\widetilde{A}}X_A\right)$ is said to be {\it smooth} if following conditions hold:
	\begin{enumerate}
		\item $\kappa X_{\A}^\infty \subset X_{\A}^\infty$.
		\item For any $s \in \mathbb{N}$
		\begin{equation*}
	\left\|\kappa\right\|_s= \sup_{\xi \in  X_{\A}^\infty~\&~ \left\|\xi\right\|_s=1}\left\|\kappa\xi\right\|_s < \infty.
		\end{equation*}
Denote by $\mathcal{K}'$ the algebra of smooth operators.		
	\end{enumerate}
\end{defn}
\begin{empt}
	The  $\mathcal{K}'$ is a Fr\'echet algebra  induced by seminorms $\left\|\cdot\right\|_s$. From \ref{smooth_smooth} it follows that if $\xi, \eta \in X_{\A}^\infty$ then $\xi \rangle\langle \eta \in \mathcal{K}'$  
\end{empt}
\begin{defn}\label{smoothly_comp_sub_defn}
	If $\xi, \eta \in X_{\A}^\infty$ then the operator $\xi \rangle\langle \eta$ is said to be a {\it rank-one smooth}. The completion in the Fr\'echet topology of the linear span of rank-one smooth operators  is said to be the {\it smoothly compact subalgebra}. Denote by $\mathcal{K}^\infty\left(X_{\A}^\infty\right)\subset\mathcal{K}\left(_{\widetilde{A}}X_A\right)$ the smoothly compact subalgebra.
\end{defn}
\begin{defn}
	If $\left(A, \widetilde{A}, _{\widetilde{A}}X_A, G\right)$ is  a noncommutative infinite covering projection of the sequence \eqref{inf_seq_trp} then from the definition \ref{main_defn} it follows that $\widetilde{A} \subset \mathcal{K}\left(_{\widetilde{A}}X_A\right)$. The algebra $\widetilde{\A} = \mathcal{K}^\infty\left(X_{\A}^\infty\right) \bigcap \widetilde{A}$ is said to be the {\it smooth covering algebra}. The algebra $\widetilde{\A}$ is a Fr\'echet algebra with a topology  induced by seminorms $\left\|\cdot\right\|_s$.
\end{defn}
\begin{defn}
	A local coherent sequence $\left\{\left(\A_n, D_n, H_n\right)\right\}_{n\in \mathbb{N}^0}$ of spectral triples is said to be {\it regular} if $\widetilde{\A}$ is a dense subalgebra of $\widetilde{A}$ with respect to $C^*$-norm of $\widetilde{A}$.
\end{defn}

\begin{defn}\label{inv_lim_defn}
If	$\left\{\left(\A_n, D_n, H_n\right)\right\}_{n\in \mathbb{N}^0}$ be a regular local coherent sequence spectral triples then the triple $\left(\widetilde{\A}, \widetilde{H}, \widetilde{D} \right)$ is said to be the {\it inverse limit} of $\left\{\left(\A_n, D_n, H_n\right)\right\}_{n\in \mathbb{N}^0}$.
\end{defn}
\begin{rem}
The inverse limit from the Definition \ref{inv_lim_defn} is not an inverse limit in the strict sense of the category theory, it rather looks like an inverse limit.
\end{rem}

\section{Covering projections of commutative spectral triples}\label{comm_case_sec}
\subsection{Covering projections of topological spaces}
\paragraph{} This section supplies a purely algebraic  analog of the topological construction given by the Section \ref{inf_to}. Let 
\begin{equation}\label{top_inv_limit1}
\mathcal{X} = \mathcal{X}_0 \xleftarrow{}... \xleftarrow{} \mathcal{X}_n \xleftarrow{} ... 
\end{equation}

be a sequence of finitely listed regular covering projections such that $\mathcal{X}$ is a locally compact  second-countable  Hausdorff topological space, and $\mathcal{X}_n$ is connected for any $n \in \mathbb{N}$. Let $\left\{G_n =G\left(\mathcal{X}_n|\mathcal{X}\right)\right\}_{n \in \mathbb{N}}$ be the set of groups of covering transformations. Let  $\overline{\mathcal{X}}$ (resp. $\overline{G}$) be a space (resp. a group) described in the Section \ref{inf_to}. There is a natural (disconnected) covering projection $\overline{\pi}: \overline{\mathcal{X}} \to \mathcal{X} $. 
Let $\widetilde{\mathcal{X}} \subset \overline{\mathcal{X}}$ be a connected component. According to the Section \ref{inf_to} there is a normal subgroup $G \subset \overline{G}$ such that $G\widetilde{\mathcal{X}}=\widetilde{\mathcal{X}}$ and a subset $J \subset \overline{G}$ of $\overline{G}/G$ representatives such that $\overline{\mathcal{X}}= \bigsqcup_{g\in J}g\widetilde{\mathcal{X}}$. The restriction $\widetilde{\pi}=\overline{\pi}|_{\widetilde{\mathcal{X}}}$ is a regular covering projection $\widetilde{\pi}: \widetilde{\mathcal{X}}\to\mathcal{X}$ and $\mathcal{X} \approx \widetilde{\mathcal{X}}/G$.
\begin{defn}
Let $\pi:\widetilde{\mathcal{X}} \to \mathcal{X}$ be a regular topological covering projection such that the group $G = G\left(\widetilde{\mathcal{X}} | \mathcal{X}\right)$ of covering transformations is finite or countable. Then there is a Hilbert $C_0(\mathcal{X})$-module
\begin{equation}\label{hilb_cond_comm}
 \mathscr{L}^2\left(\widetilde{\mathcal{X}}_{\mathcal{X}}\right) = \left\{ \varphi \in C_b(\widetilde{\mathcal{X}}) \ | \text{ if } \phi\left(x\right) =\sum_{\widetilde{x}\in \widetilde{\pi}^{-1}(x)} \varphi^*(\widetilde{x})\varphi(\widetilde{x}) \text{ then } \phi \in C_0\left(\mathcal{X}\right)\right\}.
 \end{equation}

 with a $C_0\left(\mathcal X\right)$-valued sesquilinear product given by
\begin{equation}\label{comm_pr_l2}
\langle \xi, \eta \rangle_{\mathscr{L}^2\left(\widetilde{\mathcal{X}}_{\mathcal{X}}\right)}(x) = \sum_{\widetilde{x}\in \widetilde{\pi}^{-1}(x)} \xi^*(\widetilde{x})\eta(\widetilde{x}).
\end{equation}
We say that $\mathscr{L}^2\left(\widetilde{\mathcal{X}}_{\mathcal{X}}\right)$ is an {\it associated with $\pi:\widetilde{\mathcal{X}} \to \mathcal{X}$}  Hilbert $C_0(\mathcal{X})$-module. 
\end{defn}
\begin{rem}
If $G$ is a finite group then $\mathscr{L}^2\left(\widetilde{\mathcal{X}}_{\mathcal{X}}\right)\approx C_0\left(\widetilde{\mathcal{X}}\right)_{C_0\left(\mathcal{X}\right)}$ as  Hilbert $C_0\left(\mathcal{X}\right)$-modules, so this definition compiles with the Theorem  \ref{pavlov_troisky_thm} and with the Equation \eqref{fin_form_a}.
\end{rem}


\begin{defn}\label{c_c_def_2}
A $C^*$-algebra $C_0\left(\mathcal{X}\right)$ is given by following equation
\begin{equation*}
C_0\left(\mathcal{X}\right) = \left\{\varphi \in C_b\left(\mathcal{X}\right) \ | \ \forall \varepsilon > 0 \ \ \exists K \subset \mathcal{X} \ ( K \text{ is compact}) \ \& \ \forall x \in \mathcal X \backslash K \ \left|\varphi\left(x\right)\right| < \varepsilon  \right\}.
\end{equation*}
\end{defn}
\begin{lem}\label{l2_comp}
Let $\mathcal X$ be a second-countable compact Hausdorff space. If  $\pi:  \widetilde{\mathcal X}\to \mathcal X$ is a regular covering projection such that $G = G\left( \widetilde{\mathcal X} \ | \ \mathcal X\right)$ is countable then $\mathscr{L}^2\left(\widetilde{\mathcal{X}}_{\mathcal{X}}\right) \subset C_0\left(\widetilde{\mathcal{X}}\right)$. 
\end{lem}

\begin{proof}
 Let $\left\{\widetilde{\mathcal{U}}_{\iota}\subset\widetilde{\mathcal{X}}\right\}_{\iota \in I}$ be a basis of the fundamental covering of $\widetilde{\pi}:\widetilde{\mathcal{X}}\to \mathcal{X}$. Since $\mathcal X$ is compact we can select finite family  $\left\{\widetilde{\mathcal{U}}_{\iota}\subset\widetilde{\mathcal{X}}\right\}_{\iota \in I}$ , i.e.  $\left\{\widetilde{\mathcal{U}}_{\iota}\right\}_{\iota \in I} = \left\{\widetilde{\mathcal{U}}_1, ..., \widetilde{\mathcal{U}}_n\right\}$. Let 
\begin{equation*}
G_1 \leftarrow G_2 \leftarrow ...
\end{equation*}
be a coherent sequence of finite groups with epimorphisms $h_i: G \to G_i$, and let $\left\{G^k\subset G\right\}_{k \in \mathbb{N}}$  be a $G$-covering (See the Definition \ref{g_cov_defn}). 
If $\widetilde{\mathcal V}=\bigcup_{i = 1,...,n}\widetilde{\mathcal{U}}_i$ and $K=\mathfrak{cl}\left(\widetilde{\mathcal V}\right)$ is the closure of $\widetilde{\mathcal V}$ then $K$ is compact. For any $k \in \mathbb{N}$ the set $K_k = G^kK$ is a finite union of compact sets, whence $K_k$ is compact for any $k \in \mathbb{N}$. If $\widetilde{\mathcal V}_k = \bigcup_{k \in \mathbb{N}}G^k\widetilde{\mathcal V}$ then from definitions it follows that
 $\widetilde{\mathcal X}= \bigcup_{k \in \mathbb{N}}  \widetilde{\mathcal V}_k$. Let $\varphi \in \mathscr{L}^2\left(\widetilde{\mathcal{X}}_{\mathcal{X}}\right)$ be such that $\varphi \notin  C_0\left(\mathcal{X}\right)$. From the Definition \ref{c_c_def_2} it follows that there is $\varepsilon > 0$ such that for any compact set $K \subset \widetilde{\mathcal X}$ there is $\widetilde{x} \in \widetilde{\mathcal X} \backslash K$ such that $\left|\varphi\left(\widetilde{x}\right)\right| > \varepsilon$. Let us define a sequence $\left\{\widetilde{x}_i\in \widetilde{\mathcal{X}} \right\}_{i \in \mathbb{N}}$ such that $\left|\varphi\left(\widetilde{x}_i\right)\right|> \varepsilon$ and $\widetilde{x}_i\in \ \widetilde{\mathcal X} \backslash K_i$.
There is a sequence $\left\{x_i \in \mathcal X\right\}_{i \in \mathbb{N}}$ given by $x_i = \pi \left(\widetilde{x}_{i_j}\right)$. Since $\mathcal{X}$ is  compact the sequence $\left\{x_i \right\}_{i \in \mathbb{N}}$ contains a convergent subsequence $\left\{x_{i_j} \right\}_{j \in \mathbb{N}}$. Let $x = \lim_{j \to \infty} x_{i_j}$. Let $\widetilde{x} \in \widetilde{\mathcal X}$ be such that $\widetilde{\pi}\left(\widetilde{x}\right)=x$ and $\widetilde{x} \in \widetilde{\mathcal{V}}$. From \eqref{hilb_cond_comm} it follows that the series
\begin{equation*}
\sum_{g \in G}\left|\varphi\left(g\widetilde{x}\right)\right|^2
\end{equation*}
is convergent, whence there is $r \in \mathbb{N}$ such that
\begin{equation}\label{min_phi_2}
\sum_{g \in G \backslash G^r}\left|\varphi\left(g\widetilde{x}\right)\right|^2 < \frac{\varepsilon^2}{2}.
\end{equation}
If $\widetilde{\mathcal W}$ is an open connected neighborhood of $\widetilde{x}$ which is mapped homeomorphicaly onto $\mathcal{W}=\widetilde{\pi}\left(\widetilde{\mathcal W}\right)$ and $\widetilde{\mathcal W} \subset \widetilde{\mathcal V}$ then there is  a real continuous function $\psi:\widetilde{\mathcal W} \to \mathbb{R}$ given by
\begin{equation*}
\psi\left(y\right)= \sum_{g \in G \backslash G^r} \left|\phi\left(g \widetilde{y}\right)\right|^2; \text{ where } \widetilde{y} \in \widetilde{\mathcal W} \text{ and } \widetilde{\pi}\left(\widetilde{y}\right)=y.
\end{equation*}
There is $s \in \mathbb{N}$ such that $i_s > r$ and $x_{i_j} \in \mathcal W$  for any $j \ge s$. If $j > s$  then from  $\left|\varphi\left(\widetilde{x}_{i_j}\right)\right| > \varepsilon$ and $\widetilde{x}_{i_j} \notin K_r$ it follows that 
\begin{equation*}
\psi\left(x_{i_j}\right) = \sum_{g \in G \backslash G^r} \left|\varphi\left(g\widetilde{x}'_{i_j}\right)\right|^2 \ge \left|\varphi\left(\widetilde{x}_{i_j}\right)\right|^2 > \varepsilon^2; \text{ where } \widetilde{x}'_{i_j} \in \widetilde{\mathcal W} \text{ and } \widetilde{\pi}\left(\widetilde{x}'_{i_j}\right)=x_{i_j}.
\end{equation*}
Since $\psi$ is continuous and $x = \lim_{j \to \infty} x_{i_j}$ we have $\psi\left(x\right) > \varepsilon^2$. This fact contradicts to the equation \eqref{min_phi_2}, and the contradiction proves the lemma.

\end{proof}
\begin{lem}\label{l2_l_comp} 
Let $\mathcal X$ be a second-countable locally compact Hausdorff space. If  $\widetilde{\pi}:  \widetilde{\mathcal X}\to \mathcal X$ is a regular covering projection such that $G = G\left( \widetilde{\mathcal X} \ | \ \mathcal X\right)$ is countable then $\mathscr{L}^2\left(\widetilde{\mathcal{X}}_{\mathcal{X}}\right) \subset C_0\left(\widetilde{\mathcal{X}}\right)$
\end{lem}
\begin{proof}
 If $\varphi \in \mathscr{L}^2\left(\widetilde{\mathcal{X}}_{\mathcal{X}}\right)$then  $\psi\left(x\right)=\sum_{\widetilde{x}\in \widetilde{\pi}^{-1}(x)} \left|\varphi(\widetilde{x})\right|^2 \in C_0\left(\mathcal X\right)$. Let $\varepsilon > 0$ be any number. Form the Definition \ref{c_c_def_2} it follows that there is a compact set $K \subset \mathcal X$ such $\psi\left(\mathcal X \backslash K \right) \subset [0,\varepsilon^2]$. From this fact it follows that $\left|\varphi(\widetilde{x})\right| < \varepsilon$ for any $\widetilde{x} \in \widetilde{\mathcal{X}} \backslash \widetilde{\pi}^{-1}\left(K\right)$. If $\widetilde{K}= \widetilde{\pi}^{-1}\left(K\right)$ then the restriction $\varphi|_{\widetilde{K}}$ belongs to $\mathscr{L}^2\left(\widetilde{K}_{K}\right)$. From the Lemma \ref{l2_comp} it follows that $\varphi|_{\widetilde{K}} \in C_0\left(\widetilde{K}\right)$, whence there is a compact set $\widetilde{K}_0 \subset \widetilde{K}$ such that  $\left|\varphi\left(\widetilde{x}\right) \right| < \varepsilon$ for any $\widetilde{x} \in \widetilde{K} \backslash \widetilde{K}_0$. In result we have $\left|\varphi\left(\widetilde{x}\right) \right| < \varepsilon$ for any $\widetilde{x} \in \widetilde{\mathcal X} \backslash \widetilde{K}_0$. From the Definition \ref{c_c_def_2} it follows that $\varphi \in C_0\left(\widetilde{\mathcal X}\right)$.
\end{proof}
\begin{empt}\label{desc_descr}
Let us consider the sequence \eqref{top_inv_limit1}, and let $\overline{\mathcal U} \subset \overline{\mathcal X}$ be a connected open set which is mapped homeomorphicaly on $\overline{\pi}\left(\overline{\mathcal U}\right)$. Let $\overline{\phi} \in C_0\left(\overline{\mathcal X}\right)$ be such that $\overline{\pi}\left(\overline{\mathcal X}\backslash\overline{\mathcal U}\right) = \{0\}$. For any $n \in \mathbb{N}^0$ there are covering projections $\overline{\pi}^n: \overline{\mathcal X} \to \mathcal X_n$, $\pi^n: \mathcal X_n \to \mathcal X$  and there is the descent $\phi_n \in C_0\left(\mathcal X_n\right)$ of $\overline{\phi}$ (See Definition \ref{lift_desc_defn}). From the Example \ref{fin_lem} it follows that
\begin{equation}\label{comm_seq}
C_0(\mathcal{X})=C_0(\mathcal{X}_0)\to ... \to C_0(\mathcal{X}_n) \to
\end{equation}
is a sequence of noncommutative covering projections. The sequence $\Lambda = \left\{\phi_n\right\}_{n \in \mathbb{N}^0}$ satisfies to \eqref{seq_cond}.
From 
\begin{equation*}
\left\langle\phi_n, \phi_n\right\rangle_{C_0\left(\mathcal X_n\right)}\left(x\right)= \sum_{x_n \in \left(\pi^n\right)^{-1}\left(x\right)}\left(\phi_n\right)^*\left(x_n\right)\phi_n\left(x_n\right) =  \phi^*_0\left(x\right)\phi_0\left(x\right)\in C_0\left(\mathcal X\right),
\end{equation*}
it follows that the sequence $\Lambda$  satisfies to \eqref{norm_conv_cond}, so $\Lambda$ is a coherent sequence.
\end{empt}
\begin{defn}\label{desc_defn}
Let us consider the situation \ref{desc_descr}. The coherent sequence $\Lambda = \left\{\phi_n\right\}_{n \in \mathbb{N}^0}$ is said to be the {\it descent} of $\overline{\phi}$, and we will write $\left\{\phi_n\right\}_{n \in \mathbb{N}^0} = \mathfrak{Desc}\left(\overline{\phi}\right)$.  The element  $\xi = \mathfrak{Rep}\left(\mathfrak{Desc}\left(\overline{\phi}\right)\right)\in \overline{X}_{C_0\left(\mathcal X\right)}$ is said to be a {\it represented} by $\overline{\phi}$, or $\overline{\phi}$ is a {\it representative} of $\xi$. We will write $\xi = \mathfrak{Rep}\left(\overline{\phi}\right)$.
\end{defn}
\begin{empt}\label{dense_comm}
 The sequence \eqref{comm_seq} is composable. There is the noncommutative covering projection   $\left(C_0(\mathcal{X}), \widetilde{A}, \ _{\widetilde{A}}X_{C_0(\mathcal{X})} \ , G\right)$  of the sequence \eqref{comm_seq}.   Let $\left\{\widetilde{\mathcal{U}}_{\iota}\subset\widetilde{\mathcal{X}}\right\}_{\iota \in I}$ be a basis of the fundamental covering, and let
   \begin{equation*}
  1_{C_b\left(\overline{\mathcal X}\right)} = \sum_{g \in G\left(\overline{\mathcal X}  | \mathcal X\right)} \  \sum_{\iota \in I}  g\overline{a}_\iota = \sum_{\left(g,\iota\right)\in G\left(\overline{\mathcal X}  | \mathcal X\right) \times I}\overline{a}_{\left(g,\iota\right)}
   \end{equation*}
 be a partition of unity is dominated by $\left\{\widetilde{\mathcal{U}}_{\iota}\right\}_{\iota \in I}$ . If $\overline{e}_\iota = \sqrt{\overline{a}_\iota}$ for any $\iota \in I$ then from the Corollary  \ref{cor_xi1} it follows that
 \begin{equation}\label{comm_unity_equ}
  \sum_{\iota \in I, g \in G\left(\overline{\mathcal X} | \mathcal X\right)} g \mathfrak{Rep}\left(\overline{e}_\iota \right) \rangle \langle g \mathfrak{Rep}\left(\overline{e}_\iota \right)  = 1_{M\left(\mathcal{K}\left(\overline{X}_{C_0\left(\mathcal X\right)}\right)\right)}.
  \end{equation}

If $\Xi = \left\{\mathfrak{Rep}\left(\overline{e}_\iota \right)\right\}_{\iota \in I}$ then from the Corollary \ref{cor_xi2} it follows that the set $\overline{G}\Xi C_0(\mathcal{X})$ is dense in $\overline{X}_{C_0(\mathcal{X})}$. 

\end{empt}
\begin{lem}\label{l2_eqiv_lem}
Let $\mathcal{X}$ be a locally compact  second-countable  Hausdorff topological space,  and let
\begin{equation}\label{comm_thm_equ}
C_0(\mathcal{X})=C_0(\mathcal{X}_0)\to ... \to C_0(\mathcal{X}_n) \to ...
\end{equation}
be an infinite sequence of finite noncommutative covering projections arising from the sequence \eqref{top_inv_limit1} of connected topological covering projections. If $\overline{X}_{C_0(\mathcal{X})}$  is a disconnected module of the sequence \eqref{comm_thm_equ} then there is a natural isomorphism $\overline{X}_{C_0(\mathcal{X})}\xrightarrow{\approx}\mathscr{L}^2\left(\overline{\mathcal{X}}_{\mathcal{X}}\right)$ of $C_0\left(\mathcal X\right)$-Hilbert modules.
\end{lem}
\begin{proof}
 Let $\overline{\mathcal{X}}$ be the disconnected covering space of the sequence \eqref{top_inv_limit1} with the natural covering projection $\overline{\pi}:\overline{\mathcal{X}}\to \mathcal{X}$ and let $\overline{G} = \varprojlim G\left(\mathcal X_n|\mathcal X\right)$. For any $\overline{x} \in \overline{X}_{C_0(\mathcal{X})}$ we will define a test function $\left(\overline{x}, \phi^{\overline{x}}_{[1,0]}, \overline{\mathcal U}_{\overline{x}}, \overline{\mathcal V}_{\overline{x}} \right)$ subordinated to  $\overline{\pi}: \overline{\mathcal{X}}\to \mathcal{X}$. For any $\zeta \in \overline{\mathcal{X}}$ there is the $C_0\left(\mathcal X\right)$-valued product
\begin{equation*}
\varphi^{\overline{x}}_\xi = \left\langle \xi , \mathfrak{Rep}\left(\phi^{\overline{x}}_{[1,0]} \right)\right\rangle_{\overline{X}_{C_0(\mathcal{X})}} \in C_0\left(\mathcal X\right).
\end{equation*} 
If $\overline{\varphi}^{\overline{x}}_\xi \in C_0\left(\overline{\mathcal{X}}\right)$ is the $\overline{\mathcal U}_{\overline{x}}$ -lift of $\varphi^{\overline{x}}_\xi$ then the  family  $\left\{\overline{\varphi}^{\overline{x}}_\xi|_{\overline{\mathcal V}_{\overline{x}}}:\overline{\mathcal V}_{\overline{x}}\to \mathbb{C} \right\}_{\overline{x}\in \overline{\mathcal X}}$ is coherent, whence  there is the gluing $\varphi_\zeta = \mathfrak{Gluing}\left(\left\{\overline{\varphi}^{\overline{x}}_\xi|_{\overline{\mathcal V}_{\overline{x}}}\right\}\right): \overline{\mathcal{X}} \to \mathbb{C}$. Note that $\varphi_\zeta$ is bounded because $\left\|\varphi_\zeta\right\|=\left\|\zeta\right\|$, i.e. $\varphi_{\zeta} \in C_b\left(\overline{\mathcal{X}}\right)$. So there is 
a natural $C_0(\mathcal{X})$-linear map $\alpha : \overline{X}_{C_0(\mathcal{X})} \to {C_b(\overline{\mathcal{X}})}$, given by $\zeta \mapsto \varphi_{\zeta}$.  If  $\Xi = \left\{\mathfrak{Rep}\left(\overline{e}_\iota \right)\right\}_{\iota \in I}$ then from \ref{dense_comm} it follows that the set  $\overline{G} \ \Xi \ C_0(\mathcal{X})$ is dense in $\overline{X}_{C_0(\mathcal{X})}$. 
So for any nonzero $\zeta \in \overline{X}_{C_0(\mathcal{X})}$  there is a pair $(\iota, g)  \in I \times \overline{G}$ such that  $\langle \zeta, g \mathfrak{Rep}\left(\overline{e}_\iota \right) \rangle_{\overline{X}_{C_0(\mathcal{X})}} \neq 0$. If $x \in \mathcal{X}$ is such that 
$\left\langle \zeta, g \mathfrak{Rep}\left(\overline{e}_\iota \right) \right\rangle_{\overline{X}_{C_0(\mathcal{X})}}(x)\neq 0$ then there is the unique $\overline{x}\in g\overline{\mathcal{U}}_{\iota}$ such that $\overline{\pi}(\overline{x}) = x$ and  $\alpha(\zeta)(\overline{x})\neq 0$.  It means that $\alpha$ is injective. If $\xi, \eta \in \overline{X}_{C_0(\mathcal{X})}$ then from \eqref{comm_unity_equ} it follows that
\begin{equation*}
\left\langle \xi , \eta \right\rangle_{\overline{X}_{C_0(\mathcal{X})}} = \sum_{\iota \in I, g \in \overline{G}} \left\langle \xi , g \mathfrak{Rep}\left(\overline{e}_\iota \right) \right\rangle_{\overline{X}_{C_0(\mathcal{X})}}\left\langle g \mathfrak{Rep}\left(\overline{e}_\iota \right) , \eta \right\rangle_{\overline{X}_{C_0(\mathcal{X})}}
\end{equation*}
From definition of $\alpha$ it follows that if $e_\iota \in C_0\left(\mathcal X\right)$ is the descent of the $\overline{e}_\iota$ for any $\iota \in I$ then
\begin{equation*}
\left( \left\langle \xi , g \mathfrak{Rep}\left(\overline{e}_\iota \right) \right\rangle_{\overline{X}_{C_0(\mathcal{X})}}\left\langle g \mathfrak{Rep}\left(\overline{e}_\iota \right) , \eta \right\rangle_{\overline{X}_{C_0(\mathcal{X})}}\right)\left(x\right)= \left(\alpha\left(\xi\right)\left(\overline{x}\right)\right)\left(e_\iota\left(x\right)\right)^2\left(\alpha\left(\eta\right)\left(\overline{x}\right)\right)^*
\end{equation*}
where $\overline{x}\in \overline{\mathcal X}$ is the unique point such that $\overline{x} \in g \overline{\mathcal U}_\iota$ and $\overline{\pi}\left(\overline{x}\right)=x$. From above equation it follows that for any $\iota \in I$ following condition hold
\begin{equation*}
\left(\sum_{g \in \overline{G}} \left\langle \xi , g \mathfrak{Rep}\left(\overline{e}_\iota \right) \right\rangle_{\overline{X}_{C_0(\mathcal{X})}}\left\langle g \mathfrak{Rep}\left(\overline{e}_\iota \right) , \eta \right\rangle_{\overline{X}_{C_0(\mathcal{X})}}\right)\left(\overline{x}\right) = \sum_{\overline{x} \in \overline{\pi}^{-1}\left(x\right)}\left(\alpha\left(\xi\right)\left(\overline{x}\right)\right)\left(e_\iota\left(x\right)\right)^2\left(\alpha\left(\eta\right)\left(\overline{x}\right)\right)^*.
\end{equation*}
From $\sum_{\iota \in I}e^2_\iota= 1_{M\left(C_0\left(\mathcal X\right)\right)}$ it follows that
\begin{equation*}
\sum_{\iota \in I, g \in \overline{G}} \left\langle \xi , g \mathfrak{Rep}\left(\overline{e}_\iota \right) \right\rangle_{\overline{X}_{C_0(\mathcal{X})}}\left\langle g \mathfrak{Rep}\left(\overline{e}_\iota \right) , \eta \right\rangle_{\overline{X}_{C_0(\mathcal{X})}}\left(\overline{x}\right)=\sum_{\overline{x} \in \overline{\pi}^{-1}\left(x\right)}\left(\alpha\left(\xi\right)\left(\overline{x}\right)\right)\left(\alpha\left(\eta\right)\left(\overline{x}\right)\right)^*,
\end{equation*}
or equivalently
\begin{equation*}
\left\langle \xi , \eta \right\rangle_{\overline{X}_{C_0(\mathcal{X})}} = \sum_{\overline{x} \in \overline{\pi}^{-1}\left(x\right)}\left(\alpha\left(\xi\right)\left(\overline{x}\right)\right)\left(\alpha\left(\eta\right)\left(\overline{x}\right)\right)^*.
\end{equation*}
In fact the above equation coincides with \eqref{comm_pr_l2} and from $\left\langle \zeta, \zeta\right\rangle_{\overline{X}_{C_0(\mathcal{X})}} \in C_0\left(\mathcal X\right)$ it follows that $\alpha\left(\zeta\right)$ satisfies to \eqref{hilb_cond_comm}, i.e. $\alpha\left(\zeta\right)\in \mathscr{L}^2\left(\widetilde{\mathcal{X}}_{\mathcal{X}}\right)$. Otherwise if $\varphi \in \mathscr{L}^2\left(\widetilde{\mathcal{X}}_{\mathcal{X}}\right)$ then the series given by
\begin{equation*}
\xi_{\varphi} = \sum_{\overline{g} \in \overline{G}, \ \iota \in I} \mathfrak{Rep}\left(\varphi \left(g\overline{e}_\iota\right)^2 \right) \in X_{C_0\left(\mathcal X\right)}, 
\end{equation*}
is norm convergent. So there is a linear map $\beta: \mathscr{L}^2\left(\widetilde{\mathcal{X}}_{\mathcal{X}}\right)\to \overline{X}_{C_0(\mathcal{X})}$ given by $\varphi \mapsto \xi_\varphi$. It is easy to check that $\alpha \circ \beta = \mathrm{Id}_{\mathscr{L}^2\left(\widetilde{\mathcal{X}}_{\mathcal{X}}\right)}$ and $\beta \circ \alpha = \mathrm{Id}_{\overline{X}_{C_0(\mathcal{X})}}$, so $\alpha: \overline{X}_{C_0\left(\mathcal{X}\right)} \xrightarrow{\approx}\mathscr{L}^2\left(\widetilde{\mathcal{X}}_{\mathcal{X}}\right)$ is an isomorphism of $C_0\left(\mathcal{X}\right)$-Hilbert modules.
\end{proof}
\begin{thm}\label{thm_comm_1}
Let $\mathcal{X}$ be a locally compact  second-countable  Hausdorff topological space,  and let us consider the infinite sequence \eqref{comm_thm_equ} of finite noncommutative covering projections arising from the sequence \eqref{top_inv_limit1} of connected topological covering projections. Let $G_n = G\left(C_0(\mathcal{X}_n)|C_0(\mathcal{X})\right)= G\left(\mathcal{X}_n|\mathcal{X}\right)$ be covering transformation groups.  Let  $\left(C_0(\mathcal{X}), \widetilde{A}, \ _{\widetilde{A}}X_{C_0(\widetilde{\mathcal{X}})} \ , G\right)$ be  a  noncommutative infinite covering projection of the sequence   \eqref{comm_thm_equ}. Then there is a regular connected topological covering projection $\widetilde{\pi}: \widetilde{\mathcal{X}} \to \mathcal{X}$ such that following conditions hold:
\begin{enumerate}
\item[(a)] $\widetilde{A} = C_0(\widetilde{\mathcal{X}})$ and $_{\widetilde{A}}X_{C_0(\mathcal{X})}=\mathscr{L}^2\left(\widetilde{\mathcal{X}}_{\mathcal{X}}\right)$;
\item[(b)] $G(\widetilde{\mathcal{X}}|\mathcal{X})=G$, $ \mathcal{X} =\widetilde{\mathcal{X}}/G$;
\item[(c)] The sequence \eqref{comm_thm_equ} is faithful.
\end{enumerate}
\end{thm}
\begin{proof}
a)  If $\overline{X}_{C_0(\mathcal{X})}$  is a disconnected module of the sequence \eqref{comm_thm_equ} then from the Lemma \ref{l2_eqiv_lem} it follows that there is a natural isomorphism $\overline{X}_{C_0(\mathcal{X})}\xrightarrow{\approx}\mathscr{L}^2\left(\overline{\mathcal{X}}_{\mathcal{X}}\right)$ of $C_0\left(\mathcal X\right)$-Hilbert modules. From definitions it follows that
 \begin{equation*}
 \sum_{g \in \overline{G}}\sum_{\iota \in I}g\left(\overline{e}_\iota \overline{e}_\iota\right)= 1_{M\left(C_0\left(\overline{\mathcal X}\right)\right)} \ ,
 \end{equation*}
so any $\varphi \in C_0\left(\overline{X}\right)$ can be represented as the following norm convergent series 
 \begin{equation*}
\varphi =   \sum_{g \in \overline{G}}\sum_{\iota \in I}   \mathfrak{Rep}\left(\left(g\overline{e}_\iota\right) \varphi\right)\rangle \langle \mathfrak{Rep}\left(g\overline{e}_\iota \right) \in \mathcal{K}\left(\overline{X}_{C_0\left(\mathcal{X}\right)}\right),
 \end{equation*}
whence $C_0\left(\overline{\mathcal X}\right) \subset \mathcal{K}\left(\overline{X}_{C_0\left(\mathcal{X}\right)}\right)$. Let  $\left\{G^n\subset \overline{G}\right\}_{n \in \mathbb{N}}$ be a $\overline{G}$-covering (See the Definition \ref{g_cov_defn}.) of the sequence $\left\{G_n =G\left(\mathcal{X}_n|\mathcal{X}\right)\right\}_{n \in \mathbb{N}}$. If $\varphi_n \in C_b\left(\mathcal X\right)$ is given by
 \begin{equation*}
\varphi_n = \sum_{g' \in \overline{\overline{G}/G_n}} g' \left(\sum_{g \in G^n}\sum_{\iota \in I}\left(g\overline{e}_\iota\right)^2 \varphi\right)
 \end{equation*}
 then $\varphi_n \in A_n$, i.e. there is $a_n \in C_0\left(\mathcal X_n\right)$ such that $\varphi_n \xi = a_n \xi$ for any $\xi \in X_{C_0\left(\mathcal X\right)}$. Otherwise $\lim_{n \to \infty}\varphi_n = \varphi$ in sense of the pointwise (=weak) convergence, whence $\varphi \in \left(\bigcup_{n\in \mathbb{N}}C_0\left(\mathcal X_n\right)\right)''$. So we have $C_0\left(\overline{\mathcal X}\right)\subset \mathcal{K}\left(\overline{X}_{C_0\left(\mathcal{X}\right)}\right) \bigcap  \left(\bigcup_{n\in \mathbb{N}}C_0\left(\mathcal X_n\right)\right)''$. Let denote $\overline{A}\stackrel{\text{def}}{=} \mathcal{K}\left(\overline{X}_{C_0\left(\mathcal{X}\right)}\right) \bigcap\left(\bigcup_{n\in \mathbb{N}}C_0\left(\mathcal X_n\right)\right)''$. From definitions it follows that $C_0\left(\mathcal X_n \right) \subset L^\infty\left(\overline{\mathcal X}\right)$, where $L^\infty\left(\overline{\mathcal X}\right)$ is  the algebra of essentially bounded complex-valued measurable functions. So $\bigcup_{n \in \mathbb{N}^0}C_0\left(\mathcal X_n \right)\subset L^\infty\left(\overline{\mathcal X}\right)$. Since $L^\infty\left(\overline{\mathcal X}\right)$ is weakly closed,  we have $\left(\bigcup_{n \in \mathbb{N}^0}C_0\left(\mathcal X_n \right)\right)'' \subset L^\infty\left(\overline{\mathcal X}\right)$ and $\overline{A} \subset L^\infty\left(\overline{\mathcal X}\right)$. Let $\mu$ be a measure  on $\overline{X}$ such that the natural representation of  $L^\infty\left(\overline{\mathcal X}\right)$ on $L^2\left(\overline{\mathcal X}, \mu\right)$ is faithful. Let us select any $\overline{x}\in \overline{\mathcal X}$. If  $\left(\overline{x}, \phi^{\overline{x}}_{[1,0]}, \overline{\mathcal U}, \overline{\mathcal V} \right)$ is a  test function subordinated to $\overline{\pi}$, then $\phi^{\overline{x}}_{[1,0]} \in  L^2\left(\overline{\mathcal X}, \mu\right)$. From the definition of test functions it follows that for any $\varphi \in L^\infty\left(\overline{\mathcal X}\right)$ following conditions hold
 \begin{equation*}
 \left(\varphi\left(\phi^{\overline{x}}_{[1,0]}\right)^2 \right)|_{\overline{\mathcal V}} = \varphi|_{\overline{\mathcal V}} \ , 
 \end{equation*}
 \begin{equation*}
 \left(\varphi\left(\phi^{\overline{x}}_{[1,0]}\right)^2\right) \left(\overline{\mathcal X} \backslash \overline{\mathcal U}\right)=\{0\}.
 \end{equation*}
 From the Lemma \ref{l2_l_comp} it follows that any element $\xi \in \overline{X}_{C_0\left(\mathcal{X}\right)} \xrightarrow{\approx}\mathscr{L}^2\left(\widetilde{\mathcal{X}}_{\mathcal{X}}\right)$ can be regarded as element of $C_0\left(\overline{\mathcal X}\right)$. If $\kappa=\sum_{i=1}^{\infty} \eta_i \rangle \langle \xi_i \in \mathcal{K}\left(\overline{X}_{C_0\left(\mathcal{X}\right)}\right)$ is a compact operator then a following series
\begin{equation*}
\begin{split}
\psi = \left\langle\mathfrak{Rep}\left( \phi^{\overline{x}}_{[1,0]}\right), \kappa \ \mathfrak{Rep}\left( \phi^{\overline{x}}_{[1,0]}\right)\right\rangle_{\overline{X}_{C_0\left(\mathcal{X}\right)}} = \\
=\sum_{i=1}^{\infty} \left\langle \mathfrak{Rep}\left( \phi^{\overline{x}}_{[1,0]}\right), \eta_i \right\rangle_{\overline{X}_{C_0\left(\mathcal{X}\right)}} \left\langle \xi_i, \mathfrak{Rep}\left( \phi^{\overline{x}}_{[1,0]}\right)\right\rangle_{\overline{X}_{C_0\left(\mathcal{X}\right)}} \in C_0\left(\mathcal X\right)
\end{split}
 \end{equation*}
 is norm convergent. If $\overline{\psi}\in C_0\left(\overline{\mathcal X}\right)$ is the $\overline{\mathcal U}$-lift of $\psi$ then
 \begin{equation}\label{xi_eta_phi_eqn}
\overline{\psi} =  \left(\sum_{i=1}^{\infty}\eta_i^* \xi_i\right)\left(\phi^{\overline{x}}_{[1,0]}\right)^2
 \end{equation} 
 where $\xi_i, \eta_i \in \mathscr{L}^2\left(\overline{\mathcal{X}}_{\mathcal{X}}\right)$ are regarded as elements of $C_0\left(\overline{\mathcal{X}}\right)$. Otherwise if $\kappa = \overline{\varphi} \in L^\infty\left(\overline{\mathcal{X}}\right)$ then 
 \begin{equation*}
\psi = \sum_{\overline{g}\in \overline{G}}\overline{g}  \left(\overline{\varphi}\left(\phi^{\overline{x}}_{[1,0]}\right)^2 \right),
 \end{equation*}
 whence $\left(\overline{\varphi}\left(\phi^{\overline{x}}_{[1,0]}\right)^2\right)$ is the $\overline{\mathcal U}$-lift of $\psi$.
 From above equations it follows that 
 \begin{equation*}
\overline{\varphi} \left(\phi^{\overline{x}}_{[1,0]}\right)^2= \overline{\psi} = \left(\sum_{i=1}^{\infty}\eta_i^* \xi_i\right)\left(\phi^{\overline{x}}_{[1,0]}\right)^2 .
 \end{equation*}
 where $\overline{\psi} \in C_0\left(\overline{\mathcal X}\right)$ and the series is norm convergent. From the Definition \ref{test_func_defn} of $\phi^{\overline{x}}_{[1,0]}$  it follows that 
 \begin{equation*}
 \overline{\varphi}|_{\overline{\mathcal V}}= \overline{\psi}|_{\overline{\mathcal V}}.
 \end{equation*}
 Since $\overline{\psi}$ is a continuous, the function $\overline{\varphi}|_{\overline{\mathcal V}}$ is also continuous. Thus any  $\overline{x} \in \overline{\mathcal X}$ has a neighborhood $\overline{\mathcal V}$ such that the restriction $\overline{\varphi}|_{\overline{\mathcal V}}$ is continuous, whence $\overline{\varphi}$ is continuous, and $\overline{\varphi}\in C_b\left(\overline{\mathcal X}\right)$ because $\left\|\overline{\varphi}\right\|=\left\|\ka\right\|$. We have selected an  arbitrary point $\overline{x}\in \overline{\mathcal X}$, therefore from \eqref{xi_eta_phi_eqn} it follows that
\begin{equation*}
 \overline{\varphi}\left(\overline{x}\right)= \left(\sum_{i=1}^{\infty}\eta_i^* \xi_i\right)\left(\overline{x}\right); \ \forall \overline{x} \in \overline{\mathcal X}.
 \end{equation*}
 Note that for any $\kappa \in \mathcal{K}\left(\overline{X}_{C_0\left(\mathcal{X}\right)}\right)$ following condition hold
 \begin{equation*}
\left\|\kappa\right\| = \sup_{\mu, \nu} \left\|\left\langle \mu, \kappa\nu\right\rangle_{\overline{X}_{C_0\left(\mathcal{X}\right)}}\right\|; \text{ where } \mu, \nu \in \overline{X}_{C_0\left(\mathcal{X}\right)} \text{ and } \|\mu\|=\|\nu\|=1.
 \end{equation*}
 
 From the definition of the norm it follows that
\begin{equation*}
\left\| \overline{\varphi} -  \sum_{i=1}^{k}\eta_i^* \xi_i\right\| = \sup_{\overline{x}\in \overline{\mathcal X}} \left| \overline{\varphi} -  \sum_{i=1}^{k}\eta_i^* \xi_i\right|\left(\overline{x}\right).
\end{equation*}
 However
\begin{equation*}
\begin{split}
\left| \overline{\varphi} -  \sum_{i=1}^{k}\eta_i^* \xi_i\right|\left(\overline{x}\right) \le \left\|\left\langle \phi^{\overline{x}}_{[1,0]}, \left(\overline{\varphi} - \sum_{i=1}^{k} \eta_i \rangle \langle \xi_i \right)\phi^{\overline{x}}_{[1,0]}\right\rangle\right\| \\ \le \sup_{\alpha, \beta} \left\|\left\langle \mu, \left(\overline{\varphi} - \sum_{i=1}^{k} \eta_i \rangle \langle \xi_i \right)\nu\right\rangle\right\|= \left\| \overline{\varphi} - \sum_{i=1}^{k} \eta_i \rangle \langle \xi_i \right\|; \\ \text{ where } \mu, \nu \in \overline{X}_{C_0\left(\mathcal{X}\right)}, \text{ and } \|\mu\|=\|\nu\|=1,
\end{split}
\end{equation*}
whence
\begin{equation}\label{norm_ineq}
\left\|  \overline{\varphi} -  \sum_{i=1}^{k}\eta_i^* \xi_i\right\| \le \left\|  \overline{\varphi} - \sum_{i=1}^{k} \eta_i \rangle \langle \xi_i \right\|.
\end{equation}
The series $\sum_{i=1}^{\infty} \eta_i \rangle \langle \xi_i$ is norm convergent, and from \eqref{norm_ineq} it follows that the series $ \sum_{i=1}^{\infty}\eta_i^* \xi_i$ is also norm convergent. From the Lemma \ref{l2_l_comp} if follows that $\eta_i, \xi_i \in C_0\left(\overline{\mathcal{X}}\right)$ whence  $\sum_{i=1}^{k}\eta_i^* \xi_i \in C_0\left(\overline{\mathcal{X}}\right)$. Since the series $\sum_{i=1}^{\infty}\eta_i^* \xi$ is norm convergent and any partial sum $\sum_{i=1}^{k}\eta_i^* \xi$ belongs to $C_0\left(\overline{\mathcal{X}}\right)$ following condition hold 
\begin{equation*}
\overline{\varphi} =\sum_{i=1}^{\infty}\eta_i^* \xi_i \in C_0\left(\overline{\mathcal{X}}\right)
\end{equation*}
 and 
\begin{equation*}
\mathcal{K}\left(\overline{X}_{C_0\left(\mathcal{X}\right)}\right) \bigcap  \left(\bigcup_{n\in \mathbb{N}}C_0\left(\mathcal X_n\right)\right)''\subset C_0\left(\overline{\mathcal X}\right).
\end{equation*} 
In result we have 
\begin{equation*}
\mathcal{K}\left(\overline{X}_{C_0\left(\mathcal{X}\right)}\right) \bigcap  \left(\bigcup_{n\in \mathbb{N}}C_0\left(\mathcal X_n\right)\right)''= C_0\left(\overline{\mathcal X}\right).
\end{equation*}
If $\widetilde{\mathcal X}\subset \overline{\mathcal X}$ is a connected component, then $C_0\left(\widetilde{\mathcal X}\right)\subset C_0\left(\overline{\mathcal X}\right)$ is a maximal irreducible subalgebra. If $C_0\left( \overline{\mathcal X}\right)= C_0\left( \widetilde{\mathcal X}\right) \oplus A'$ then $X_{C_0\left(\mathcal X\right)} = \left\{\xi \in \overline{X}_{C_0\left(\mathcal X\right)} \ | \  A'\xi = \{0\} \right\}=\mathscr{L}^2\left(\widetilde{\mathcal{X}}_{\mathcal{X}}\right)$.
\newline
b) The action of $\overline{G}$ on $C_0\left(\overline{\mathcal{X}}\right)$ arises from the action of $\overline{G}$ on $\overline{\mathcal{X}}$, and according to topological construction \ref{top_inv_limit} we have  $\mathcal{X} = \overline{\mathcal{X}}/\overline{G}$. If $G \subset \overline{G}$ is the maximal subgroup such that $GC_0\left(\widetilde{\mathcal{X}}\right)=C_0\left(\widetilde{\mathcal{X}}\right)$ then $G$ is the maximal subgroup such that $G\widetilde{\mathcal{X}}=\widetilde{\mathcal{X}}$ and vice versa. So $G(\widetilde{\mathcal{X}}|\mathcal{X})=G$, $ \mathcal{X} =\widetilde{\mathcal{X}}/G$.
\newline
c) If $a \in C_0\left(\mathcal X_n\right)$ is a nonzero function then there is an open set $\mathcal U \subset \mathcal X_n$ such that
\begin{itemize}
\item $a|_{\mathcal U}\neq 0$;
\item  $\mathcal U$ is evenly covered by $\pi_n:\overline{\mathcal X} \to \mathcal X_n$.
\end{itemize}
If $\overline{\mathcal{U}}\subset\overline{\mathcal{X}}$ be an connected open subset which is homeomorphically mapped onto $\mathcal U$ and $\overline{a}\in C_0\left(\overline{\mathcal X}\right)$ is such that $\overline{a}|_{\overline{\mathcal{U}}}\neq 0$ then $\overline{a} a \neq 0$ and $a\overline{a}\neq 0$, so the actions of $C_0\left(\mathcal X_n\right)$ on $C_0\left(\overline{\mathcal X}\right)$ are faithful. From \cite{spanier:at} it follows that the sequence of regular covering projections
\begin{equation*}
\widetilde{\mathcal X} \to \mathcal X_n  \to \mathcal X
\end{equation*}
induces the epimorphism  of groups $G\left(\widetilde{\mathcal X} | \mathcal X\right)\to G\left(\mathcal X_n | \mathcal X\right)$. From $G\left(C_0\left(\widetilde{\mathcal X}\right) | C_0\left(\mathcal X\right)\right) \approx G\left(\widetilde{\mathcal X} | \mathcal X\right)$ and $G\left(C_0\left(\mathcal X_n\right) | C_0\left(\mathcal X\right)\right) \approx G\left(\mathcal X_n | \mathcal X\right)$ follows that the natural map   $G\left(C_0\left(\widetilde{\mathcal X}\right) | C_0\left(\mathcal X\right)\right) \to G\left(C_0\left(\mathcal X_n\right) | C_0\left(\mathcal X\right)\right)$ is a group epimorphism.

  \end{proof}
\begin{lem}\label{comm_case}
Let $\pi: \widetilde{\mathcal{X}} \to \mathcal{X}$ be a regular topological covering projection by a connected space $\widetilde{\mathcal{X}}$ such that a group $G = G\left(\widetilde{\mathcal{X}} | \mathcal{X}\right)$ of covering transformations is countable, i.e. $\mathcal{X}\approx \widetilde{\mathcal{X}}/ G$.  Let $... \to G_n \to ... \to G_1$ a coherent sequence (See Definition \ref{g_cov_defn}) of finite $G$-quotients with epimorphisms $h_n: G \to G_n$. Let $\mathcal{X}_n = \widetilde{\mathcal{X}}/ \mathrm{ker} \ h_n$. For any $n \in \mathbb{N}$ there is a natural finitely listed covering projection $\pi_n:\mathcal{X}_n \to \mathcal{X}$ such that $G\left(\mathcal{X}_n | \mathcal{X}\right) = G_n$. There is a following sequence of *-homomorphisms
\begin{equation}\label{alg_inv_limit}
C_0(\mathcal{X})\to C_0(\mathcal{X}_1) \to ... \to C_0(\mathcal{X}_n) \to ... \ .
\end{equation} 
If $\left(C_0(\mathcal{X}), \widetilde{A}, \ _{C_0\left(\widetilde{\mathcal{X}}'\right)}X_{C_0\left(\mathcal{X}\right)} \ , G'\right)$ {\it noncommutative infinite covering projection} of the sequence \eqref{alg_inv_limit} then $\widetilde{\mathcal{X}}' \approx \widetilde{\mathcal{X}}$ and $G'\approx G$.

\end{lem}
\begin{proof}
If  $\overline{\mathcal{X}}$ is a disconnected covering space of the sequence \eqref{top_inv_limit1}, $\overline{G}$ is a disconnected group of \eqref{top_inv_limit1} and $J=\overline{\overline{G}/G} \subset \overline{G}$, $J'=\overline{\overline{G}/G'}$ are  sets of representatives of $G$ and $G'$ in $\overline{G}$ then from 
\begin{equation*}
\overline{\mathcal{X}} = \bigsqcup_{g \in J}g\widetilde{\mathcal{X}}= \bigsqcup_{g \in J'}g\widetilde{\mathcal{X}}'.
\end{equation*} 
it follows that $\widetilde{\mathcal{X}}' \approx \widetilde{\mathcal{X}}$ and $G'\approx G$ because both $\widetilde{\mathcal{X}}'$ and $\widetilde{\mathcal{X}}$ are connected.
\end{proof}
\begin{rem}
From the Theorem \ref{thm_comm_1} and the Lemma \ref{comm_case} it follows that an infinite covering can be constructed algebraically.
\end{rem}

\subsection{Covering projections of spectral triples}\label{comm_sptr_subs}

\paragraph*{}In this section we will consider following objects:
\begin{enumerate}
 \item A compact orientable Riemannian manifold 	$(M,\eps)$ without boundary  with a 
 $\mathrm{Spin}^{\mathbf{c}}$ structure
 	$\left(\sS,C\right)$ where $\sS = \Gamma_{\mathrm{smooth}}\left(M, S\right)$ where $\Gamma_{\mathrm{smooth}}$ means the functor of $C^\infty$-sections and $S$ is the spinor bundle (See the Definition \ref{df:spin-structure}). 
 	\item A covering projection $\widetilde{\pi}: \widetilde{M} \to M$.
 	\item The pullback $\widetilde{S}$ of $S$ by $\widetilde{\pi}$ (See the Definition \ref{inv_image_top_defn}).
\end{enumerate}

\begin{empt} 	
 	Let $\pi: \widetilde{M} \to M$ be a regular covering projection.
 	From the Proposition \ref{comm_cov_mani} it follows that  $\widetilde{M}$ has the natural structure of $C^\infty$-manifold. Moreover from \cite{do_carmo:rg} it follows that $\widetilde{M}$ is a Riemannian manifold with {\it covering metric} $\widetilde{g}$. 
\end{empt}
\begin{defn}
	A $\mathbb{R}$-linear map $\varphi: \Gamma_{\mathrm{smooth}}\left(M, S\right) \to \Gamma_{\mathrm{smooth}}\left(M, S\right)$ is said to be {\it local} if for any open subset $U$ and any $s', s'' \in \Gamma_{\mathrm{smooth}}\left(M, S\right)$ such that from $s'|_{U}=s''|_{U}$ it follows that
	\begin{equation*}
	\varphi\left(s'\right)|_{U}=\varphi\left(s''\right)|_{U}.
	\end{equation*}
\end{defn}
\begin{empt}\label{local_pullback_descr}
	Let $\varphi: \Gamma_{\mathrm{smooth}}\left(M, S\right) \to \Gamma_{\mathrm{smooth}}\left(M, S\right)$ be a local operator and let $\left\{\widetilde{U}_\iota\right\}_{\iota \in I}$ be a one-to-one covering  with respect to $\widetilde{\pi}$ and $U_\iota=\widetilde{\pi}\left(\widetilde{U}_\iota\right)$. For any $\iota \in I$ there is a $\mathbb{C}$-linear isomorphism $\al_\iota: \Gamma_{\mathrm{smooth}}\left(U_\iota, S|_{U_\iota}\right) \to \Gamma_{\mathrm{smooth}}\left(\widetilde{U}_\iota, \widetilde{S}|_{\widetilde{U}_\iota}\right)$. If $\widetilde{s} \in \Gamma_{\mathrm{smooth}}\left(\widetilde{M}, \widetilde{S}\right)$ then for any $\iota \in I$ there is the unique section $\widetilde{t}_\iota \in \Gamma_{\mathrm{smooth}}\left(\widetilde{U}_\iota, \widetilde{S}|_{\widetilde{U}_\iota}\right)$ given by
	\begin{equation*}
	\widetilde{t}_\iota = \al_\iota\left(\varphi\left(\al^{-1}_\iota\left(\widetilde{s}|_{\mathcal U_\iota}\right)\right)\right).
	\end{equation*}
 A family $\left\{\widetilde{t}_\iota\right\}_{\iota \in I}$ is coherent because $\varphi$ is local. So there is the gluing $\widetilde{t} = \mathfrak{Gluing}\left(\left\{\widetilde{t}_\iota\right\}_{\iota \in I}\right)\in \Gamma_{\mathrm{smooth}}\left(\widetilde{M}, \widetilde{S}\right)$. In fact definition of $\widetilde{t}$ does not depend on the family $\left\{\widetilde{U}_\iota\right\}_{\iota \in I}$.
\end{empt}
\begin{defn}\label{pullback_spin_defn}
In the situation \ref{local_pullback_descr} there is a local operator $\widetilde{\varphi}: \Gamma_{\mathrm{smooth}}\left(\widetilde{M}, \widetilde{S}\right) \to \Gamma_{\mathrm{smooth}}\left(\widetilde{M}, \widetilde{S}\right)$ given by $\widetilde{s} \mapsto \widetilde{t}$. The operator $\widetilde{\varphi}$ is said to be the $\widetilde{\pi}$-\textit{pullback} of $\varphi$. Henceforth we write $\widetilde{\varphi}= \mathfrak{pullback}_{\widetilde{\pi}}\left(\varphi\right)$.
\end{defn}
\begin{rem}\label{pull_equiv_rem}
Any pullback is $G\left(\widetilde{M}|M\right)$-equivariant, i.e.
\begin{equation*}
	\mathfrak{pullback}_{\widetilde{\pi}}\left(\varphi\right) \left(g\widetilde{s}\right)= g\left(\mathfrak{pullback}_{\widetilde{\pi}}\left(\varphi\right) \left(\widetilde{s}\right)\right)
\end{equation*}
for any $\widetilde{s} \in \Gamma_{\mathrm{smooth}}\left(\widetilde{M}, \widetilde{S}\right)$ and $g \in G\left(\widetilde{M}|M\right)$.
\end{rem}

\begin{empt}\label{comm_ind_triple_descr}  
 Similarly to  \eqref{comm_hilb_space} we can define a scalar product on $\widetilde{H}=\widetilde{A}\otimes_A H$ given by	
\begin{equation*}
\left(\widetilde{\phi},\widetilde{\psi}\right) =\int_{\widetilde{M}} \pairing{\widetilde{\phi}}{\widetilde{\psi}} \nu_{\widetilde{g}}
\word{for} \widetilde{\phi}, \widetilde{\psi} \in \widetilde{\sS}
\end{equation*}
where $\nu_{\widetilde{g}}$ means Riemannian measure on $\widetilde{M}$.
Since $\widetilde{M} \to M$ is a finitely listed covering above integral can be presented by following way
	\begin{equation}\label{comm_scalar_product}
	\begin{split}
	 \int_{\widetilde{M}} \pairing{a\otimes\phi}{b\otimes\psi} \,\widetilde{\nu}_g = \\ \int_M \left(\sum_{y = \pi_n^{-1}\left(x\right)} a^*\left(y\right)b\left(y\right)\right)\left(\phi\left(x\right)|\psi\left(x\right)\right) \,\nu_g = \left(\phi, \left\langle a, b\right\rangle_{C\left(\widetilde{M}\right)}\psi\right),
	\end{split}
	\end{equation}	
i.e. the Hilbert scalar product on $\widetilde{H}$. complies with \eqref{hilb_scalar_correct}
It is clear that both the Dirac $\slashed D$ operator and charge conjugation operator $C$ are local, so there are pullbacks $\widetilde{\slashed D}	\stackrel{\mathrm{def}}{=}\mathfrak{pullback}_{\widetilde{\pi}}\left(\slashed D\right)$, $\widetilde{C}	\stackrel{\mathrm{def}}{=}\mathfrak{pullback}_{\widetilde{\pi}}\left(C\right)$. So $\left(\widetilde{M}, \widetilde{\eps}\right)$ is a orientable Reimannian manifold without boundary with the $\mathrm{Spin}^{\mathbf{c}}$ structure
$\left(\widetilde{\sS}, \widetilde{C}\right)$ where $\widetilde{\sS} = \Gamma_{\mathrm{smooth}}\left(\widetilde{M}, \widetilde{S}\right)$. If $\mathrm{dim}~M$ is even then from the Definition \ref{df:spt-even} it follows the existence of the operator $\Ga$ which is also local, so there is the pullback $\widetilde{\Ga}	\stackrel{\mathrm{def}}{=}\mathfrak{pullback}_{\widetilde{\pi}}\left(\Ga\right)$.
\end{empt}
\begin{defn}\label{pi_lift_st_defn} Let us consider the situation \ref{comm_ind_triple_descr}, and suppose that the covering projection $\widetilde{\pi}:\widetilde{M} \to M$ is finite fold and regular. The spectral triple $\left(\Coo\left(\widetilde{M}\right),L^2(\widetilde{M},\widetilde{S}), \widetilde{\slashed D}\right)$ is said to be the $\widetilde{\pi}$-\textit{pullback} of $\left(\Coo\left(M\right),L^2(M,S), \slashed D\right)$. We will write $\left(\Coo\left(\widetilde{M}\right),L^2(\widetilde{M},\widetilde{S}), \widetilde{\slashed D}\right) = \mathfrak{pullback}_{\widetilde{\pi}}\left(\left(\Coo\left(M\right),L^2(M,S), \slashed D\right)\right)$.
\end{defn}
\begin{rem}
	If a sequence $\widetilde{\widetilde{M}} \xrightarrow{\widetilde{\widetilde{\pi}}}\widetilde{M} \xrightarrow{\widetilde{\pi}}M$ is such that both $\widetilde{\widetilde{\pi}}$ and $\widetilde{\pi}$ are covering projections then following condition hold
	\begin{equation*}
\mathfrak{pullback}_{\widetilde{\widetilde{\pi}}\circ\widetilde{\pi}}\left(\left(\Coo\left(M\right),L^2(M,S), \slashed D\right)\right)=	\mathfrak{pullback}_{\widetilde{\widetilde{\pi}}}\left(\mathfrak{pullback}_{\widetilde{\pi}}\left(\left(\Coo\left(M\right),L^2(M,S), \slashed D\right)\right)\right).
	\end{equation*}
\end{rem}
\begin{lem}\label{comm_coherent_lemma}
	Let $(M,\eps)$ be a compact orientable Riemannian manifold without boundary  with a 
	$\mathrm{Spin}^{\mathbf{c}}$ structure
	$\left(\sS,C\right)$. Let
	\begin{equation*}
	M = M_0 \xleftarrow{}M_1 \xleftarrow{}... \xleftarrow{} M_n \xleftarrow{}...
	\end{equation*}	
	be a sequence of finite fold regular covering projections which induces the sequence of *-homomorphisms
	\begin{equation*}
	C\left(M\right)=C\left(M_0\right)\to C\left(M_1\right)\to ... \to C\left(M_n\right)\to ... .
	\end{equation*} 
	If $\pi_n: M_n \to M$ is a natural covering projection for  and $\left(\Coo\left(M_n\right),L^2(M_n,S_n), \slashed D_n\right)$ is the $\pi_n$-pullback of $\left(\Coo\left(M\right),L^2\left(M,S\right), \slashed D\right)$
then the sequence $\left\{\left(\Coo\left(M_n\right),L^2(M_n,S_n), \slashed D_n\right)\right\}_{n\in \mathbb{N}^0}$ is  a coherent sequence of spectral triples. 
\end{lem}
\begin{proof}
	We need check conditions 1-6 of the Definition \ref{coh_spec_triple_defn}.
\begin{enumerate}
		\item There is a sequence of injective *-homomorphisms
		\begin{equation*}
			\Coo\left(M\right) = \Coo\left(M_1\right) \to \Coo\left(M_2\right) \to ... \to \Coo\left(M_n\right) \to ... \ .
		\end{equation*} 

		\item For any $n \in \mathbb{N}$ algebra $C\left(M_{n}\right)$ is the $C^*$-completion of $\Coo\left(M_{n}\right)$ and there is a finite noncommutuative covering projection  $\left(C\left(M_{n-1}\right), C\left(M_{n}\right), G\left(C\left(M_{n}\right) \ | \ C\left(M_{n-1}\right)\right)\right)$.
	\item The sequence of finite noncommutative covering projections 
	\begin{equation}\label{com_reimann_seq_eqn}
	C\left(M\right) = C\left(M_1\right) \to C\left(M_2\right) \to ... \to C\left(M_n\right) \to ... \ 
	\end{equation} 
	is composable.
	\item If $g \in G\left(M_n,M\right)$ then $g\Coo\left(M_n\right) = \Coo\left(M_n\right)$ , and $\Coo\left(M_n\right)^{G\left(M_n| M_{m}\right)}= \Coo\left(M_m\right)$ for any $m, n \in \mathbb{N}^0$.
	\item The $\Coo\left(M_n\right)$-module $\Ga_{\mathrm{smooth}}\left(M_n, S_n\right)= \bigcap_{k\in\bN} \Dom \slashed D_n^k$ is given by $\Ga_{\mathrm{smooth}}\left(M_n, S_n\right) = \Coo\left(M_n\right) \otimes_{\Coo\left(M\right)} \Ga_{\mathrm{smooth}}\left(M, S\right)$ where $\Ga_{\mathrm{smooth}}\left(M, S\right)= \bigcap_{k\in\bN} \Dom\slashed D^k\subset L^2\left(M, S, \nu^g\right)$. The Hilbert space $H=L^2\left(M,S\right)$ of any commutative spectral triple is a Hilbert completion of the space $\sS = \Ga_{\mathrm{smooth}}\left(M, S\right)$ of smooth sections of the spinor bungle. However $\sS$ is a dense subspace of $\overline{\sS}=\Ga\left(M, S\right)$ continuous sections of the spinor bundle. The bungle $S_n$ is a $\pi_n$-lift of $S$, and from \ref{tensor_bundle} it follows that
	\begin{equation*}
	\overline{\sS_n}=\Ga\left(M_n, S_n\right) = C\left(M_n\right) \otimes_{C\left(M\right)} \overline{\sS}
	\end{equation*} 
	whence $L^2\left(M_n,S_n\right)$ is the Hilbert completion of $C\left(M_n\right) \otimes_{C\left(M\right)}L^2\left(M,S\right)$. However the Hilbert completion of $C\left(M_n\right) \otimes_{C\left(M\right)}L^2\left(M,S\right)$ coincides with $C\left(M_n\right) \otimes_{C\left(M\right)}L^2\left(M,S\right)$ because  $C\left(M_n\right)$ is a finitely generated $C\left(M\right)$-module, i.e. $H_n = C\left(M_n\right) \otimes_{C\left(M\right)} H_0$. From \eqref{comm_scalar_product} it follows that for any $n \in \mathbb{N}$ the Hilbert scalar product on $H_n$ complies with \eqref{hilb_scalar_correct}.
	\item  According to definition $\slashed D_n = \mathfrak{pullback}_{\pi^n}\left(\slashed D\right)$ and from the Remark \ref{pull_equiv_rem} it follows that $\slashed D_n$ is $G\left(M_n|M\right)$ equivariant, i.e. for any $g \in G\left(M_n,M\right)$ and $\xi \in \xi \in \Ga_{\mathrm{smooth}}\left(M_n, S_n\right)$  following condition hold
	\begin{equation*}
	g \left(\slashed D\xi\right)= \slashed D\left(g\xi\right).
	\end{equation*}
	From the Definition \ref{pullback_spin_defn} it follows that if $m < n$ then
	\begin{equation*}
\slashed D_n|_{\Ga_{\mathrm{smooth}}\left(M_m, S_m\right)}= \slashed D_m
	\end{equation*}
\end{enumerate}
\end{proof}
\begin{lem}\label{comm_local_lemma}
The coherent sequence of spectral triples from the Lemma \ref{comm_coherent_lemma} is local.
\end{lem}
\begin{proof}
From the Theorem \ref{thm_comm_1} it follows that there are a connected topological space $\widetilde{M}$ and a covering projection $\widetilde{p}:\widetilde{M} \to M$ such that $\left(C\left(M\right), C\left(\widetilde{M}\right),~ _{C\left(\widetilde{M}\right)}X_{C\left(M\right)}, G \right)$ is the noncommutative covering projection of the sequence \eqref{com_reimann_seq_eqn}. Let $x\in M$ be any point and $\widetilde{x}\in \widetilde{M}$ is such that $\widetilde{p}\left(\widetilde{x}\right)=x$. From the Remark \ref{cut_loci_rem} there is the cut loci $\Om_x \subset M$ such that following conditions hold:
\begin{enumerate}
\item  $1_{\Om_x}\in L^\infty\left(M\right)$ and if $1_M= 1_{\Om_x}$ as element of $ L^\infty\left(M\right)$,
\item There is a connected set $\widetilde{\Om}_{\widetilde{x}}$ which is mapped homeomophically on $\Om_x$ and $\widetilde{\Om}_{\widetilde{x}}$ is a fundamental domain of the $\widetilde{p}:\widetilde{M} \to M$. For any $n \in \mathbb{N}$ there is a covering projection $\widetilde{p}_n:\widetilde{M} \to M_n$ and $\Om^n_{\widetilde{p}_n\left(\widetilde{x}\right)}=\widetilde{p}_n\left(\widetilde{\Om}_{\widetilde{x}}\right)$ is a fundamental domain of the covering projection $p_n: M_n \to M$. 
\end{enumerate}
\paragraph*{}If $\widehat{H}^n = 1_{\Om^n_{\widetilde{p}_n\left(\widetilde{x}\right)}}L^2\left(M_n,S_n\right)$ then from $1_{M_n}= \sum_{g \in G\left(M_n|M\right)} g~1_{\Om^n_{\widetilde{p}_n\left(\widetilde{x}\right)}}$ it follows that
\begin{equation*}
L^2\left(M_n,S_n\right) = \bigoplus_{g \in G\left(M_n|M\right)}g \widehat{H}^n.
\end{equation*}
From locality of the operators $\slashed D_n$ it follows that they satisfy the Definition \ref{local_defn}. The condition (a) of the Definition \ref{local_sec_thiples_defn} follows from
\begin{equation*}
1_{\Om^{n-1}_{\widetilde{p}_{n-1}\left(\widetilde{x}\right)}}= \sum_{g \in G\left(M_n|M_{n-1}\right)} g~1_{\Om^n_{\widetilde{p}_n\left(\widetilde{x}\right)}}
\end{equation*}
The condition (b) of the Definition \ref{local_sec_thiples_defn} follows from $\widetilde{\Om}_{\widetilde{x}} \subset \widetilde{M}$. Really if $\left\{\xi_n \in L^2\left(M_n,S_n\right)\right\}_{n \in \mathbb{N}}$ is such that $\xi_n \in 1_{\Om^n_{\widetilde{p}_n\left(\widetilde{x}\right)}}L^2\left(M_n,S_n\right)$ then from $_{C(M)}X_{C_0(\widetilde{M})} = C_0\left(\widetilde{M}\right) \otimes_{C_0\left(\overline{M}\right)}\overline{X}_{C\left(M\right)}$ it follows that
\begin{equation*}
\mathfrak{Rep}_{\overline{H}}\left(\left\{\xi_n\right\}\right)\in 1_{\widetilde{\Om}_{\widetilde{x}}}\overline{H} =  _{C(M)}X_{C_0(\widetilde{M})}\otimes_{C\left(\overline{M}\right)}\overline{H}=\widetilde{H}.
\end{equation*}
\end{proof}

\begin{rem}\label{locality_explanation_rem}
	The sequence is local because the operator $\widetilde{\slashed D}$ can be defined as gluing of local operators, i.e. operators defined on one-to-one subsets.
\end{rem}
\begin{empt}\label{rho_x_constr}
	Let $M$ be a manifold with a $\mathrm{Spin}^{\mathbf{c}}$ structure and $m\in\mathbb{N}$ is such that $\dim~M = 2m$ or $\dim~M = 2m +1$. If $\left(M, \pi, S\right)$ is the spinor bundle and $x \in M$ is any point then from \cite{hajac:toknotes} it follows that $S_x = \pi^{-1}\left(x\right)$ is a complex Hilbert space of dimension $2^m$ and there is a natural representation $\rho_x: \Coo\left(M\right) \to B\left(S_x\right)$ and
	\begin{equation*}
	\left\|a\right\|= \sup_{x \in M} \left\|\rho_x(a)\right\|.
	\end{equation*}
 For any $x \in M$ there is a representation  $\rho^1_x: \Coo\left(M\right)\to B\left(S^2_x\right)$ which corresponds to the representation $\pi^1: \Coo\left(M\right) \to B\left(\left(L^2\left(M,S\right)\right)^2\right)$ given by 
	\begin{equation*}
\pi^1\left(a\right)=\begin{pmatrix} a & 0\\
[\slashed D,a] & a\end{pmatrix} \in B\left(\left(L^2\left(M, S\right)\right)^2\right).
	\end{equation*}
Similarly to \ref{s_repr} for any $s \in \mathbb{N}$ and $x \in M$ we can  inductively define   $\rho^s_x: \Coo\left(M\right)\to B\left(S^{2^s}_x\right)$ which corresponds to $\pi^s: \Coo\left(M\right) \to B\left(\left(L^2\left(M,S\right)\right)^{2^s}\right)$ given by
\begin{equation}\label{pi_s_comm_eqn}
\pi^{s}\left(a\right) =  \begin{pmatrix}  \pi^{s-1}(a) & 0 \\ \left[\slashed D,\pi^{s-1}(a)\right] &  \pi^{s-1}(a)\end{pmatrix}
\end{equation}	

\end{empt}	
\begin{defn}\label{smooth_mod_el_defn}
	Let $M$ be an orientable  Riemannian manifold with a $\mathrm{Spin}^{\mathbf{c}}$ structure and let $\widetilde{\pi} : \widetilde{M} \to M$ be a covering projection. There is the unbounded Dirac operator $\slashed D$ on $L^2\left(M, S\right)$ and let $\widetilde{\slashed D}$ be the $\widetilde{\pi}$-pullback of $\slashed D$. An element $\xi \in \mathscr{L}^2\left(\widetilde{M}_{M}\right)$ of associated with $\widetilde{\pi}:\widetilde{M} \to M$  Hilbert $C(M)$-module is said to be {\it smooth} if following conditions hold
	\begin{enumerate}
		\item $\xi \in \Coo\left(\widetilde{M}\right)$.
		\item If $x \in M$ then for any $s \in \mathbb{N}$ and following condition hold
		\begin{equation*}
	\varphi_\xi\left(x\right) =	\sum_{y \in \widetilde{\pi}^{-1}\left(x\right)}\left\|\rho^s_y\left(\xi\right)\right\|^2 ;~\forall x \in M
		\end{equation*}
		then $\varphi_\xi \in C\left(M\right)$.
	\end{enumerate}
\end{defn}	
\begin{defn}\label{l2_m_m_norm_defn}
	If $\Xi \subset \mathscr{L}^2\left(\widetilde{M}_{M}\right)$ be a linear span of smooth elements then there is the Fr\'echet topology on 
	$\Xi$ induced by seminorms $\left\|\cdot\right\|_s$ given by
		\begin{equation*}
		\left\|\xi\right\|_s= \sup_x \sqrt{	\sum_{y \in \widetilde{\pi}^{-1}\left(x\right)}\left\|\rho^s_y\left(\xi\right)\right\|^2}.
		\end{equation*}
	The completion of $\Xi$ with respect to  this the Fr\'echet topology is said to be the {\it smooth module associated} with $\widetilde{\pi}:\widetilde{M} \to M$. Denote by $\mathscr{L}^2_\infty\left(\widetilde{M}_{M}\right)$ the smooth module associated with $\widetilde{\pi}:\widetilde{M} \to M$. 
\end{defn}
\begin{empt}\label{lip_slased_d_const}
 From \cite{varilly:noncom} it follows that Reimannian metric satisfies to the following condition
 \begin{equation*}
 \mathfrak{dist}\left(x, y\right) = \sup \left\{\left|a
 \left(x\right)-a
 \left(y\right)\right|~|~ a \in \Coo\left(M\right), ~ \left\|\left[\slashed D, a\right]\right\| \le 1\right\}.
 \end{equation*}
 Conversely if $\left\|\left[\slashed D, a\right]\right\| = C$ then 
 \begin{equation*}
 \left|a
 \left(x\right)-a
 \left(y\right)\right| \le C ~\mathfrak{dist}\left(x, y\right).
 \end{equation*}
 Similarly if $\pi^s(a)$ is given by \eqref{pi_s_comm_eqn} and $\left\|\left[\slashed D,\pi^{s+1}(a)\right]\right\| = C_{s+1}$ then for any $x,y \in M$
 \begin{equation}\label{lip_slased_d_equ}
 \left\|\pi^s(a)
 \left(x\right)-\pi^s(a)
 \left(y\right)\right\| \le C_{s+1} ~\mathfrak{dist}\left(x, y\right).
 \end{equation}
 \end{empt}

\begin{lem}
	Any element $\xi \in \mathscr{L}^2_\infty\left(\widetilde{M}_{M}\right)$ corresponds to the function $a \in C_0\left(\widetilde{M}\right)$ such that for any $s \in \mathbb{N}$ the function $f^s: \widetilde{M} \to \mathbb{R}$ given by
	\begin{equation*}
	f^s\left(x\right)= 		\left\|\rho^s_x\left(a\right)\right\|^2
	\end{equation*}
	
	is continuous. 
\end{lem}
\begin{proof}
	From the Lemma \ref{l2_comp} it follows that $\mathscr{L}^2\left(\widetilde{M}_{M}\right) \subset C_0\left(\widetilde{M}\right)$, so 	any element $\xi \in \mathscr{L}^2_\infty\left(\widetilde{M}_{M}\right)$ corresponds to the function $a \in C_0\left(\widetilde{M}\right)$.
Now this lemma follows from the inequality \eqref{lip_slased_d_equ}.
\end{proof}

\begin{defn}\label{coo0_defn}
Let us consider the situation of the Definition \ref{smooth_mod_el_defn}. We say that the element $a \in \Coo\left(\widetilde{M}\right)$ is \textit{zero at infinity with derivations} if for any $s \in \mathbb{N}$, $\varepsilon > 0$ there is a compact set $K_{s, \varepsilon}\subset \widetilde{M}$ such that
\begin{equation*}
\left\|\rho^s_x\left(a\right)\right\| < \varepsilon;~\forall x \in \widetilde{M} \backslash K_{s, \varepsilon}.
\end{equation*}	
Algebra of zero an infinity with derivations elements will be denoted by $\Coo_0\left(\widetilde{M}_M\right)$. 
\end{defn}
\begin{lem}\label{l2_in_smooth}
		Let $M$ be an orientable  Riemannian manifold with a $\mathrm{Spin}^{\mathbf{c}}$ structure. If $\widetilde{\pi} : \widetilde{M} \to M$ is a covering projection and $G\left(\widetilde{M}~|~M\right)$ is countable then $\mathscr{L}^2_\infty\left(\widetilde{M}_{M}\right) \subset \Coo_0\left(\widetilde{M}_M\right)$.
\end{lem}
\begin{proof}
	Proof of this lemma is similar to the proof of the Lemma \ref{l2_comp}.
	Let $\left\{\widetilde{U}_{\iota}\subset\widetilde{M}\right\}_{\iota \in I}$ be a basis of the fundamental covering of $\widetilde{\pi}:\widetilde{M}\to M$. Since $M$ is compact we can select finite family  $\left\{\widetilde{U}_{\iota}\subset\widetilde{M}\right\}_{\iota \in I}$ , i.e.  $\left\{\widetilde{U}_{\iota}\right\}_{\iota \in I} = \left\{\widetilde{U}_1, ..., \widetilde{U}_n\right\}$. Let 
	\begin{equation*}
	G_1 \leftarrow G_2 \leftarrow ...
	\end{equation*}
	be a coherent sequence of finite groups $G_i = G\left(M_i| M\right)$ with epimorphisms $h_i: G \to G_i$, and let $\left\{G^k\subset G\right\}_{k \in \mathbb{N}}$  be a $G$-covering (See the Definition \ref{g_cov_defn}). 
	If $\widetilde{V}=\bigcup_{i = 1,...,n}\widetilde{U}_i$ and $K=\mathfrak{cl}\left(\widetilde{V}\right)$ is the closure of $\widetilde{ V}$ then $K$ is compact. For any $k \in \mathbb{N}$ the set $K_k = G^kK$ is a finite union of compact sets, whence $K_k$ is compact. If $\widetilde{V}_k = \bigcup_{k \in \mathbb{N}}G^k\widetilde{V}$ then from the Definition \ref{fund_domu} it follows that
	$\widetilde{M}= \bigcup_{k \in \mathbb{N}}  \widetilde{V}_k$. Let $\varphi \in \mathscr{L}_\infty^2\left(\widetilde{M}_{M}\right)$ be such that $\varphi \notin  \Coo_0\left(\widetilde{M}_M\right)$. From the Definition \ref{coo0_defn} it follows that there are $s \in \mathbb{N}$ and $\varepsilon > 0$ such that for any compact set $K \subset \widetilde{U}$ there is $\widetilde{x} \in \widetilde{\mathcal X} \backslash K$ such that $\left\|\rho^s_{\widetilde{x}}\left(\varphi\right)\right\| > \varepsilon$, where $\rho^s_{\widetilde{x}}$ is defined in \ref{rho_x_constr}.  From \ref{lip_slased_d_const} it follows that for any $s \in \mathbb{N}$ the function $\widetilde{x} \mapsto \left\|\rho^s_{\widetilde{x}}\right\|$ is continuous. Let us define a sequence $\left\{\widetilde{x}_i\in \widetilde{M} \right\}_{i \in \mathbb{N}}$ such that $\left|\rho^s_x\left(\widetilde{x}\right)\right|> \varepsilon$ and $\widetilde{x}_i\in \ \widetilde{M} \backslash K_i$.
	There is a sequence $\left\{x_i \in M\right\}_{i \in \mathbb{N}}$ given by $x_i = \pi \left(\widetilde{x}_{i_j}\right)$. Since $M$ is  compact the sequence $\left\{x_i \right\}_{i \in \mathbb{N}}$ contains a convergent subsequence $\left\{x_{i_j} \right\}_{j \in \mathbb{N}}$. Let $x = \lim_{j \to \infty} x_{i_j}$. Let $\widetilde{x} \in \widetilde{M}$ be such that $\widetilde{\pi}\left(\widetilde{x}\right)=x$ and $\widetilde{x} \in \widetilde{V}$. From the Definitions \ref{smooth_mod_el_defn}, \ref{l2_m_m_norm_defn}  it follows that the series
	\begin{equation*}
	\sum_{g \in G}\left\|\rho^s_{g\widetilde{x}}\left(\varphi\right)\right\|^2
	\end{equation*}
	is convergent, whence there is $r \in \mathbb{N}$ such that
	\begin{equation}\label{min_phi_2s}
	\sum_{g \in G \backslash G^r}\left\|\rho^s_{g\widetilde{x}}\left(\varphi\right)\right\|^2 < \frac{\varepsilon^2}{2}.
	\end{equation}
	If $\widetilde{W}$ is an open connected neighborhood of $\widetilde{x}$ which is mapped homeomorphicaly onto $W=\widetilde{\pi}\left(\widetilde{W}\right)$ and $\widetilde{W} \subset \widetilde{V}$ then there is  a real continuous function $\psi:\widetilde{W} \to \mathbb{R}$ given by
	\begin{equation*}
	\psi\left(y\right)= \sum_{g \in G \backslash G^r} \left\|\rho^s_{g\widetilde{y}}\left(\varphi\right)\right\|^2; \text{ where } \widetilde{y} \in \widetilde{W} \text{ and } \widetilde{\pi}\left(\widetilde{y}\right)=y.
	\end{equation*}
	There is $r \in \mathbb{N}$ such that $i_r > r$ and $x_{i_j} \in  W$  for any $j \ge r$. If $j > r$  then from  $\left|\varphi\left(\widetilde{x}_{i_j}\right)\right| > \varepsilon$ and $\widetilde{x}_{i_j} \notin K_r$ it follows that 
	\begin{equation*}
	\psi\left(x_{i_j}\right) = \sum_{g \in G \backslash G^r} \left\|\rho^s_{g\widetilde{x}'_{i_j}}\left(\varphi\right)\right\|^2 \ge \left\|\rho^s_{g\widetilde{x}_{i_j}}\left(\varphi\right)\right\|^2 > \varepsilon^2; \text{ where } \widetilde{x}'_{i_j} \in \widetilde{W} \text{ and } \widetilde{\pi}\left(\widetilde{x}'_{i_j}\right)=x_{i_j}.
	\end{equation*}
	Since $\psi$ is continuous and $x = \lim_{j \to \infty} x_{i_j}$ we have $\psi\left(x\right) > \varepsilon^2$. This fact contradicts to the equation \eqref{min_phi_2s}, and the contradiction proves the lemma.
	
\end{proof}

\begin{lem}\label{l2_sub_x_inf_lem}
	 
 If $X_{\Coo\left(M\right)}^\infty$ is the connected smooth module (Definition \ref{smooth_mod_defn}) of the coherent sequence described in the Lemma \ref{comm_coherent_lemma} then there is the natural isomorphism of $\Coo\left(M\right)$-modules
 \begin{equation*}
\mathscr{L}^2_\infty\left(\widetilde{M}_{M}\right) \approx X_{\Coo\left(M\right)}^\infty. 
 \end{equation*}
\end{lem}
\begin{proof}
	1) Inclusion $\mathscr{L}^2_\infty\left(\widetilde{M}_{M}\right) \subset X_{\Coo\left(M\right)}^\infty$.
	\newline
Let $G = G\left(\widetilde{M}|M\right)$ be the	group of covering transformations.
Since $M$ is compact there is a finite  basis $\left\{\widetilde{U}_1, ..., \widetilde{U}_n\right\}$ of the fundamental covering of $\widetilde{\pi}:\widetilde{M}\to M$. From the Proposition \ref{smooth_part_unity_prop} it follows that there exist a partition of unity
\begin{equation*}
1_{C_b\left(\widetilde{M}\right)} = \sum_{g \in G }  \sum_{i=1}^{n}  g\widetilde{a}_i = \sum_{\left(g,\iota\right)\in G \times \{1,...,n\}}\widetilde{a}_{\left(g,i\right)}
\end{equation*}
dominated by $\left\{\widetilde{U}_1, ..., \widetilde{U}_n\right\}$ such that $\widetilde{a}_i \in \Coo\left(\widetilde{M}\right)$ for any $i \in \{1,...,n\}$.

If $\widetilde{e}_i = \sqrt{\widetilde{a}_i}$  then $\widetilde{e}_i \in \Coo\left(\widetilde{M}\right)$ and from the Corollary  \ref{cor_xi1} it follows that
\begin{equation}\label{smooth_unity_eqn}
\sum_{i \in \{1,...,n\}, g \in G} g \mathfrak{Rep}\left(\widetilde{e}_i \right) \rangle \langle g \mathfrak{Rep}\left(\widetilde{e}_i \right)  = 1_{M\left(\mathcal{K}\left(X_{C\left( M\right)}\right)\right)}.
\end{equation}
From $\widetilde{e}_i \in \Coo\left(\widetilde{M}\right)$ it follows that the descent   $\mathfrak{Desc}\left(\widetilde{e}_i\right)$ of $\widetilde{e}_i$ (Definition \ref{desc_defn}) is a smooth coherent sequence (Definition \ref{smooth_coh_defn}), whence $\mathfrak{Rep}\left(\widetilde{e}_i\right) \in X^\infty_{\Coo\left(M\right)}$.
If for any $a \in \Coo\left(M\right)$ and $s \in \mathbb{N}$ we denote 
\begin{equation*}
\left\|a\right\|_s= \left\|\pi^s\left(a\right)\right\|;~ \text{ where } \pi^s \text{ is given by \eqref{pi_s_comm_eqn}}
\end{equation*}
then from 
\begin{equation*}
C_s = \max_{i \in \{1,...,n\}}\left\|\widetilde{e}_i\right\|_s.
\end{equation*}
it follows that
\begin{equation*}
\left\|\widetilde{e}_ia\right\|_s < C_s\left\|a\right\|_s,
\end{equation*}
whence $ \mathfrak{Rep}\left(g\widetilde{e}_i \right)a\in X^\infty_{\Coo\left(M\right)}$. If $\varphi \in \mathscr{L}^2_\infty\left(\widetilde{M}_{M}\right)$ then form $\mathscr{L}^2_\infty\left(\widetilde{M}_{M}\right)\subset \mathscr{L}^2\left(\widetilde{M}_{M}\right)$
and from the Lemma \ref{l2_comp} it follows that $\varphi$ defines a unique element $\xi_\varphi \in X_{C(M)}$. From \eqref{smooth_unity_eqn} it follows that 
\begin{equation*}
\xi_\varphi = \sum_{i \in \{1,...,n\}, g \in G} g \mathfrak{Rep}\left(\widetilde{e}_i \right)  \left\langle \xi, g \mathfrak{Rep}\left(\widetilde{e}_i \right)\right\rangle_{X_{C\left(M\right)}}
\end{equation*}
All summands of the above series belong to $X^\infty_{\Coo\left(M\right)}$. Let us prove that the series is convergent in the  Fr\'echet topology given by the Definition \ref{smooth_mod_defn}. Let $s \in \mathbb{N}$ be any natural number and $\varepsilon > 0$. From the Definition \ref{smooth_mod_el_defn} it follows that there is a compact set $K$ such that for any $\widetilde{x}\in \widetilde{M} \backslash K$ following condition hold
\begin{equation*}
\rho^s_{\widetilde{x}}\left(\varphi\right) < \frac{\varepsilon}{C_s}
\end{equation*}
Since $K$ is compact there is a finite subset $G' \in G$ such that
\begin{equation*}
\left(\sum_{g \in G' }  \sum_{i=1}^{n}  g\widetilde{a}_i\right)\left(K\right)= \{1\}.
\end{equation*}
 From above equations it follows that
 \begin{equation}
 \left\|\xi_\varphi - \sum_{g \in G' }  \sum_{i=1}^{n} g \mathfrak{Rep}\left(\widetilde{e}_i \right)  \left\langle \xi, g \mathfrak{Rep}\left(\widetilde{e}_i \right)\right\rangle_{\overline{X}_A}\right\|_s < \varepsilon.
 \end{equation}
 So the series is convergent in the Fr\'echet topology, whence $\varphi_\xi \in \overline{X}^\infty_{\Coo\left(M\right)}$.
 \newline
 	2) Inclusion $X_{\Coo\left(M\right)}^\infty \subset\mathscr{L}^2_\infty\left(\widetilde{M}_{M}\right)$.
 	\newline
 	If $\xi \in X_{\Coo\left(M\right)}^\infty$   then form $X_{\Coo\left(M\right)}^\infty\subset X_{C\left(M\right)}$
 	and from the Lemma \ref{l2_comp} it follows that $\xi$ defines a unique element $\varphi_\xi \in \mathscr{L}^2\left(\widetilde{M}_{M}\right)$. If $\varphi_\xi \notin \mathscr{L}^2_\infty\left(\widetilde{M}_{M}\right)$ then there are $s \in \mathbb{N}$ and $\varepsilon > 0$ such that for any compact set $K$ there is $\widetilde{x} \in \widetilde{M} \backslash K$ such that a following condition holds
 	\begin{equation*}
 \left\|\rho^s_{\widetilde{x}}
 \left({\widetilde{\varphi}_\xi}\right)\right\|_s> \varepsilon.
 	\end{equation*}
 Let us define a sequence $\left\{\widetilde{x}_i\in \widetilde{M} \right\}_{i \in \mathbb{N}}$ such that $\left|\left|\rho^s_{\widetilde{x}_i}\left(\varphi\right)\right|\right|> \varepsilon$ and $\widetilde{x}_i\in \ \widetilde{M} \backslash K_i$.
There is a sequence $\left\{x_i \in \mathcal X\right\}_{i \in \mathbb{N}}$ given by $x_i = \pi \left(\widetilde{x}_{i_j}\right)$. Since $M$ is  compact the sequence $\left\{x_i \right\}_{i \in \mathbb{N}}$ contains a convergent subsequence $\left\{x_{i_j} \right\}_{j \in \mathbb{N}}$. Let $x = \lim_{j \to \infty} x_{i_j}$. Let $\widetilde{x} \in \widetilde{M}$ be such that $\widetilde{\pi}\left(\widetilde{x}\right)=x$ and $\widetilde{x} \in \widetilde{V}$. From \eqref{hilb_cond_comm} it follows that the series
\begin{equation*}
	\sum_{g \in G}\left|\varphi\left(g\widetilde{x}\right)\right|^2
\end{equation*}
is convergent, whence there is $r \in \mathbb{N}$ such that
\begin{equation}\label{min_phi_2_pi}
	\sum_{g \in G \backslash G^r}\left\|\rho^s_{g\widetilde{x}}\left(\varphi\right)\right\|^2 < \frac{\varepsilon^2}{2}.
\end{equation}
Let $\widetilde{W}$ be an open connected neighborhood of $\widetilde{x}$ which is mapped homeomorphicaly onto $W=\widetilde{\pi}\left(\widetilde{W}\right)$ such that following conditions hold:
\begin{enumerate}
	\item If $\widetilde{y} \in \widetilde{W}$ then $\mathfrak{dist}\left(\widetilde{y}, \widetilde{x}\right) < \frac{\varepsilon}{4 \left\|\left[\slashed D~,~\pi^{s+1}(\varphi)\right]\right\|}$. From this condition and \eqref{lip_slased_d_equ} it follows that
	\begin{equation}\label{deriv_estimate}
	\left\|\rho^s_{g\widetilde{y}}\left(\varphi\right)-\rho^s_{g\widetilde{x}}\left(\varphi\right)\right\| < \frac{\varepsilon}{4}; ~ \forall y \in \widetilde{W}.
	\end{equation}
	\item  $\widetilde{W} \subset \widetilde{ V}$.
\end{enumerate}

There is  a real continuous function $\psi:\widetilde{W} \to \mathbb{R}$ given by
\begin{equation*}
	\psi\left(y\right)= \sum_{g \in G \backslash G^r} \left\|\rho^s_{g\widetilde{y}}\left(\varphi\right)\right\|^2; \text{ where } \widetilde{y} \in \widetilde{W} \text{ and } \widetilde{\pi}\left(\widetilde{y}\right)=y.
\end{equation*}
There is $s \in \mathbb{N}$ such that $i_s > r$ and $x_{i_j} \in  W$  for any $j \ge s$. If $j > s$  then from  $\left\|\rho^s_{g\widetilde{x}_{i_j}}\right\| > \varepsilon$ and $\widetilde{x}_{i_j} \notin K_r$ it follows that 
\begin{equation*}
	\psi\left(x_{i_j}\right) = \sum_{g \in G \backslash G^r} \left\|\rho^s_{g\widetilde{x}'_{i_j}}\left(\varphi\right)\right\|^2 \ge \left\|\rho^s_{g\widetilde{x}_{i_j}}\left(\varphi\right)\right\|^2 > \varepsilon^2; \text{ where } \widetilde{x}'_{i_j} \in \widetilde{\mathcal W} \text{ and } \widetilde{\pi}\left(\widetilde{x}'_{i_j}\right)=x_{i_j}.
\end{equation*}
Since $\psi$ is continuous and $x = \lim_{j \to \infty} x_{i_j}$ we have $\psi\left(x\right) > \varepsilon^2$. This fact contradicts to equations \eqref{min_phi_2_pi},\eqref{deriv_estimate}, and the contradiction proves the lemma.  
\end{proof}
\begin{lem}
 Let us consider the coherent sequence described in the Lemma \ref{comm_coherent_lemma}. If  $\mathcal{K}^\infty\left(X^\infty_{\Coo\left(M\right)}\right)$ is the smoothly compact subalgebra (Definition \ref{smoothly_comp_sub_defn}) then
 \begin{equation}
\Coo_0\left(\widetilde{M}_M\right) =C_0\left(\widetilde{M}\right)\bigcap\mathcal{K}^\infty\left(X_{\Coo\left(M\right)}\right).
 \end{equation}
 
\end{lem}
\begin{proof}

	1) Inclusion $C_0\left(\widetilde{M}\right)\bigcap\mathcal{K}^\infty\left(X_{\Coo\left(M\right)}\right)\subset \Coo_0\left(\widetilde{M}_M\right)$.
	\newline
If $a \in C_0\left(\widetilde{M}\right)\bigcap\mathcal{K}^\infty\left(X_{\Coo\left(M\right)}\right)$  given by 
\begin{equation*}
a = \sum_{i=1}^{\infty} \xi_i \rangle\langle \eta_i; ~~ \xi_i, \eta_i \in \overline{X}^\infty_{\Coo\left(M\right)}
\end{equation*}
then from \eqref{norm_ineq} it follows that
\begin{equation}\label{smooth_comp_series}
a = \sum_{i=1}^{\infty} \xi^*_i \eta_i
\end{equation}
where $\xi_i, \eta_i$ are being regarded as elements of $C_0\left(\widetilde{M}\right)$. From the Lemma \ref{l2_in_smooth} it follows that $\xi_i \eta_i \in \Coo_0\left(\widetilde{M}_M\right)$ whence $a_n = \sum_{i=1}^{n} \xi^*_i \eta_i \in \Coo_0\left(\widetilde{M}_M\right)$ for any $n \in \mathbb{N}$.  From the Definition \ref{smoothly_comp_sub_defn} it follows that the series \eqref{smooth_comp_series} is convergent in the Fr\'echet topology induced by seminorms $\left\|\cdot\right\|_s$, therefore $a \in \Coo_0\left(\widetilde{M}_M\right)$.
\newline
2) Inclusion
$\Coo_0\left(\widetilde{M}_M\right) \subset C_0\left(\widetilde{M}\right)\bigcap\mathcal{K}^\infty\left(X_{\Coo\left(M\right)}\right)$.
\newline
Let $a \in  \Coo_0\left(\widetilde{M}_M\right)$. If $\widetilde{e}_1,...,\widetilde{e}_n\in C^\infty\left(\widetilde{M}\right)$ are defined in the Lemma \ref{l2_sub_x_inf_lem} then following condition hold
\begin{equation*}
a = \sum_{g \in G}\sum_{i = 1}^{n} g\widetilde{e}_i\langle g \widetilde{e}_i,a\rangle_{\overline{X}_{C\left(M\right)}}= \sum_{g \in G}\sum_{i = 1}^{n} g\widetilde{e}_i
\rangle\langle g \widetilde{e}_ia = \sum_{g \in G}\sum_{i = 1}^{n} \left(g\widetilde{e}_i\right)\left(g\widetilde{e}_ia\right)
\end{equation*}
and the series is convergent in the Fr\'echet topology induced by seminorms $\left\|\cdot\right\|_s$. From this fact and it follows that $a \in C_0\left(\widetilde{M}\right)\bigcap\mathcal{K}^\infty\left(X_{\Coo\left(M\right)}\right)$.

\end{proof}
\begin{lem}
	The coherent sequence described in the Lemma \ref{comm_coherent_lemma} is regular.
\end{lem}
\begin{proof}
For any $\widetilde{a} \in C_0\left(M\right)$ and $\varepsilon > 0$ there is a finite subset $G'\subset G$ such that
\begin{equation*}
\left\| \widetilde{a}-\sum_{i=1}^{m}\sum_{g\in G'}a_{ig}\left(g \widetilde{e}_i\right)^2\right\| < \frac{\varepsilon}{2}
\end{equation*}
	where $a_{ig}\in A$ and $\widetilde{e}_1,...,\widetilde{e}_n \in C^\infty\left(\widetilde{M}\right)$ are defined in the Lemma \ref{l2_sub_x_inf_lem}. Since $C^\infty\left(M\right)$ is dense in $C\left(M\right)$ there are $a'_{ig}\in C^\infty\left(M\right)$ such that $\left\|a'_{ig} - a_{ig}\right\|\le \frac{\varepsilon}{2\left|G'\right|}$, so
	\begin{equation*}
	\widetilde{a}' =\sum_{i=1}^{m}\sum_{g\in G''}a'_{ig}\left(g \widetilde{e}_i\right)^2 \in C_0^\infty\left(\widetilde{M}_M\right)
	\end{equation*}
	and
	\begin{equation*}
	\left\|\widetilde{a}' -\sum_{i=1}^{m}\sum_{g\in G''}a_{ig}\left(g \widetilde{e}_i\right)\right\| < \frac{\varepsilon}{2}
	\end{equation*}
	In result we have 	$\left\|\widetilde{a} -\widetilde{a}'\right\| < \varepsilon$, whence $\Coo\left(\widetilde{M}_M\right)$ is dense in $C\left(\widetilde{M}\right)$ because  $a'\in \Coo\left(\widetilde{M}_M\right)$.

\end{proof}
\section{Covering projection of the noncommutative torus}\label{nt_sec}
\subsection{Covering projection of the $C^*$-algebra}

\begin{empt}	Let $A_{\theta}$ be a noncommutative torus generated by  unitaries $u, v \in U\left(A_{\theta}\right)$ and  let 
	\begin{equation}\label{fin_torus_groups}
	\mathbb{Z}_{n_1}  \times \mathbb{Z}_{m_1}  \xleftarrow{}...\xleftarrow{} \mathbb{Z}_{n_k} \times \mathbb{Z}_{m_k}  \xleftarrow{} ...
	\end{equation}
	be an infinite sequence of finite  groups. Let $\theta_0= \theta$ and let us inductively define $\theta_k$ for any $k \in \mathbb{N}$ and *-homomorphism $A_{\theta_{k-1}}\to A_{\theta_{k}}$ such that
	\begin{enumerate}
		\item 
		\begin{equation*}
		\theta_k= \left\{
		\begin{array}{c l}
		\frac{\theta + 2\pi i_1}{m_1n_1} & k=1   \\ \\
		\frac{\theta_{k-1}+2\pi i_k}{\frac{m_k}{m_{k-1}}\frac{n_k}{n_{k-1}}} & k > 1    \end{array}\right.
				\end{equation*}
				where $i_1, ... , i_k, ... \in \mathbb{N}^0$ arbitrary nonnegative numbers.
				\item $A_{\theta_{k}}$ is generated by two unitary elements $u_k, v_k$ and the *-homomorphism $A_{\theta_{k-1}} \to A_{\theta_{k}}$ is given by
				\begin{equation*}
				u_{k-1} \mapsto u^m_k;~v_{k-1} \mapsto v^n_k
				\end{equation*}
				where
			\begin{equation*}
			m \text{ (resp. } n \text{)} = \left\{
			\begin{array}{c l}
				m_1 \text{ (resp. } n_1\text{)}  & k=1   \\ \\
			\frac{m_k}{m_{k-1}} \text{ (resp. }\frac{n_k}{n_{k-1}}\text{)} & k > 1    \end{array}\right.
			\end{equation*}				
	\end{enumerate}
	In \cite{ivankov:inv_lim} it is shown that the sequence
	\begin{equation}\label{nt_seq}
	A_\theta = A_{\theta_0} \to  A_{\theta_1} \to ...
	\end{equation}
	is a composable sequence of finite noncommutative covering projections.
	If $e_i$ (resp. $e^n_i$) is defined by \eqref{e_1_e_2} \eqref{e_n_i} then
	any $i, j \in \{1,2\}$ following sequences 
	\begin{equation*}
	\Lambda_{i,j}=\left\{e^{m_k}_i\left(u_{m_k}\right)e^{n_k}_j\left(v_{n_k}\right)\right\}_{k\in \mathbb{N}^0}; \ 
	\Lambda'_{i,j}=\left\{e^{n_k}_j\left(v_{n_k}\right)e^{m_k}_i\left(u_{m_k}\right)\right\}_{k\in \mathbb{N}^0}
	\end{equation*}
	are coherent. If $I = \{1, 2\} \times \{1,2\}$ and $\Lambda_{\iota=(i,j)}=\Lambda_{i,j}$ (resp.  $\Lambda'_{\iota=(i,j)}=\Lambda'_{i,j}$) $\forall \iota \in I$ then elements  $\xi_{\iota} = \mathfrak{Rep}\left(\Lambda_{\iota}\right)$, $\xi'_{\iota} = \mathfrak{Rep}\left(\Lambda'_{\iota}\right)$ satisfy conditions of the Corollary \ref{cor_xi1}. If $\Xi=\left\{\xi_{\iota}\right\}_{\iota \in I}$ then from the Corollary \ref{cor_xi2} it follows that the set  $\overline{G}\Xi  A_{\theta}$ is dense in $\overline{X}_{A_{\theta}}$. For any $(x,y)\in \mathbb{R}\times\mathbb{R}$ let  $(x,y) \bullet \mathfrak{Rep}\left(\Lambda_{i,j}\right)$ be given by
	
	\begin{equation}
	(x,y)\bullet\mathfrak{Rep}\left(\Lambda_{i,j}\right)=\mathfrak{Rep}\left(\left\{e^{m_k}\left(\mathrm{exp}\left(\frac{ix}{m_k}\right)u_{m_k}\right)e^{m_k}\left(\mathrm{exp}\left(\frac{iy}{n_k}\right)v_{n_k}\right)\right\}_{k\in \mathbb{N}^0}\right); 
	\end{equation}
	Since a linear span of $\overline{G}\Xi  A_{\theta}$ is dense in $\overline{X}_{A_{\theta}}$ the action of $\mathbb{R}\times\mathbb{R}$ can be uniquely extended to the continuous action $\left(\mathbb{R}\times\mathbb{R}\right) \times \overline{X}_A\to\overline{X}_A$, $\xi \mapsto (x,y) \bullet \xi$. This action induces a natural action  $\left(\mathbb{R}\times\mathbb{R}\right) \times \overline{A}_\theta\to\overline{A}_\theta$ on disconnected algebra.
	Since both $\mathbb{R}\times\mathbb{R}$ and a maximal irreducible subalgebra algebra $\widetilde{A}_\theta\subset \overline{A}_\theta$ are connected a following condition hold
	\begin{equation}\label{r_r_t}
	\left(\mathbb{R}\times\mathbb{R}\right) \bullet \widetilde{A}=\widetilde{A}.
	\end{equation}
	If we include $\mathbb{Z}\times\mathbb{Z} \subset \overline{G}$ then
	\begin{equation*}
	(i, j) a = (2\pi i, 2 \pi j) \bullet a; \ \forall a \in \overline{a}
	\end{equation*}
	and from \eqref{r_r_t} it follows that $\left(
	\mathbb{Z}\times\mathbb{Z}\right)\widetilde{A}=\widetilde{A}$. Otherwise if $g \in \left(\mathbb{Z}\times\mathbb{Z}\right)$, $g' \in \overline{G} \backslash \left(\mathbb{Z}\times\mathbb{Z}\right)$ and $\eta, \zeta \in \Xi$ then $\langle g \eta, g' \zeta\rangle_{ \overline{X}_{A_{\theta}}}=0$. From this fact it follows that the covering transformation group $G$ is equal to $G=\mathbb{Z}\times\mathbb{Z}$ and a linear span of   $\left(\mathbb{Z}\times\mathbb{Z}\right) \ \Xi \ A_{\theta}$ is dense in  $X_{A_{\theta}}$. There is a noncummutative infinite covering projection $\left(A_{\theta}, \ \widetilde{A}_{\theta}, \ _{\widetilde{A}_{\theta}}X_{A_{\theta}}, \ \mathbb{Z}\times\mathbb{Z}\right)$ of the sequence \eqref{nt_seq}.
\end{empt}

\begin{lem}\label{xij_eta_ij_lem}
	If for $i, j \in \{1,2\}$ elements $\xi_{ij},\eta_{ij}\in _{\widetilde{A}_{\theta}}X_{A_{\theta}}$ are given by
\begin{equation*}
	\xi_{ij} =\mathfrak{Rep}\left(\left\{e^{m_k}_i\left(u_{m_k}\right)e^{n_k}_j\left(v_{n_k}\right)\right\}_{k\in \mathbb{N}^0}\right),~ \eta_{ij}= \left\{e^{m_k}_i\left(u^*_{m_k}\right)e^{n_k}_j\left(v^*_{n_k}\right)\right\}_{k\in \mathbb{N}^0}\in _{\widetilde{A}_{\theta}}X_{A_{\theta}}
	\end{equation*}	
then $\xi_{ij} \rangle\langle \eta_{ij}\in~\overline{A}_{\theta}$. 
\end{lem}
\begin{proof}
	From the Definition \ref{main_defn} if follows that we need to show that
\begin{equation*}
\xi_{ij} \rangle\langle \eta_{ij}\in \left(\bigcup_{k \in \mathbb{N}}A_{\theta_k}\right)''. 
\end{equation*}	
If $a_k \in A_{\theta_k}$ is given by 
\begin{equation*}
a_k =e^{m_k}_i\left(u_{m_k}\right)e^{n_k}_j\left(v_{n_k}\right)e^{n_k}_i\left(v_{n_k}\right)e^{m_k}_j\left(u_{m_k}\right)
\end{equation*}	
 then it is clear that $a_k = \bigcup_{k \in \mathbb{N}}A_{\theta_k}$ and
 \begin{equation*}
 \lim_{k\to\infty}a_k = \xi_{ij} \rangle\langle \eta_{ij}
 \end{equation*}
 where we mean convergence in the weak topology.
\end{proof}
\begin{cor}
The sequence \eqref{nt_seq} of covering projections is faithful.	
\end{cor}
\begin{proof}

The natural homomorphisms $\mathbb{Z} \times \mathbb{Z}\to\mathbb{Z}_{m_k} \times \mathbb{Z}_{n_k}$ are epimorphic. From \eqref{circ_sum_nt} it follows that
\begin{equation*}
\sum_{g \in \mathbb{Z}_{m_k} \times \mathbb{Z}_{n_k};~,i,j \in \{1,2\}} \left(ge^{m_k}_i\left(u_{m_k}\right)e^{n_k}_j\left(v_{n_k}\right)\right)\left(ge^{n_k}_i\left(v_{n_k}\right)e^{m_k}_j\left(u_{m_k}\right)\right) = 1_{A_{\theta_k}},
\end{equation*}
whence for any nonzero $a \in A_{\theta_k}$ there are $g_0 \in \mathbb{Z}_{m_k} \times \mathbb{Z}_{n_k}$ and $i_0,j_0$ such that $g_0e^{n_k}_{i_0}\left(v_{n_k}\right)e^{m_k}_{j_0}\left(u_{m_k}\right)a\neq 0$. If $\widetilde{g} \in \mathbb{Z} \times \mathbb{Z}$ is such that $h_k\left(\widetilde{g}\right)=g_0$ then $ g_0\xi_{i_0j_0} \rangle\langle g_0\eta_{i_0j_0} a \neq 0$. But from the Lemma \ref{xij_eta_ij_lem} it follows that $g_0\xi_{i_0j_0} \rangle\langle g_0\eta_{i_0j_0} \in \overline{A}_\theta$. So the right action of $A_{\theta_k}$ on $\overline{A}_{\theta}$ is faithful. Similarly we can prove that the right action is also faithful.
\end{proof}
\begin{empt}\label{dense_nt_constr}
	From the Corollary \ref{cor_xi1} it follows that
	\begin{equation*}
	1_{M\left(\widetilde{A}_\theta\right)} = \sum_{g \in \mathbb{Z}\times\mathbb{Z}} ~\sum_{i,j \in \{1,2\}} g \xi_{ij} \rangle\langle g \eta_{ij} 
	\end{equation*}
	in sense of the strict convergence.
	Any $\kappa \in \mathcal{K}\left(_{\widetilde{A}_{\theta}}X_{A_{\theta}}\right)$ can be represented as sum of the norm convergent series
	\begin{equation*}
	\kappa = \sum_{k\in\mathbb{N}} \phi_k \rangle\langle \psi_k
	\end{equation*}
	whence
	\begin{equation*}
	\kappa = \left(\sum_{g \in \mathbb{Z}\times\mathbb{Z}} ~\sum_{i,j \in \{1,2\}} g \xi_{ij} \rangle\langle g \eta_{ij}\right)\sum_k \phi_k \rangle\langle \psi_k\left(\sum_{g \in \mathbb{Z}\times\mathbb{Z}} ~\sum_{i,j \in \{1,2\}} g \xi_{ij} \rangle\langle g \eta_{ij}\right)=
	\end{equation*}
	\begin{equation*}
	=\sum_{g,g' \in \mathbb{Z}\times\mathbb{Z}} ~\sum_{i,j,i',j' \in \{1,2\}}g \xi_{ij} \rangle a_{ijgi'j'g'}\langle g' \eta_{i'j'}=\sum_{g,g' \in \mathbb{Z}\times\mathbb{Z}} ~\sum_{i,j,i',j' \in \{1,2\}}g \xi_{ij}  a_{ijgi'j'g'}\rangle\langle g' \eta_{i'j'}
	\end{equation*}
	where
	\begin{equation*}
	a_{ijgi'j'g'} = \sum_{k\in\mathbb{N}} \left\langle g \eta_{ij}~,~ \kappa g'\xi_{i'j'}\right\rangle_{_{\widetilde{A}_{\theta}}X_{A_{\theta}}}\in A_\theta.
	\end{equation*}
	Let $\left\{G^k \subset\mathbb{Z}\times \mathbb{Z} \right\}_{k \in \mathbb{N}}$ be $\mathbb{Z}\times \mathbb{Z}$ covering of the sequence \eqref{fin_torus_groups} (Definition \ref{g_cov_defn}). Since $\kappa$ is compact for any $\varepsilon > 0$ there is $l \in \mathbb{N}$ such that if
	\begin{equation*}
	\overline{\kappa}=\sum_{g,g' \in G^l} ~\sum_{i,j,i',j' \in \{1,2\}}g \xi_{ij} \rangle a_{ijgi'j'g'}\langle g' \eta_{i'j'}
	\end{equation*}
	then
	\begin{equation*}
	\left\|\kappa - \overline{\kappa}\right\| < \frac{\varepsilon}{2}.
	\end{equation*}	
	If $\kappa \in \widetilde{A}_\theta$ then there is a sequence $\left\{a_k \in A_{\theta_k}\right\}_{k \in \mathbb{N}}$ such that
	\begin{equation*}
	\lim_{k \to\infty}a_k = \kappa
	\end{equation*}
	in the weak topology. For any $i,j,i',j' \in \{1,2\}$ and $g, g' \in G^l$ we will define the element $a^k_{ijgi'j'g'} \in  A_\theta$ given by
	\begin{equation*}
	a^k_{ijgi'j'g'}=
	\left\langle h_k\left(g\right)\left(e^{n_k}_j\left(v_{n_k}\right)e^{m_k}_i\left(u_{m_k}\right)\right)~,~a_k\left(h_k\left(g'\right)\left(e^{m_k}_{i'}\left(u_{m_k}\right)e^{n_k}_{j'}\left(v_{n_k}\right)\right)\right) \right\rangle_{A_{\theta_k}}
	\end{equation*}
	where $h_k : \mathbb{Z}\times \mathbb{Z} \to \mathbb{Z}_{m_k}\times \mathbb{Z}_{n_k}$ is the natural epimorphism of groups. 	From $\lim_{k \to\infty}a_k = \kappa$ it follows that
	\begin{equation*}
	\lim_{k \to\infty} a^k_{ijgi'j'g'} = a_{ijgi'j'g'}
	\end{equation*}
	and there is $k_0\in \mathbb{N}$ such that
	\begin{equation*}
	\left\|a_{ijgi'j'g'} - a^{k_0}_{ijgi'j'g'}\right\| < \frac{\varepsilon}{16 \left|G^l\right|}
	\end{equation*}
	and if
	\begin{equation*}
	\kappa_{\mathrm{fin}}=\sum_{g,g' \in G^l} ~\sum_{i,j,i',j' \in \{1,2\}}g \xi_{ij} \rangle a^{k_0}_{ijgi'j'g'}\langle g' \eta_{i'j'}
	\end{equation*}
	then
	\begin{equation}\label{kappa_fin_eps_eqn}
	\left\|\kappa - \kappa_{\mathrm{fin}}\right\| < \varepsilon.
	\end{equation}
	If  $ \left\{b_r \in A_{\theta_r}\right\}_{r \in \mathbb{N}}$
	is the sequence given by $b_r = 0$ for $r \le k$ and
	\begin{equation*}
	b_r = 
	\sum_{g,g' \in G^k} ~\sum_{i,j,i',j' \in \{1,2\}}h_r\left(g\right)\left(e^{n_r}_j\left(v_{n_r}\right)e^{m_r}_i\left(u_{m_r}\right)\right)a^{k_0}_{ijgi'j'g'}\left(h_r\left(g'\right)\left(e^{m_r}_{i'}\left(u_{m_r}\right)e^{n_r}_{j'}\left(v_{n_r}\right)\right)\right)
	\end{equation*}
	for $r > k$ then $\lim_{r \to \infty} b_r = \kappa_{\mathrm{fin}}$ in the weak topology i.e. $\kappa_{\mathrm{fin}} \in \widetilde{A}_\theta$.
\end{empt}
\
\begin{lem}\label{dense_nt_lem}
The space of operators $\kappa_{\mathrm{fin}}$ given by construction $\ref{dense_nt_constr}$ is dense in $\widetilde{A}_\theta$,
\end{lem}
\begin{proof}
	Follows from the Equation  \eqref{kappa_fin_eps_eqn}.
\end{proof}

\subsection{Covering projection of the spectral triple}

\paragraph{}
\begin{empt}\label{nt_spectral_triple}{\it Noncommutative torus as a spectral triple} \cite{hajac:toknotes,varilly:noncom}.

	There is a state $\tau_0: A_{\theta} \to \mathbb{C}$ given by
	\begin{equation*}
	\tau_0\left(\sum_{r,s} a_{rs}u^r v^s\right) = a_{00}.
	\end{equation*}.
	
	The GNS representation space $L^2(A_{\theta}, \tau_0)$ may be described as the completion of the vector space $A_{\theta}$ with respect to the Hilbert norm
	\begin{equation*}
	\|a\|_2= \tau_0\left(\sqrt{a^*a}\right).
	\end{equation*}
	
	Denote by $\underline{a}$ the image in $L^2
	\left(A_{\theta}, \tau_0\right)$ of $a \in A_{\theta}$.
	The Hilbert product	on $L^2(A_{\theta}, \tau_0)$ is given by
	\begin{equation*}
	\left(\underline{a}, \underline{b}\right)= \tau_0\left(a^*b\right).
	\end{equation*}

 Since $\underline{1}$ is cyclic and separating the Tomita involution is given by
	\begin{equation*}
	J_0(\underline{a})=\underline{a^*}.
	\end{equation*}
	To define structure of spectral triple we shall introduce double GNS Hilbert space $H = L^2(A_{\theta}, \tau_0) \oplus L^2(A_{\theta}, \tau_0)$ and define
	\begin{equation*}
	J =  \begin{pmatrix}
	0 & -J_0 \\
	J_0 & 0
	\end{pmatrix}
	\end{equation*}

There are two derivatives $\delta_1$, $\delta_2$ given by
\begin{equation*}
\delta_1 \left(\sum_{r, s}a_{r,s}u^rv^s\right)= \sum_{rs}2\pi i r a_{rs}u^rv^s,
\end{equation*}
\begin{equation*}
\delta_2 \left(\sum_{r, s}a_{r,s}u^rv^s\right)= \sum_{rs}2\pi i s a_{rs}u^rv^s.
\end{equation*}
which satisfy Leibniz rule, i.e.
\begin{equation*}
\delta_j(ab) = (\delta_ja)b = a(\delta_jb); \ (j=1,2; \ a, b\in \mathcal{A}_{\theta}).
\end{equation*}
There are derivations
\begin{equation}\label{deri_tau}
\partial = \partial_{\tau} = \delta_1 + \tau \delta_2; \ (\tau \in \mathbb{C}, \ \mathrm{Im}(\tau)\neq 0),
\end{equation}
\begin{equation*}
\partial^\dagger = -\delta_1 - \overline{\tau}\delta_2.
\end{equation*}
Actions of $\mathcal{A}_\theta$ and Dirac operator $D$ on $H$ are given by
\begin{equation*}
\pi(a)= \begin{pmatrix}
a & 0 \\
0 & a
\end{pmatrix},
\end{equation*}
\begin{equation*}
D= \begin{pmatrix}
0 & \underline{\partial}^\dagger \\
\underline{\partial} & 0
\end{pmatrix}.
\end{equation*}
There is a subalgebra $\mathcal{A}_{\theta} \subset A_\theta$ given by
	\begin{equation*}
	\mathcal{A}_{\theta} = \left\{a \in A_{\theta} \ | \ a = \sum_{r,s} a_{rs}u^r v^s \ \& \ \mathrm{sup}_{r, s \in \mathbb{Z}}\left(1 + r^2 + s^2\right)^k|a_{rs}| < \infty, \ \forall k \in \mathbb{N}\right\}.
	\end{equation*}
and a spectral triple $\left(\A_\theta, H, D\right)$.
\end{empt}
\begin{empt}\label{nt_st_descr}{\it Spectral triple of a noncommutative covering projection}.
Let $m, n, k \in \mathbb{N}$, $\widetilde{\theta}= \frac{\theta + 2 \pi k}{mn}$ and let $A_\theta \to A_{\widetilde{\theta}}$ be a *-homomorphism given by
\begin{equation*}
u \mapsto \widetilde{u}^m;~v \mapsto \widetilde{v}^n
\end{equation*}
where $\widetilde{u},\widetilde{v}\in U\left(A_{\widetilde{\theta}}\right)$ are unitary generators of $A_{\widetilde{\theta}}$. Similarly to \ref{nt_spectral_triple} we can define a smooth algebra
\begin{equation*}
	\mathcal{A}_{\widetilde{\theta}} = \left\{a \in A_{\widetilde{\theta}} \ | \ a = \sum_{r,s} a_{rs}\widetilde{u}^r \widetilde{v}^s \ \& \ \mathrm{sup}_{r, s \in \mathbb{Z}}\left(1 + r^2 + s^2\right)^k|a_{rs}| < \infty, \ \forall k \in \mathbb{N}\right\}.
 \end{equation*}
 and a state
	\begin{equation*}
	\widetilde{\tau}_0\left(\sum_{r,s} a_{rs}\widetilde{u}^r \widetilde{v}^s\right) = a_{00}.
	\end{equation*}
	The GNS representation space $L^2(A_{\widetilde{\theta}}, \widetilde{\tau}_0)$ may be described as the completion of the vector space $A_{\widetilde{\theta}}$ in the Hilbert norm
	\begin{equation*}
	\|a\|_2= \widetilde{\tau}\left(\sqrt{a^*a}\right).
	\end{equation*}
	
	Denote by $\underline{a}$ the image in $L^2\left(A_{\widetilde{\theta}}, \widetilde{\tau}_0\right)$ of $a \in A_{\widetilde{\theta}}$. Similarly to \ref{nt_spectral_triple} we define
	\begin{equation*}
	\widetilde{J}_0(\underline{a})=\underline{a^*}.
	\end{equation*}
	To define structure of spectral triple we shall introduce double GNS Hilbert space $\widetilde{H} = L^2\left(A_{\widetilde{\theta}}, \widetilde{\tau}_0\right) \oplus L^2\left(A_{\widetilde{\theta}}, \widetilde{\tau}_0\right)$ and define
	\begin{equation*}
	\widetilde{J} =  \begin{pmatrix}
	0 & -\widetilde{J}_0 \\
	\widetilde{J}_0 & 0
	\end{pmatrix}
	\end{equation*}
	
	Thee are two derivatives $\widetilde{\delta}_1$, $\widetilde{\delta}_2$
	\begin{equation*}
	\widetilde{\delta}_1 \left(\sum_{r, s}a_{r,s}\widetilde{u}^r\widetilde{v}^s\right)= \frac{1}{m}\sum_{rs}2\pi i r a_{rs}\widetilde{u}^r\widetilde{v}^s,
	\end{equation*}
	\begin{equation*}
\widetilde{\delta}_2 \left(\sum_{r, s}a_{r,s}\widetilde{u}^r\widetilde{v}^s\right)= \frac{1}{n}\sum_{rs}2\pi i s a_{rs}\widetilde{u}^r\widetilde{v}^s,	
\end{equation*}
	which satisfy Leibniz rule.
	There are derivations
	\begin{equation*}
	\widetilde{\partial} = \widetilde{\partial}_{\tau} = \widetilde{\delta}_1 + \tau \widetilde{\delta}_2,
	\end{equation*}
	\begin{equation*}
\widetilde{\partial}^\dagger = -\widetilde{\delta}_1 - \overline{\tau}\widetilde{\delta}_2.
	\end{equation*}
	Hilbert space of the spectral triple  $\widetilde{H}$, action of $\mathcal{A}_{\widetilde{\theta}}$ and Dirac operator $\widetilde{D}$ are given by
	\begin{equation*}
	\pi(a)= \begin{pmatrix}
	a & 0 \\
	0 & a
	\end{pmatrix},
	\end{equation*}
	\begin{equation*}
	\widetilde{D}= \begin{pmatrix}
	0 & \underline{\widetilde{\partial}}^\dagger \\
	\underline{\widetilde{\partial}} & 0
	\end{pmatrix}.
	\end{equation*}
	In result we have a spectral triple $\left({\A}_{\widetilde{\theta}}, \widetilde{H}, \widetilde{D}\right)$.
An application of this construction gives  a sequence of spectral triples $\left\{\left(\A_{\theta_n}, H_n, D_n\right)\right\}_{n\in \mathbb{N}^0}$  such that $A_{\theta_n}$ is the $C^*$-completion of $\A_{\theta_n}$ and $A_{\theta_n}$ is a member of the sequence \eqref{nt_seq}.
 
\end{empt}
\begin{lem}
	In the situation described in \ref{nt_st_descr} the sequence $\left\{\left(\A_{\theta_n}, H_n, D_n\right)\right\}_{n\in \mathbb{N}^0}$ is coherent.
\end{lem}
\begin{proof}
	We need to prove conditions 1-6 of the Definition \ref{coh_spec_triple_defn}. Conditions 1-3 follow from the construction \ref{nt_st_descr}. Let $\left(A_\theta, \widetilde{A}_{\theta}, \mathbb{Z}_m \times \mathbb{Z}_n\right)$ be a finite noncommutative covering projection from \ref{nt_st_descr}. If $g =\left(\overline{p},\overline{q}\right)\in \mathbb{Z}_m \times \mathbb{Z}_n$ then
	\begin{equation*}
	g \left(\sum_{r,s} a_{rs}\widetilde{u}^r \widetilde{v}^s\right)= \sum_{r,s} e^{\frac{2\pi i pr}{m}}e^{\frac{2\pi i qs}{n}} a_{rs}\widetilde{u}^r \widetilde{v}^s.
	\end{equation*}
	 From $\left|a_{rs}\right|= \left|a_{rs}e^{\frac{2\pi i pr}{m}}e^{\frac{2\pi i qs}{n}}\right|$ it follows that
\begin{equation*}
\begin{split}
\mathrm{sup}_{r, s \in \mathbb{Z}}\left(1 + r^2 + s^2\right)^k\left|a_{rs}\right| < \infty; ~ \forall k \in \mathbb{N} \Rightarrow \\ \Rightarrow \mathrm{sup}_{r, s \in \mathbb{Z}}\left(1 + r^2 + s^2\right)^k\left|e^{\frac{2\pi i pr}{m}}e^{\frac{2\pi i qs}{n}}a_{rs}\right| < \infty; ~ \forall k \in \mathbb{N},
\end{split}
\end{equation*}	
whence $g\A_{\widetilde{\theta}} =\A_{\widetilde{\theta}}$ and the Condition 4 of \ref{nt_st_descr} is proven. From $\sH^\infty = \bigcap_{k\in\bN} \Dom D^k = \A_\theta \oplus \A_\theta$ and from the definition of the scalar product on $L^2\left(A_{\widetilde{\theta}}, \widetilde{\tau}_0\right)$ it follows that the Condition 5 hold. Condition 6 directly follows from the definition of the operator $\widetilde{D}$.
\end{proof}
\begin{empt}\label{nt_local_constr}
	Let us consider a case of $\theta = 0$. In this case $A_{\theta=0}$ is a universal algebra generated  by unitary elements $\widehat{u}, \widehat{v}$ with the single relation $\widehat{u} \widehat{v} = \widehat{v} \widehat{u}$ and 
\begin{equation*}
\A_\theta = \A_0= \left\{a \in A_{0} \ | \ a = \sum_{r,s} a_{rs}u^r v^s \ \& \ \mathrm{sup}_{r, s \in \mathbb{Z}}\left(1 + r^2 + s^2\right)^k|a_{rs}| < \infty, \ \forall k \in \mathbb{N}\right\}=
\end{equation*}	
\begin{equation*}
=C^\infty\left(\mathbb{T}^2 = S^1 \times S^1\right)
\end{equation*}	
	There is a state $\widehat{\tau}_0: A_{0} \to \mathbb{C}$ given by
	\begin{equation*}
	\widehat{\tau}_0\left(\sum_{r,s} a_{rs}\widehat{u}^r \widehat{v}^s\right) = a_{00}.
	\end{equation*}
	and these is a GNS Hilbert space $L^2\left(A_0,\widehat{\tau}_0\right)$.
There are two derivatives $\widehat{\delta_1}$, $\widehat{\delta_2}$ given by
\begin{equation*}
\widehat{\delta_1} \left(\sum_{r, s}a_{r,s}\widehat{u}^r\widehat{v}^s\right)= \sum_{rs}2\pi i r a_{rs}\widehat{u}^r\widehat{v}^s,
\end{equation*}
\begin{equation*}
\widehat{\delta_2} \left(\sum_{r, s}a_{r,s}\widehat{u}^r\widehat{v}^s\right)= \sum_{rs}2\pi i s a_{rs}\widehat{u}^r\widehat{v}^s.
\end{equation*}
There is the natural covering projection $\mathbb{R}^2 \to \mathbb{T}^2$. On can select unbounded functions $x, y$ on $\mathbb{R}^2$ such that $\widehat{u}$ (resp. $\widehat{v}$) corresponds to $e^{2\pi ix}$ (resp $e^{2\pi iy}$) and it is clear that 
\begin{equation*}
\widehat{\delta_1}=  \frac{\partial}{\partial x}~,
\end{equation*}
\begin{equation*}
\widehat{\delta_2} =  \frac{\partial}{\partial y}~,
\end{equation*}
i.e. both $\widehat{\delta_1}$ and $\widehat{\delta_1}$ are first-order differential operators.
Similarly to \ref{nt_spectral_triple} one can  introduce double GNS Hilbert space $\widehat{H} = L^2(A_{0}, \widehat{\tau}_0) \oplus  L^2(A_{0}, \widehat{\tau}_0)$ and define derivations
\begin{equation*}
\widehat{\partial} = \widehat{\partial}_{\tau} = \widehat{\delta}_1 + \tau {\delta}_2; \ (\tau \in \mathbb{C}, \ \mathrm{Im}(\tau)\neq 0),
\end{equation*}
\begin{equation*}
\widehat{\partial}^\dagger = -\widehat{\delta}_1 - \overline{\tau}\widehat{\delta}_2.
\end{equation*}

\begin{equation*}
\widehat{D}= \begin{pmatrix}
0 & \underline{\widehat{\partial}}^\dagger \\
\underline{\widehat{\partial}} & 0
\end{pmatrix}.
\end{equation*}
There is the isomorphism $ L^2\left(A_{0}, \widehat{\tau}_0\right)\approx L^2\left(A_{\theta}, \tau_0\right)$ of Hilbert spaces given by
\begin{equation*}
\sum_{m \in \mathbb{Z}, n \in \mathbb{Z}} a_{mn}u^mv^n \mapsto \sum_{m \in \mathbb{Z}, n \in \mathbb{Z}} a_{mn}\widehat{u}^m\widehat{v}^n.
\end{equation*}
which naturally induces an isomorphism of direct sums  $\varphi: \widehat{H}= L^2\left(A_{0}, \widehat{\tau}_0\right)\bigoplus L^2\left(A_{0}, \widehat{\tau}_0\right)\approx  H = L^2\left(A_{\theta}, \tau_0\right) \bigoplus L^2\left(A_{\theta}, \tau_0\right)$.
A direct calculation shows that following conditions hold
\begin{enumerate}
	\item $\left(\xi, \eta \right)= \left(\varphi\left(\xi\right),\varphi\left(\eta\right)\right)$ $\forall \xi, \eta \in L^2\left(A_{\theta}, \tau_0\right)$,
\item $\mathrm{Dom}~D= \varphi\left(\mathrm{Dom}~\widehat{D}\right)$, 
\item $D \varphi\left(\psi\right)= \varphi\left(\widehat{D}\psi\right)$.
\end{enumerate}
If $A_\theta \to A_{\theta'}$ is given by
\begin{equation*}
u\mapsto u'^m; ~ v \mapsto v'^n
\end{equation*}
then there is a Hilbert space $\widetilde{H}=A_{\theta'} \otimes_{A_\theta} H$
and the natural action of $G\left(A_{\theta'}|A_{\theta}\right)\approx \mathbb{Z}_m \times \mathbb{Z}_n$ on $\widetilde{H}$. Similarly there is a the homomorphism $A_0 \to A'_0$ given by
\begin{equation*}
\widehat{u}\mapsto \widehat{u}'^m; ~ \widehat{v} \mapsto \widehat{v}'^n
\end{equation*}

 and the Hilbert space $\widetilde{\widehat{H}}=A'_{0} \otimes_{A_0}\widehat{H}$. There is an isomorphism $\widetilde{\varphi}:\widetilde{H}\approx\widetilde{\widehat{H}}$ such that $ \widetilde{\varphi}\left(g \xi\right) = g \widetilde{\varphi}\left( \xi\right)$ for any $\xi \in \widetilde{\widehat{H}}$ and $g \in \mathbb{Z}_m \times \mathbb{Z}_n$. 
\end{empt}
\begin{rem}\label{nt_local_rem}
The construction \ref{nt_local_constr} means that some properties of Hilbert space  and its Dirac operator of the noncommutative torus are the same as a Hilbert space and Dirac operator of  the commutative torus.
\end{rem}
\begin{lem}
	In the situation described in \ref{nt_st_descr} the coherent sequence $\left\{\left(\A_{\theta_n}, H_n, D_n\right)\right\}_{n\in \mathbb{N}^0}$ is local.
\end{lem}
\begin{proof}
	From \ref{nt_local_constr} and the Remark \ref{nt_local_rem} it follows that general case can be reduced to commutative one, i.e. $\theta = 0$  and $A_0 = C\left(\mathbb{T}^2\right)$. However $\mathbb{T}^2$ is a Riemannian manifold and Dirac operator is local (differential), so an application of the Lemma \ref{comm_local_lemma} completes the proof of this lemma. 
\end{proof}
\begin{lem}\label{nt_algebra_dense_lem}
	The $\widetilde{\A}_{\theta}$ is a dense subalgebra of $\widetilde{A}_{\theta}$.
\end{lem}
\begin{proof}
	From the construction \ref{dense_nt_constr} it follows that for any $\widetilde{a} \in \widetilde{A}_{\theta}$ and $\varepsilon > 0$ there is a finite subset $G' \subset \mathbb{Z} \times \mathbb{Z}$ and an operator $\widetilde{a}' \in \widetilde{A}_{\theta}$ given by
	\begin{equation*}
	\widetilde{a}' = \sum_{g,g' \in G'} ~\sum_{i,j,i',j' \in \{1,2\}}g \xi_{ij}  \rangle a^k_{ijgi'j'g'}\langle g' \eta_{i'j'}
	\end{equation*} 
	such that
	\begin{equation*}
	\left\|\widetilde{a} - \widetilde{a}'\right\| < \frac{\varepsilon}{2}.
	\end{equation*}.
	The element $a^k_{ijgi'j'g'}$ is given by
\begin{equation*}
a^k_{ijgi'j'g'}=\left\langle h_k\left(g\right)\left(e^{n_k}_j\left(v_{n_k}\right)e^{m_k}_i\left(u_{m_k}\right)\right),a_k\left(h_k\left(g'\right)\left(e^{m_k}_{i'}\left(u_{m_k}\right)e^{n_k}_{j'}\left(v_{n_k}\right)\right)\right) \right\rangle_{A_{\theta_k}}
\end{equation*}
where $a_k \in A_{\theta_k}$. Since $\A_{\theta_k}$ is dense subalgebra of $A_{\theta_k}$ whence there is $\overline{a}_k \in \A_{\theta_k}$ such that
\begin{equation*}
	\left\|a_k- \overline{a}_k\right\| < \frac{\varepsilon}{8\left|G'\right|}.
\end{equation*}
Since $e^{m_k}_i, e^{n_k}_j \in \Coo\left(S^1\right)$ we have $e^{m_k}_i\left(u\right) e^{n_k}_j\left(v\right), e^{n_k}_j\left(v\right) e^{m_k}_i\left(u\right)\in \A_{\theta_k}$ and 
\begin{equation*}
\overline{a}^k_{ijgi'j'g'}=\left\langle h_k\left(g\right)\left(e^{n_k}_j\left(v_{n_k}\right)e^{m_k}_i\left(u_{m_k}\right)\right),\overline{a}_k\left(h_k\left(g'\right)\left(e^{m_k}_{i'}\left(u_{m_k}\right)e^{n_k}_{j'}\left(v_{n_k}\right)\right)\right) \right\rangle_{A_{\theta_k}} \in \A_\theta,
\end{equation*}
If 
\begin{equation*}
	\widetilde{a}'' = \sum_{g,g' \in G'} ~\sum_{i,j,i',j' \in \{1,2\}}g \xi_{ij}  \rangle \overline{a}^k_{ijgi'j'g'}\langle g' \eta_{i'j'}
\end{equation*}
then
\begin{equation}\label{nt_dense_eq}
\left\|\widetilde{a} - \widetilde{a}''\right\| < \varepsilon
\end{equation}
and
\begin{equation*}
 \widetilde{a}''\in \mathcal{K}^\infty\left(\overline{X}^\infty_{\A_\theta}\right).
\end{equation*}
Otherwise if  $ \left\{b_r \in A_{\theta_r}\right\}_{r \in \mathbb{N}}$
	is the sequence such that $b_r = 0$ for $r \le k$ and
	\begin{equation*}
	b_r = 
\sum_{g,g' \in G^k} ~\sum_{i,j,i',j' \in \{1,2\}}h_r\left(g\right)\left(e^{n_r}_j\left(v_{n_r}\right)e^{m_r}_i\left(u_{m_r}\right)\right)\overline{a}^k_{ijgi'j'g'}\left(h_r\left(g'\right)\left(e^{m_r}_{i'}\left(u_{m_r}\right)e^{n_r}_{j'}\left(v_{n_r}\right)\right)\right)
		\end{equation*}
		for $r > k$ then $\lim_{r \to \infty} b_r = \widetilde{a}''$ in the weak topology i.e. $\widetilde{a}'' \in \left(\bigcup A_{\theta_k}\right)''$, whence $\widetilde{a}'' \in \widetilde{\A}_\theta$. From \eqref{nt_dense_eq} it follows that $\widetilde{\A}_\theta$ is dense in $\widetilde{A}_\theta$.

\end{proof}
From the Lemma \ref{nt_algebra_dense_lem} it follows that the described in \ref{nt_st_descr} sequence of spectral triples is regular.

\section{Epilogue}
\paragraph*{} There is a good noncommutative generalization of  covering projections \cite{hajac:toknotes} based on monads, comonads and adjoint functors. But this generalization cannot describe coverings of locally compact spaces. In the Definition \ref{fin_def} finite sums and projective modules are \textit{manually} replaced with infinite sums and compact operators. In result we have no a beautiful theory, but we have new constructions. This replacement can be regarded as illegal. However Max Planck \textit{manually} replaced integrals with sums \cite{planck:law} and the quantum mechanics had been obtained.


\begin{thebibliography}{10}


\bibitem{arveson:c_alg_invt} W. Arveson. {\it An Invitation to $C^*$-Algebras}, Springer-Verlag. ISBN 0-387-90176-0, 1981.



\bibitem{bass} H. Bass. {\it Algebraic K-theory.} W.A. Benjamin, Inc. 1968.


\bibitem{blackadar:ko} B. Blackadar. {\it K-theory for Operator Algebras}, Second edition. Cambridge University Press 1998.



\bibitem{bogachev_measure_v2}V. I. Bogachev. {\it Measure Theory} (volume 2). Springer-Verlag, Berlin, 2007.


\bibitem{bourbaki_sp:gt} N. Bourbaki, {\it General Topology}, Chapters 1-4, Springer, Sep 18, 1998.


\bibitem{brickell_clark:diff_m} F. Brickell and R. S. Clark.
{\it Differentiable manifolds; An introduction.} London ; New York: V. N. Reinhold Co., 1970.

\bibitem{chun-yen:separability} Chun-Yen Chou. {\it Notes on the Separability of $C^*$-Algebras.}
TAIWANESE JOURNAL OF MATHEMATICS Vol. 16, No. 2, pp. 555-559, April 2012 This paper is available online at http://journal.taiwanmathsoc.org.tw, 2012.

\bibitem{frank:frames} Michael Frank , David R. Larson, {\it Frames in Hilbert $C^*$-modules and $C^*$-algebras}, arXiv:math/0010189, 2000.





\bibitem{do_carmo:rg} Manfredo P. do Carmo. {\it Riemannian Geometry.} Birkh\"auser, 1992.


\bibitem{chavel:riemann} Isaac Chavel. {\it Riemannian Geometry: A Modern Introduction} (Cambridge Studies in Advanced Mathematics) Paperback – July 6, 2006.

\bibitem{connes:grav} A. Connes. {\it Gravity coupled with matter and foundation of noncommutative geometry}. Commun. Math. Phys. 182 (1996), 155–176. 1996.







\bibitem{davis_kirk_at}
James F. Davis. Paul Kirk
{\it Lecture Notes in Algebraic Topology}. Department of Mathematics, Indiana University, Blooming- ton, IN 47405, 2001.





\bibitem{nicolas_ginoux:dirac_spectrum}Nicolas Ginoux. {\it The Dirac Spectrum.} Springer, Jun 11, 2009.

\bibitem{varilly_bondia} Jos\'e M. Gracia-Bondia, Joseph C. Varilly, Hector Figueroa. {\it Elements of Noncommutative Geometry}, Springer, 2001.   92  96    142





\bibitem{halmos:set} Paul R.  Halmos {\it Naive Set Theory.} D. Van Nostrand Company, Inc., Prineston, N.J., 1960.

\bibitem{hartshorne:ag} Robin Hartshorne. {\it Algebraic Geometry.} Graduate Texts in Mathematics, Volume 52, 1977.




 
 \bibitem{ivankov:inv_lim}  Petr R. Ivankov. {\it Inverse Limits of Noncommutative Covering Projections},  arXiv:1412.3431, 2014.
 

 


\bibitem{hajac:toknotes}
{\it Lecture notes on noncommutative geometry and quantum groups}, Edited by Piotr M. Hajac.

\bibitem{hermann:nonstandard} Robert A. Hermann {\it Nonstandard Analysis - A Simplified Approach.}  	arXiv:math/0310351, 2010.



\bibitem{kakariadis:corr}
Evgenios T.A. Kakariadis, Elias G. Katsoulis, {\it Operator algebras and $C^*$-correspondences: A survey.} 	arXiv:1210.6067, 2012.



\bibitem{karoubi:k} M. Karoubi. {\it K-theory, An Introduction.} Springer-Verlag 1978.

\bibitem{kastler:connes_lott} Daniel Kastler, Thomas Schucker, {\it The Standard Model a la Connes-Lott}, arXiv:hep-th/9412185, 1994.


\bibitem{koba_nomi:fgd} S. Kobayashi, K. Nomizu. {\it Foundations of Differential Geometry}. Volume 1. Interscience publishers a division of John Willey \& Sons, New York - London. 1963.

\bibitem{lee:smooth} John M. Lee. {\it Introduction to Smooth Manifolds}. University of Washington. Department of Mathematics. Version 3.0, December 31, 2000.

\bibitem{milne:etale}
J.S. Milne. {\it \'Etale cohomology.} Princeton Univ. Press  1980.

\bibitem{bram:atricle} Bram Mesland. {\it Unbounded bivariant $K$-theory and correspondences in noncommutative geometry} arXiv:0904.4383, 2009.





\bibitem{munkres:topology} James R. Munkres. {\it Topology.} Prentice Hall, Incorporated, 2000.

\bibitem{murphy} G.J. Murphy. {\it $C^*$-Algebras and Operator Theory.} Academic Press 1990.


\bibitem{Pa1} {W.~L.~Paschke}, \textit{Inner product modules over B*-algebras},
  Trans.~Amer.~Math.~Soc. {\bf 182}(1973), 443-468.

\bibitem{pavlov_troisky:cov} Alexander Pavlov, Evgenij Troitsky. {\it Quantization of branched coverings.}   arXiv:1002.3491, 2010.


\bibitem{pedersen:ca_aut}Pedersen G.K. {\it $C^*$-algebras and their automorphism groups}. London ; New York : Academic Press, 1979.

\bibitem{planck:law}Planck, M.  {\it Zur Theorie des Gesetzes der Energieverteilung im Normalspektrum}. Verhandlungen der Deutschen Physikalischen Gesellschaft 2: 237. Translated in ter Haar, D. (1967). {\it The Old Quantum Theory, On the Theory of the Energy Distribution Law of the Normal
	Spectrum}. Pergamon Press. p. 82. LCCN 66029628. , 1900.

\bibitem{reed_simon:mp_1}Michael Reed, Barry Simon. {\it Methods of modern mathematical physics 1: Functional Analysis}. Academic Press, 1972.






























\bibitem{takesaki:oa_ii} Takesaki, Masamichi. {\it Theory of Operator Algebras II } Encyclopaedia of Mathematical Sciences, 2003. 

\bibitem{spanier:at}
E.H. Spanier. {\it Algebraic Topology.} McGraw-Hill. New York 1966.





























\bibitem{varilly:noncom} J.C. V\'arilly. {\it An Introduction to Noncommutative Geometry}. EMS 2006.

\bibitem{Wegge-Olsen} {N.~E.~Wegge-Olsen}, {\it K-theory and C*-algebras  -- a Friendly Approach}, Oxford University Press, Oxford, England,   1993.
   




















\end{thebibliography}
\end{document}